\documentclass[leqno,11pt]{article} 
\usepackage{xcolor}             
\usepackage{amsmath, amssymb}      
\usepackage{amsthm} 
\usepackage{amscd}       
\usepackage{cancel}      
\usepackage{footmisc}  
\usepackage{amsmath, amssymb}             
\usepackage{amsthm} 
\usepackage{ascmac}       
\usepackage{amscd}           
\usepackage{color}  
\usepackage[normalem]{ulem}

\newcommand{\ev}{{\rm Eval}}   
   
\newcommand{\deri}{{\rm Deri}}   
\newcommand{\id}{{\rm Id}} 
\newcommand{\card}{{\rm Card}} 
  
\newcommand{\edom}{{\rm End}} 

\newcommand{\erreurv}{c_{n,v}}
\newcommand{\den}{{\rm den}}
\newcommand{\enn}{{\Bbb N}}
\newcommand{\cu}{{\C}}

\DeclareFontEncoding{OT2}{}{} 
  \newcommand{\textcyr}[1]{%
    {\fontencoding{OT2}\fontfamily{wncyr}\fontseries{m}\fontshape{n}%
     \selectfont #1}}
\newcommand{\euler}{{\mbox{\textcyr{Z}}}}

\theoremstyle{plain} 
\newtheorem{theorem}{\indent\sc Theorem}[section]
\newtheorem{lemma}[theorem]{\indent\sc Lemma}
\newtheorem{corollary}[theorem]{\indent\sc Corollary}
\newtheorem{proposition}[theorem]{\indent\sc Proposition}

\theoremstyle{definition} 
\newtheorem{definition}[theorem]{\indent\sc Definition}
\newtheorem{facts}[theorem]{\indent\sc Facts}  
\newtheorem{choice}[theorem]{\indent\sc Choice}
\newtheorem{remark}[theorem]{\indent\sc Remark}
\newtheorem{example}[theorem]{\indent\sc Example}
\newtheorem{notation}[theorem]{\indent\sc Notation}

\newcommand{\C}{\mathbb{C}} 
\newcommand{\R}{\mathbb{R}}
\newcommand{\ru}{{\R}}
\newcommand{\Q}{\mathbb{Q}} 

\newcommand{\Z}{\mathbb{Z}} 
\newcommand{\qu}{{\Q}}

\newcommand{\N}{\mathbb{N}}

\newcommand{\const}{{dm\log(2)\rule{0mm}{4mm}+d\left(\log(dm+1)+dm\log\left(\rule{0mm}{3.5mm}\frac{dm+1}{dm}\right)\right)}}

\def\2{I\hspace{-.1em}I}

\title{Linear independence of values of hypergeometric functions \\and arithmetic Gevrey series}
\author{\textsc{Sinnou David}, \textsc{Noriko Hirata-Kohno} and \textsc{Makoto Kawashima}} 
\date{\empty}

\begin{document}  

\maketitle
 
%

\begin{abstract} 
We prove new linear independence results for the values of generalized hypergeometric functions ${}_pF_q$ 
at several distinct algebraic points, over suitable algebraic number fields. 
Our approach provides a uniform construction of Pad\'{e} approximants of type II, together with 
a novel non-vanishing argument for generalized Wronskians of Hermite type. 
This method applies uniformly across all parameter regimes. Even in the case $p = q+1$, we extend known results from single-point to multi-points settings over general number fields, in both complex and $p$-adic settings. When $p < q+1$, we establish linear independence results over arbitrary number fields; and for $p > q+1$, we confirm that the values do not satisfy global linear relations in the $p$-adic setting in a framework of arithmetic Gevrey series. 
The results generalize and strengthen earlier works, demonstrating the flexibility of our Pad\'{e} construction for families of contiguous hypergeometric functions, 
through a new non-vanishing proof for the determinant, that is crucial for the universality.
\end{abstract}


\maketitle

\section{Introduction}
The generalized hypergeometric function is a mathematical object that ubiquitously appears in various fields such as number theory, 
differential equations, and mathematical physics. 
Among the families of series, it is the Gevrey series that emerges as formal solutions to differential equations, singular perturbations 
and difference equations, initially emphasized by M.~Gevrey \cite{Gevrey} in 1918. 
In \cite{An2}, Y.~Andr\'{e} defined the notion of \textit{arithmetic Gevrey series} to extend the theory of $E$-functions and $G$-functions introduced 
by C.~L.~Siegel \cite{Siegel1} in his study of transcendental number theory. 

\medskip

In this article, we investigate new arithmetic properties of values of the generalized hypergeometric functions, a most significant class of arithmetic Gevrey series, relying on our uniform explicit Pad\'e approximations of type II.
Our construction of Pad\'{e} approximations for the generalized hypergeometric function ${}_pF_q$ does not depend on the choice of $p$ and $q$; indeed covering all the cases uniformly $p = q+1$,  $p < q+1$ and  $p > q+1$.
Despite the arithmetic and analytic behavior fundamentally differs according to the cases, 
our approach allows a universal and systematic construction of Pad\'e approximants without distinction.
We present a new proof of the non-vanishing property for the generalized Wronskian of Hermite type, to achieve the linear independence in all the cases, that is the crucial point to ensure the universality of our method.

\medskip

When $p = q + 1$, let us consider the solutions of hypergeometric differential equations assumed to be $G$-functions (\textit{e.\,g.} the polylogarithm function ${\rm Li}_s(z)=\sum_{k} \frac{z^k}{k^s}$).  
It is widely expected that values of $G$-functions  at algebraic points are linearly independent  over $\qu$ modulo the obvious obstructions\footnote{Remark that ${\rm Li}_1(3/4)-2{\rm Li}_1(1/2)=0$.}. 
In both cases, archimedean and non-archimedean, arithmetic properties of the generalized hypergeometric functions have been studied  by many authors, including Andr\'{e}~\cite{andre}, E.~Bombieri~\cite{Bom}, D.~V.~Chudnovsky and G.~V.~Chudnovsky~\cite{Chubrothers}, P.~D\'{e}bes~\cite{Debes}, A.~I.~Galochkin~\cite{G2}, Yu.~Nesterenko~\cite{Nest1}~\cite{Nest} and W.~Zudilin~\cite{Z}.  

Moreover, for specific hypergeometric $G$-functions such as the 
polylogarithms \cite{Lewin}, refined results have been obtained through explicit constructions of Pad\'{e} approximations, notably by G.~V.~Chudnovsky~\cite{ch2}, M.~Hata~\cite{Ha,Ha1993}, R.~Marcovecchio~\cite{marc,marc2}, V.~Meril\"{a}~\cite{Me1,Me2}, E.~M.~Nikishin~\cite{N}, T.~Rivoal \cite{Ri}, G.~Rhin and P.~Toffin~\cite{R-T}, Rhin and C.~Viola~\cite{RV0}, V.~N.~Sorokin~\cite{Sorokin}, Viola and Zudilin~\cite{VZ1} and the authors~\cite{DHK2,DHK3,DHK4}.   

\medskip

When $p < q + 1$, in particular $p = q$, the solutions belong to the class of $E$-functions, the most significant examples including the classical exponential series $e^z=\sum_{k}z^k/k!$. The Siegel--Shidlovsky theorem (\textit{cf.}~\cite{Siegel,Shi}) for $E$-functions determines the algebraic relations among the functions and their values. The whole theory has been rewritten by the seminal works of Andr\'{e}~\cite{An2-II,An4} and F.~Beukers~\cite{Beu}. Based on the works, S.~Fischler and Rivoal \cite[Theorem~$2$]{F-R} and \'{E}.~Delaygue \cite[Theorem~1]{De} recently obtained linear independent results over $\overline{\qu}$. 

General independence criteria for the values of hypergeometric functions in the case  $p < q + 1$, have been also obtained  by notably Galochkin~\cite{G3}, V.~Kh.~Salikhov~\cite{Sal1,Sal2}, and more recently by Gorelov~\cite{Gor1,Gor2}.  

\medskip

In the case  $p\le q+1$, V.~Meril\"{a }  (\cite{Me1}, \cite{Me2}) sketched an approach involving  several points via Pad\'{e} type II approximations. 
We also refer to the algebraic independence of the two special values of Gauss' hypergeometric function
${}_{2}F_{1} \biggl(\begin{matrix} \tfrac{1}{2}, \tfrac{1}{2}\\ 1 \end{matrix} \biggm| \alpha\biggr)\enspace$ and
${}_{2}F_{1} \biggl(\begin{matrix} \tfrac{1}{2}, -\tfrac{1}{2}\\ 1 \end{matrix} \biggm|  \alpha\biggr)\enspace$ when $\alpha$ 
is a non-zero algebraic number of sufficiently small modulus, sketched in \cite[Theorem 3.4]{Chubrothers2}.
This result was later proven by Y.~Andr\'e in \cite{andre}, together with the $p$-adic analogue.

In the case $p > q + 1$, the functions have been referred to as ${\euler}\emph{-series}$ since D.~Bertrand, in reference to Euler and his series ${\euler }(z)=\sum_kk!z^k$; in particular, when $ p = q+2$, they are called \emph{Euler-type series}. In Archimedean fields, these series have radius of convergence zero, thus their values cannot be studied. However, the radius of convergence is positive in non-Archimedean fields, thus it makes us study their values. In this case, it has been shown that there are no global relations among the values (see~\cite{Bom}). The general theory in the case \( p = q+2 \) has been studied by D.~Bertrand, V.~G.~Chirskii and
J.~Yebbou in \cite{B-C-Y}. 
Chirskii investigated in particular  the case $p > q + 1$~\cite{Ch1,Ch2,Ch3}. 
Improved results over those in~\cite{B-C-Y} have been obtained for explicit general hypergeometric Euler-type series via the construction of Pad\'{e} approximations by T.~Matala-aho and W.~Zudilin~\cite{Ma-Zu}, L.~Sepp\"{a}l\"{a}~\cite{Sep} and K.~V\"{a}\"{a}n\"{a}nen~\cite{Va1}.

\medskip

Our statement extends previous ones due to Chudnovsky in \cite[Theorem 3.1]{ch2} \cite[Theorem I]{ch11}, \cite[Theorem 0.3]{ch12} \cite[Theorem I]{Chubrothers} and Nesterenko \cite[Theorem 1]{Nest1} \cite[Theorem 1]{Nest}, which all dealt with values at only one point over the rational number field.

Our approach is inspired by previous works including a formal explicit construction of Pad\'e approximants.  However, standard derivation or primitivation  (as in \cite{DHK2,DHK3,DHK4} ({\it refer} \cite{Kaw,KP,KP2})) can no longer be used as the case of polylogarithms.
We introduce an appropriate operator instead, which gives the property for the given set of hypergeometric functions.

Although we refer the classical philosophy  showing suitable Pad\'e approximation and giving linear independence criteria  (either over the field of rational numbers or quadratic imaginary fields),  
our setting simultaneously manages for several functions and several special values.

Actually we prove a zero estimate in a highly non-trivial way to obtain 
the linear independence of the set of approximation constructed, referring \cite{B-C-Y} (Euler function case),
thus achieves {optimal},
non-vanishing conditions of our Wronskians.  
We mention that the work of Beukers \cite{beu2} involves several results on the algebraicity of values of the function.
A historical survey and  further reference is given in \cite{DHK2, DHK3}, with comparison concerning with earlier works. 

Summing up, our criterion provides the linear independence of values of full generalized hypergeometric functions including the contiguous ones, whose functional linear independence has been discussed in  \cite{Nest1,Nest}.
Thus our contribution, if any, is an uncharted non-vanishing property for the generalized Wronskian of Hermite type, along with a formal construction that allows a systematic treatment of Euler type, $E$-functions and $G$-functions simultaneously. It should also be noted that our construction allows us to replace technical analytic estimates by simple norm operator evaluation for often better or at least competitive quantitative estimates\footnote{Quantitative estimates are discussed in Section \ref{maintheoremssection}.}.

\section{Notations and main results} \label{main results}
Let us first define Gevrey series. 
\begin{definition}
Let $j$ be a rational number. 
A formal power series $$f(z)=\sum_{k=0}^{\infty}a_kz^k \in \C[[z]]$$ is a {\it Gevrey of order} $j$ if and only if the associated series $$f^{[j]}(z):=\sum_{k=0}^{\infty}\dfrac{a_k}{k!^j}z^k$$ has a positive radius of convergence.
It is called {\it Gevrey of precise order} $j$ if  and only if $f^{[j]}$ has positive and finite radius of convergence.
\end{definition}
Now we introduce arithmetic Gevrey series, defined by Andr\'{e} ({\it confer} \cite[$2.1.1.$ Definition]{An3}). 
\begin{definition}
Let $j$ be a rational number, and $\overline{\qu}\hookrightarrow \cu$ an embedding.
A formal power series $$f(z)=\sum_{k=0}^{\infty}a_kz^k \in \overline{\Q}[[z]]$$ is an {\it arithmetic Gevrey series of order $j$} if and only if
\begin{itemize} \setlength\itemsep{.01em}
\item[(i)] \text{There exists} $ C_1>0$  \text{such that for all} $n\ge0 \ \text{and}$ $\sigma\in {\rm{Gal}}(\overline{\Q}/\Q),$ $\left|\sigma\left(\dfrac{a_n}{n!^j}\right)\right|\le C^{n+1}_1\enspace,$
where $\vert\cdot\vert$ denotes the usual complex absolute value induced by the chosen embedding.
\item[(ii)] \text{There exists} $ C_2>0 $  \text{such that for all}  $n\ge0, \ {\rm{den}}\left(\dfrac{a_0}{0!^j},\ldots,\dfrac{a_n}{n!^j}\right)\le C^{n+1}_2\enspace.$
\end{itemize}
Here ${\rm{Gal}}(\overline{\Q}/\Q)$ is the Galois group of $\overline{\Q}/\Q$. Siegel's $G$-functions ({respectively}  $E$-functions) \cite{Siegel1} are nothing but holonomic\footnote{A power series in $K[[z]]$ over a field $K$ is said to be {\it holonomic} if it satisfies a linear differential equation over $K[z]$.} arithmetic Gevrey series of order $0$ ({respectively} order $-1$) and holonomic arithmetic Gevrey series of order $1$ are called Euler-type series.
\end{definition}

Let $p, q$ be non-negative integers and $a_1,\ldots,a_{p}, b_1,\ldots,b_{q}\in \Q\setminus\{0\}$, none of them are negative integers.
We define the generalized hypergeometric function with parameters $a_i,b_j$ by 
\begin{eqnarray*}
{}_{p}F_{q} \biggl(\begin{matrix} a_1,\ldots, a_p\\ b_1, \ldots, b_{q} \end{matrix} \biggm| z\biggr)
=\displaystyle\sum_{k=0}^{\infty}\dfrac{(a_1)_k \cdots (a_{p})_k}{(b_1)_k\cdots(b_{q})_k}\dfrac{z^k}{k!}\enspace,
\end{eqnarray*}
where $(a)_k$ is the Pochhammer symbol: $(a)_0=1$, $(a)_k=a(a+1)\cdots (a+k-1)$.
Whenever $a_1,\ldots,a_{p}, b_1,\ldots,b_{q}\in\Q\setminus\{0\}$, the function  $${}_{p}F_{q} \biggl(\begin{matrix} a_1,\ldots, a_p\\ b_1, \ldots, b_{q} \end{matrix} \biggm| z\biggr)$$ is indeed a holonomic arithmetic Gevrey series of order $p-q-1$. 

We collect some notations which we use throughout the article. 
Let $K$ be an algebraic number field of arbitrary degree $[K:\Q] < \infty$.
Let us denote by $\N$ the set of strictly positive integers\footnote{Note that this convention is not the one commonly used in Europe where $\enn$ would include $0$}.
The set of places of $K$ is denoted by ${{\mathfrak{M}}}_K$ (with ${\mathfrak{M}}^{\infty}_K$ and ${{\mathfrak{M}}}^{f}_K$
representing the set of infinite places and finite places, respectively).
For $v\in {{\mathfrak{M}}}_K$, we denote the completion of $K$ with respect to $v$ by $K_v$.
Let us denote the normalized absolute value $| \cdot |_v$ for $v\in {{\mathfrak{M}}}_K$:
\begin{align*}
&|p|_v=p^{-\tfrac{[K_v:\Q_p]}{[K:\Q]}} \ \text{if} \ v\in{{\mathfrak{M}}}^{f}_K \ \text{and} \ v\mid p\enspace, \ \ \ \ \ \ |x|_v=|\sigma_v(x)|^{\tfrac{[K_v:\R]}{[K:\Q]}} \ \text{if} \ v\in {{\mathfrak{M}}}^{\infty}_K\enspace,
\end{align*}
where $p$ is a prime number and $\sigma_v$ the embedding $K \hookrightarrow \C$ corresponding to $v$. 
On $K_v^n$, the norm $|\cdot|_v$ denotes the norm of the supremum.

Let $m$ be a positive integer and $\boldsymbol{\beta}=(\beta_0,\ldots,\beta_m) \in K^{m+1}$. 
The absolute {affine} height of $\boldsymbol{\beta}$ is defined by
\begin{align*}
&{\mathrm{H}}(\boldsymbol{\beta})=\prod_{v\in {{\mathfrak{M}}}_K} \max\{1, |\beta_0|_v,\ldots,|\beta_m|_v\}\enspace,
\end{align*}
and the logarithmic absolute height by ${\rm{h}}(\boldsymbol{\beta})={\rm{log}}\, \mathrm{H}(\boldsymbol{\beta})$. 
We denote the local logarithmic absolute  ${\rm{log}}\max\{1,|\beta_i|_v\}$ by ${\mathrm{h}}_v(\boldsymbol{\beta})$ for each $v\in \mathfrak{M}_K$. 
Then, ${\mathrm{h}}(\boldsymbol{\beta})={{\sum_{v\in \mathfrak{M}_K}}}{\mathrm{h}}_v(\boldsymbol{\beta})$.
 
Define the denominator of $S$ by 
{\small{$${\rm{den}}(S)=\min\{n\in \Z \mid n>0 \ \text{such that} \ n\alpha \ \text{is an algebraic integer for every} \ \alpha\in S\}$$}}
for a finite set $S$ of algebraic numbers. 
Let $y$ be a real number. Write the least ({respectively} the greatest) integer greater ({respectively} less) than or equal to $y$ by $\lceil y \rceil$ (resp. $\lfloor y \rfloor$).
Denote by
$$\mu_n(x)={\rm{den}}(x)^n{\displaystyle{\prod_{\substack{q:\text{prime} \\ q|{\rm{den}}(x)}}}} q^{\lfloor {\tfrac{n}{q-1}}\rfloor }, \ \ \ \ \ \mu(x)={\rm{den}}(x){\displaystyle{\prod_{\substack{q:\text{prime} \\ q|{\rm{den}}(x)}}}} q^{\tfrac{1}{q-1}}\enspace$$
for $n\in\N$ and $x\in\Q$.
{{We also denote by\footnote{Note that $\mu_v(x)$ mimics  $\vert\mu(x)\vert_v$ for $v\in \mathfrak{M}^{f}_K$, {\em however} since $\mu(x)$ needs not be in the field $K$ it has to be defined accordingly.}
$$\mu_v(x)=\left\{\begin{array}{ll} 1 & \mbox{if $v\in \mathfrak{M}^{\infty}_K$ or $v\in \mathfrak{M}^{f}_K$ \& $\vert x \vert_v\leq 1$}\enspace,\\
\vert {\rm{den}}(x) \vert_v\vert p\vert_v^{\tfrac{1}{p-1}} & \mbox{if $v\in \mathfrak{M}^{f}_K$ \& $\vert x\vert_v>1$ where $p$ is the prime below $v$}\enspace.\end{array}\right.$$}}

Now we are ready to state our main theorems. 
Let $p,q$ be positive integers and $a_1, \ldots, a_p, b_1, \ldots, b_q$ be non-zero rational numbers such that none of them is a negative integer. 
We now fix an algebraic number field $K$ and a place $v\in \mathfrak{M}_K$.
We denote the radius of convergence of ${}_{p}F_{q} \biggl(\begin{matrix} a_1,\ldots, a_p\\ b_1, \ldots, b_{q} \end{matrix} \biggm| z\biggr)$ in ${{K_v}}$ by $r_v$. 
The following provides a table of $r_v$. 
$$
\begin{array}{|c|c|c|c|c|c|c|c|c|}
\hline
{} & v\in \mathfrak{M}^{\infty}_{\overline{\Q}} & v\in \mathfrak{M}^{f}_{\overline{\Q}}\\
\hline 
p<q+1 & r_v=\infty & r_v<\infty \\
\hline
p=q+1 & r_v<\infty & r_v<\infty\\
\hline
p>q+1 & r_v=0 & r_v<\infty \\
\hline
\end{array}
$$
Let us fix $\boldsymbol{\alpha}=(\alpha_1,\ldots,\alpha_m)\in (K\setminus\{0\})^m$.
Additionally, we now assume 
{\small{
\begin{align} \label{det non zero}
\text{neither} \ a_k \ \text{nor} \ a_k + 1 - b_j \ (1 \leq k \leq p \ \text{and} \ 1 \leq j \leq q) \ \text{is a strictly positive integer.}
\end{align}
}}

Under the assumption~\eqref{det non zero}, our main results describe the arithmetic properties of the values of the generalized hypergeometric functions
\[
{}_{p}F_{q} \biggl(\begin{matrix} a_1,\ldots, a_p\\ b_1, \ldots, b_{q} \end{matrix} \biggm| \dfrac{\alpha_i}{z}\biggr)
\]
within their respective radii of convergence. The following table indicates which theorem corresponds to each arithmetic hypergeometric Gevrey series and to which type of place:
{\scriptsize{\[
\hspace{-5mm}\begin{array}{|c|c|c|}
\hline
 & v \in \mathfrak{M}^{\infty}_{\overline{\Q}} & v \in \mathfrak{M}^{f}_{\overline{\Q}} \\
\hline
p < q+1 & \text{Theorem~2.4~(Theorem~6.5: general)} & \text{---} \\ 
\hline
p = q+1 & \text{Theorem~2.3~(Theorem~6.3: general,~quantitative)} & \text{Theorem~2.3~(Theorem~6.3)} \\ 
\hline
p > q+1 & \text{---} & \text{Theorem~2.5~(Theorem~6.10: general)} \\ 
\hline
\end{array}
\]
}}
First let us consider the case where $p=q+1$ (denoted by $d$).  
Let $\varepsilon_v=1$ if $v|\infty$ and $0$ otherwise.
For $\beta\in K\setminus\{0\}$ and $v\in \mathfrak{M}_K$, define a real number:
\begin{align*}
V_v(\boldsymbol{\alpha},\beta) &=\displaystyle  \log|\beta|_{v}+dm({\rm{h}}_{v}(\boldsymbol{\alpha})-{\rm{h}}(\boldsymbol{\alpha},\beta))-(dm+1){\rm{h}}_{v}(\boldsymbol{\alpha})\\ 
&-\left(\const\right)\\
&-dm\sum_{j=1}^d{\rm{den}}(a_j)-(dm+1)\sum_{j=1}^d\log\mu(b_j)+{{\sum_{j=1}^d \log\mu_v(a_j)}}  
\enspace.
\end{align*}

\begin{theorem} \label{hypergeometric} 
Assume that each coordinate of $\boldsymbol{\alpha}$ is pairwise distinct and Equation~\eqref{det non zero} holds. Suppose $V_{v}(\boldsymbol{\alpha},\beta)>0$. Then the $dm+1$ elements in $K_{v}$$:$
\begin{align*}
&{}_{d}F_{d-1} \biggl(\begin{matrix} a_1,\ldots, a_d\\ b_1, \ldots, b_{d-1} \end{matrix} \biggm| \dfrac{\alpha_i}{\beta}\biggr)\enspace, \ \
{}_{d}F_{d-1} \biggl(\begin{matrix} a_1+1,\ldots, a_{s}+1,a_{s+1},\ldots,a_d\\ b_1, \ldots, \ldots,b_{d-1} \end{matrix} \biggm| \dfrac{\alpha_i}{\beta}\biggr) 
\end{align*}
$(1\le i \le m, 1\le s \le d-1)$ and $1$ are linearly independent over $K$.
\end{theorem}
Next we consider the case $p<q+1$. Here, we fix an embedding $\overline{\Q}\hookrightarrow \C$.
\begin{theorem} \label{hypergeometric2} 
Assume $p<q+1$ and each coordinate of $\boldsymbol{\alpha}$ is pairwise distinct and Equation~\eqref{det non zero} holds.
Then, the $(q+1)m+1$ complex numbers:
\begin{align*}
&{}_{p}F_{q} \biggl(\begin{matrix} a_1,\ldots, a_p\\ b_1, \ldots, b_{q} \end{matrix} \biggm| \alpha_i\biggr)\enspace, \ \ 
{}_{p}F_{q} \biggl(\begin{matrix} a_1+1,\ldots,\ldots,\ldots,a_p+1\\ b_1+1, \ldots, b_{s}+1,b_{s+1},\ldots,b_{q} \end{matrix} \biggm| \alpha_i\biggr)\enspace
\end{align*}
$(1\le i \le m, 1\le s \le q)$ and $1$ are linearly independent over $\overline{\Q}$. 
\end{theorem}

Finally, we address the case $p>q+1$. 
{{Assume all $\alpha_i$ are algebraic integers with 
\begin{align*}
|\alpha_i|_{v}<\prod_{j=1}^{d}\mu_v(a_j)|p_v|^{\tfrac{d-d'}{p_v-1}}_v
\end{align*}
for any $v\in {{\mathfrak{M}}}^{f}_K$ above a rational prime $p_v$ that divides $\prod_{j=1}^p {\rm{den}}(a_j)$.
\begin{theorem} \label{hypergeometric3}
Assume that each coordinate of $\boldsymbol{\alpha}$ is pairwise distinct and Equation~\eqref{det non zero} holds.
Let $\boldsymbol{\lambda}=(\lambda_0,\lambda_{s,i})_{\substack{1\le s \le p \\ 1\le i \le m}}\in K^{dm+1}\setminus\{\boldsymbol{0}\}$.
Then there exists an effectively computable positive real number $H_0$ such that, whenever $H(\boldsymbol{\lambda}) \ge H_0$, for any $H \ge H(\boldsymbol{\lambda})$, there exists a prime
$$p'\in  \left] \left(\frac{3pm\log H}{(p-q-1)\log\log H}\right)^{\tfrac{1}{8dm}}, \, {\frac{12pm \displaystyle{\max_{1\leq j\leq q}}\{\den(b_j)\}\log H}{(p-q-1)\log\log H}}\right[$$
and a place $v\in \mathfrak{M}^{f}_K$ above $p'$ for which the linear forms in hypergeometric values in $K_{v}:$
{\footnotesize{
$$\lambda_0+\sum_{i=1}^m \lambda_{p,i}\cdot {}_{p}F_{q} \biggl(\begin{matrix} a_1,\ldots, a_p\\ b_1, \ldots, b_{q} \end{matrix} \biggm| \alpha_i\biggr)+\sum_{s=1}^{p-1}\sum_{i=1}^m\lambda_{s,i}\cdot 
{}_{p}F_{q} \biggl(\begin{matrix} a_1+1,\ldots, a_{s}+1,a_{s+1},\ldots,a_p\\ b_1, \ldots, \ldots,b_{q} \end{matrix} \biggm| \alpha_i\biggr)$$
}}
does not vanish.
\end{theorem}
{It can be remarked that the range of primes needed to ensure non-vanishing is a {\it short} interval in the sense that both left and right hand side are proportional to $\log(H)/\log\log(H)$. Whereas in  Matala-aho-Zudilin~\cite{Ma-Zu}, V\"{a}\"{a}n\"{a}nen~\cite{Va1} and L.~Sepp\"{a}l\"{a}~\cite{Sep} (special case of the Euler series $\sum_kk!z^k$),  and in Bertrand-Chirskii-Yebbou~\cite{B-C-Y} (special case $p=q+2$ but Gevrey not necessarily contiguous) or the later works of Chirskii (\cite{Ch1,Ch2,Ch3}), the required size of the prime interval is a {\it large} one: the left hand side is of the order of the logarithm of the right hand side (which is similar to ours, of the order of $\log(H)/\log\log(H)$). This is due to our optimal construction and a factorial is therefore not lost in the estimates.
 }
 
\medskip

This article is organized as follows.
In Section~$\ref{ghf}$, we describe our setup for generalized hypergeometric functions. 
In Section~$\ref{pa}$, we proceed with our construction of Pad\'{e} approximants, generalizing the method used in \cite{DHK2, DHK3, DHK4}.
Section~\ref{nonvan} is devoted to proving the non-vanishing of the crucial determinant, via the study of kernels of linear maps associated with contiguous hypergeometric functions.
In Section~$6$, we give the proof of Theorem~$\ref{hypergeometric}$, $\ref{hypergeometric2}$ and $\ref{hypergeometric3}$.
More general statements, together with totally effective linear independence measures in case of $p=q+1$, are also provided in this section, as given by Theorem~$\ref{thm 2}$, $\ref{general 2}$ and $\ref{general 3}$.

\section{Pad\'{e} approximation of generalized hypergeometric functions}\label{pasub}
Throughout this section, let $K$ be a field of characteristic $0$.
Denote the ring of $K$-linear endomorphisms (respectively automorphisms) of $K[t]$ by ${\rm{End}}_K(K[t])$ ({respectively} ${\rm{Aut}}_K(K[t])$). 
We canonically embed the Weyl algebra $K[t,\tfrac{d}{dt}]$ into ${\rm{End}}_K(K[t])$.
\subsection{Preliminaries}\label{ghf}

\subsubsection{Linear properties of differential operators} 
\begin{notation} \label{notationderiprim}
\begin{itemize}
\item[$({\rm{i}})$]  For $\alpha\in K$, denote by ${\ev}_{\alpha}$ the linear evaluation map $K[t]\longrightarrow K$, $P\longmapsto P(\alpha)$. Whenever there is an ambiguity in a setting of variables, we will denote the map by $\ev_{t\rightarrow \alpha}$.
\item[$({\rm{ii}})$] For $P\in K[t]$, we denote by $[P]$ the multiplication by $P$ (the map $Q\longmapsto PQ$). If there is no ambiguity, we will sometimes omit the brackets.
\item[$({\rm{iii}})$] For a $K$-automorphism $\varphi$ of a $K$-module $M$ and an integer  $k$, put
$$\varphi^{k}=\begin{cases}
\overbrace{\varphi\circ\cdots\circ\varphi}^{k\,\text{times}} & \ \text{if} \ k>0\\
{\rm{id}}_M &  \ \text{if} \ k=0\\
\overbrace{\varphi^{-1}\circ\cdots\circ\varphi^{-1}}^{-k\,\text{times}} & \ \text{if} \ k<0\enspace.
\end{cases}
$$
\end{itemize}
\end{notation}

The following are elementary remarks on the action of the differential on polynomials, formal series that we use several times. We regroup them for the convenience of the reader
\begin{facts}\label{elem}
\begin{itemize}
\item[$({\rm{i}})$] The linear operator  on $K[t]$ defined by  $A(t)\longmapsto t\frac{d}{dt}(A(t))$ has eigenvalue $k$ on the element $t^k$ of the canonical basis of $K[t]$. 
\item[$({\rm{ii}})$] Let $\alpha\in K$ and $A(t)\in K[t]$, then $A\left(t\frac{d}{dt}+\alpha\right)$ has eigenvalue $A(k+\alpha)$ on the element $t^k$ of the canonical basis of $K[t]$.   In particular, if we assume moreover that  $\alpha+k$ is not a root of $A$ for any $k\geq 0$,   then $A\left(t\frac{d}{dt}+\alpha\right)\in {\rm Aut}_K(K[t])$. Moreover, the operator $A\left(t \frac{d}{dt}+\alpha\right)\in \edom_K(K[t])$ preserves each ideal $(t^n)$, $n\geq 0$ viewed as $K$-vector spaces.
\item[$({\rm{iii}})$] Let $H(t)\in K[t]$. For any $k\geq 0$, we have $[t^k]\circ H\left(t \tfrac{d}{dt}\right)=H\left(t\tfrac{d}{dt}-k\right)\circ [t^k].$
\end{itemize}

\end{facts}

\begin{proof} $({\rm{i}})$ and $({\rm{ii}})$ do not require proof. For (iii), apply (ii) to $A=H$, $\alpha=0$,  for any non-negative integer $m$, the left hand side is
$$
[t^k]\circ H\left(t \frac{d}{dt}\right) (t^m)=H(m)t^{m+k}\enspace.
$$
Whereas, the right hand side, again by $({\rm{ii}})$, with $A=H$ and $\alpha=-k$.
$$
H\left(t \tfrac{d}{dt}-k\right)\circ [t^k](t^m)=H(k+m-k)t^{m+k}=H(m)t^{m+k} \enspace.
$$
\end{proof}

\subsubsection{Generalized contiguous hypergeometric functions}
In this subsection, we introduce the generalized hypergeometric function. 
First, let us introduce polynomials $A(X),B(X)\in K[X]$ satisfying $\max\{{\rm{deg}}\,A,{\rm{deg}}\,B\}>0$. 
Assume 
\begin{align} \label{AB}
A(k)B({{k}})\neq 0 \ \ \ (k\ge0)\enspace.
\end{align} 
Consider the differential equation 
\begin{equation}\tag{$\partial E_{A,B}$}\label{defhyper}
\Big(B\left(-z\tfrac{d}{dz}\right)z-A\left(-z\tfrac{d}{dz}\right)\Big)f(z)=B(0)\enspace.
\end{equation} 

\begin{facts}\label{defF} The equation~(\ref{defhyper}) has a unique solution with residue $1$ in $(1/z)\cdot K[[1/z]]$ given by
$$F_{A,B}\left(\frac{1}{z}\right)= F\left(\frac{1}{z}\right)=\sum_{k=0}^\infty \frac{c_k}{z^{k+1}}$$ where the sequence $\bold{c}=(c_{k})_{k \ge0}$ is inductively defined by:  
\begin{align}\label{recurrence 1} 
c_0=1, \kern20pt c_{k+1}=c_k\cdot\dfrac{A(k)}{B(k+1)} \ \ \ (k\ge0)\enspace.
\end{align}
\end{facts} 
\begin{proof}  Hypothesis~(\ref{AB}) ensures that the sequence $c_k$ is well defined and $c_k\neq 0$, $k\geq 0$. Moreover,  using Facts~\ref{elem}, (ii),  one readily checks that any solution in  $(1/z)\cdot K[[1/z]]$ necessarily satifies~(\ref{recurrence 1}) which uniquely defines $F(1/z)$.
\end{proof}

We now construct   series  called contiguous 
to $F$ in the sense that they are linked  to $F$ by a order $1$ differential operator. Let $\gamma \in K$, and introduce  for a given choice of $A,B$ as above, the series 
$F_{A(X+\gamma),B(X+\gamma)}$  (well defined provided $\gamma$ is not a rational integer $\leq 0$) which satisfies 
\begin{equation}\label{contiguous}
F_{A(X+\gamma),B(X+\gamma)}\left(1/z\right)=\left(-z\frac{d}{dz}+\gamma-1\right)\left(F_{A,B}\left(1/z\right)\right)\enspace.
\end{equation}

In other words, for each sequence of elements in $K$, it is possible to construct a chain of functions each linked to the next one by an order one differential operator.  For our purpose, it is enough to restrict ourselves to  finite chains.

Put $d=\max\{{\rm{deg}}\,A,{\rm{deg}}\,B\}$ and take $\boldsymbol{\gamma}=(\gamma_1,\ldots,\gamma_{d-1})\in K^{d-1}$. Let $s$ be an integer with ${{1\le s \le d}}$. 
We define the power series $F_{{s}}(\boldsymbol{\gamma},z)$ by 
\begin{align} \label{Fs1}
F_{{d}}(\boldsymbol{\gamma},z)=F(z), \ \ \ F_{{s}}(\boldsymbol{\gamma},z)=\sum_{k=0}^{\infty}(k+\gamma_1)\cdots (k+\gamma_{{d-s}})c_kz^{k+1} 
\end{align}
for $1\le s \le d-1$.
We denote $F_{{s}}(\boldsymbol{\gamma},z)$ by $F_{{s}}(z)$ when no confusion may arise.
Notice that $F_{{s}}(1/z)$ satisfies
$$F_{{s}}(1/z)=\left(-z \tfrac{d}{dz}+(\gamma_1-1)\right)\circ \cdots\circ\left (-z\tfrac{d}{dz}+(\gamma_{{d-s}}-1)\right) (F_{{d}}(1/z))\enspace.$$
\begin{remark} \label{Main ex}
Let $p, q$ be positive integers and $a_1,\ldots,a_p, b_1,\ldots,b_q\in K\setminus\{0\}$ such that none of them is a negative integer.
Put 
$$A(X)=(X+a_1+1)\cdots(X+a_p+1), \ \ \ B(X)=(X+b_1)\cdots(X+b_q)(X+1)$$ 
$$c_k=\dfrac{(a_1)_{k+1}\cdots(a_p)_{k+1}}{(b_1)_{k+1}\cdots(b_q)_{k+1}(k+1)!} \ \ (k\ge 0)\enspace.$$
Then, $(c_k)_{k\ge0}$ satisfies $$c_{k+1}=c_k\cdot\dfrac{A(k)}{B(k+1)}\enspace.$$
For this sequence, $$F\left(\dfrac{1}{z}\right)={}_{p}F_{q} \biggl(\begin{matrix} a_1,\ldots, a_p\\ b_1, \ldots, b_{q} \end{matrix} \biggm| \dfrac{1}{z}\biggr)-1
\enspace.$$
In the case of $p\ge q+1$ and $\gamma_1=a_{1}+1,\ldots,\gamma_{p-1}=a_{p-1}+1$, the series $F_s(1/z)$ has the expression:
{\begin{align}
F_{p}(1/z)&={}_{p}F_{q} \biggl(\begin{matrix} a_1,\ldots, a_p\\ b_1, \ldots, b_{q} \end{matrix} \biggm| \dfrac{1}{z}\biggr)-1\enspace, \label{Fs p>}\\
F_{s}(1/z)&=a_1\cdots a_{p-s} \left({}_{p}F_{q} \biggl(\begin{matrix} a_1+1,\ldots, a_{p-s}+1,a_{p-s+1},\ldots,a_p\\ b_1, \ldots, \ldots,b_{q} \end{matrix} \biggm| \dfrac{1}{z}\biggr)-1\right)\enspace, \nonumber
\end{align}
for $1\le s \le p-1$. 

\smallskip

In the case of $p<q+1$ and $\gamma_1=1,\gamma_2=b_{q},\ldots,\gamma_{q}=b_2$, the series $F_s(1/z)$ has the expression:
{\small{\begin{align}
F_{q+1}(1/z)&={}_{p}F_{q} \biggl(\begin{matrix} a_1,\ldots, a_p\\ b_1, \ldots, b_{q} \end{matrix} \biggm| \dfrac{1}{z}\biggr)-1\enspace, \label{Fs}\\
F_{s}(1/z)&=\dfrac{a_1\cdots a_p}{b_1\cdots b_{s}z}\cdot {}_{p}F_{q} \biggl(\begin{matrix} a_1+1,\ldots,\ldots,\ldots,a_p+1\\ b_1+1, \ldots, b_{s}+1,b_{s+1},\ldots,b_{q} \end{matrix} \biggm| \dfrac{1}{z}\biggr)\enspace, \nonumber
\end{align}}}
for $1\le s \le q$.}
\end{remark}

Throughout this section, we fix  $A(X)$ and $B(X)$ such that \eqref{AB} holds, and ensure that $\max\{{\rm{deg}}\,A,{\rm{deg}}\,B\} > 0$ and recall $d = \max\{{\rm{deg}}\,A,{\rm{deg}}\,B\}$, we moreover set  ${\rm{deg}}\,B = d'$. We denote by $\bold{c} = (c_k)_{k\ge0}$ the sequence satisfying \eqref{recurrence 1}, where $c_k \in K\setminus\{0\}$ defining $F$. 

Additionally, we fix $\gamma_1,\ldots,\gamma_{d-1}\in K$ (at this stage, it is not necessary to assume that $\gamma_i$ is not an integer $\leq 0$).  The chain of 
power series $F_{{s}}(z)$, defined in \eqref{Fs1} for the given sequence $\boldsymbol{c}$ and $\gamma_1, \ldots, \gamma_{d-1}$ in $K$ now fixed, is denoted by $F_{{s}}(z)$. 
Take $m$ as a strictly positive integer and $\alpha_1, \ldots, \alpha_m$ in $K\setminus\{0\}$ (at this stage it is not necessary to assume them pairwise distinct).

We are now in a position to define the operators that will play a role analogous to derivation and primitivation (which were enough to deal with simpler classes of functions like the polylogarithms, {\it confer} \cite{DHK2,DHK3,DHK4}).
\begin{definition} \label{defope}
\begin{itemize}
\item[$({\rm{i}})$] Let $\boldsymbol{c}=(c_k)_{k\geq 0}$ be a sequence of elements of $K\setminus\{0\}$.

Define $\mathcal{T}_{\bold{c}}\in {\rm{Aut}}_K(K[t])$ by 
\begin{align} \label{Phi}
\mathcal{T}_{\bold{c}}:K[t]\longrightarrow K[t]; \ t^k\mapsto \dfrac{t^k}{c_k}\enspace.
\end{align}
\item[$({\rm{ii}})$] Let $\gamma_1,\ldots,\gamma_{d-1}\in K$ and $\alpha_1,\ldots,\alpha_m\in K\setminus\{0\}$. 
We define 
\begin{align} 
\varphi_{i,d} =[\alpha_i]\circ {\rm{Eval}}_{\alpha_i}\circ \mathcal{T}_{\bold{c}}^{-1} \ \text{for} \ 1\le i \le m\enspace. \label{eval i}
\end{align}
\begin{align}
\varphi_{{i,{s}}}=\varphi_{i,d}\circ \left(t \tfrac{d}{dt}+\gamma_1\right)\circ \cdots \circ \left(t \tfrac{d}{dt}+\gamma_{{d-s}}\right) \ \text{for} \ 1\le i \le m, 1\leq s\leq d-1\enspace. \label{phi 0}
\end{align}
\end{itemize}
\end{definition}

The following statement is one of our new ingredients. 

\begin{lemma} \label{difference}

\begin{itemize}
\item[$({\rm{i}})$] The operators $\mathcal{T}_{\bold{c}}$ and $t\frac{d}{dt}$ commute.
\item[$({\rm{ii}})$]
Let $k$ be a positive integer. We have$:$
{\small{
$$[t^k]\circ \mathcal{T}_{\bold{c}}=\mathcal{T}_{\bold{c}}\circ A\left(t \tfrac{d}{dt}-1\right)\circ \cdots \circ A\left(t \tfrac{d}{dt}-k\right)\circ B\left(t \tfrac{d}{dt}\right)^{-1}\circ \cdots \circ B\left(t\tfrac{d}{dt}-(k-1)\right)^{-1}\circ [t^k] \enspace.$$
}}
Note that $B\left(t\frac{d}{dt}-j\right)$ is not necessarily invertible on the whole of $K[t]$. However, by Hypothesis~\eqref{AB},  and Facts~\eqref{elem} its restriction to the ideal $(t^k)$ $($stable subvector space$)$ is and the right hand side is thus well defined since the morphism $[t^k]$ {maps} onto $(t^k)$.
\end{itemize}
\end{lemma}
\begin{proof}{(i) is clear, we move to (ii).}
Let $m$ be a non-negative integer. Since $t^m$ is an eigenvector for all the operators involved (except multiplication by $t^k$), one gets its image by multiplication of eigenvalues: 
$$[t^k]\circ \mathcal{T}_{\bold{c}}(t^m)=\frac{1}{c_m}t^{k+m}\enspace,$$
similarly for the right hand side,
{\small{\begin{align*}
&\mathcal{T}_{\bold{c}}\circ A\left(t \tfrac{d}{dt}-1\right)\circ \cdots \circ A\left(t \tfrac{d}{dt}-k\right)\circ B\left(t \tfrac{d}{dt}\right)^{-1}\circ \cdots \circ B\left(t\tfrac{d}{dt}-(k-1)\right)^{-1}\circ [t^k](t^m)\\
&= \dfrac{1}{c_{m+k}}\dfrac{A(m+k-1)\cdots A(m)}{B(m+k)\cdots B(m+1)}t^{k+m}
\end{align*}
}}
Equality then follows from 
the recurrence relation $(\ref{recurrence 1})$ which yields 
$$\dfrac{1}{c_{m+k}}=\dfrac{B(m+k)\cdots B(m+1)}{A(m+k-1)\cdots A(m)}\cdot\dfrac{1}{c_m} \enspace,$$
which achieves the proof of $({\rm{ii}})$.
\end{proof}

\subsection{Construction of Pad\'e approximants}\label{pa}
We are now ready for our construction of Pad\'{e} approximants, of the hypergeometric functions at distinct points.
We define the order function ${\rm{ord}}_{\infty}$ at $z=\infty$ by
\begin{align*} 
{\rm{ord}}_{\infty}:K((1/z))\rightarrow \Z\cup \{\infty\}; \ \ \sum_{k} \dfrac{{{c_k}}}{z^k}\mapsto \min\{k\in \Z\mid {{c_k}}\neq 0\}\enspace.
\end{align*}    
We first recall the following lemma ({\it see} \cite{Feldman}):
\begin{lemma} \label{pade}
Let $r$ be a positive integer, $f_1(z),\ldots,f_r(z)\in (1/z)\cdot K[[1/z]]$ and $\boldsymbol{n}:=(n_1,\ldots,n_r)\in \N^{r}$.
Put $N:=\sum_{i=1}^rn_i$.
Let $M$ be a positive integer  with $M\ge N$. Then, there exists a family of polynomials 
$(P_0(z),P_{1}(z),\ldots,P_r(z))\in K[z]^{r+1}\setminus\{\bold{0}\}$ satisfying the following conditions:
\begin{align*} 
&(i) \ {\rm{deg}}\,P_{0}(z)\le M\enspace,\\
&(ii) \ {\rm{ord}}_{\infty} (P_{0}(z)f_j(z)-P_j(z))\ge n_j+1 \ \text{for} \ 1\le j \le r\enspace.
\end{align*}
\end{lemma}

\begin{definition}
We say that a vector of polynomials $(P_0(z),P_{1}(z),\ldots,P_r(z)) \in K[z]^{r+1}$ satisfying the properties $({\rm{i}})$ and $({\rm{ii}})$ a weight $\boldsymbol{n}$ and degree $M$ Pad\'{e}-type approximant of $(f_1,\ldots,f_r)$.
For such approximants $(P_0(z),P_{1}(z),\ldots,P_r(z))$ of $(f_1,\ldots,f_r)$, 
we call the formal Laurent series $(P_{0}(z)f_j(z)-P_{j}(z))_{1\le j \le r}$ weight $\boldsymbol{n}$ degree $M$ Pad\'{e}-type approximations of $(f_1,\ldots,f_r)$.
\end{definition}

The following statement provides for Pad\'e approximation in our situation.

\begin{proposition} \label{GHG pade} $(${\it confer} {\rm{\cite[Theorem~$5.5$]{ch brothers}}} $)$
We keep the above notation.
For a non-negative integer $\ell$, we define polynomials$:$
\begin{align}
& H_{\ell}(t)=t^{\ell}\prod_{i=1}^m(t-\alpha_i)^{dn}\enspace, \label{Hl}\\
&P_{\ell}(z)=\left[\dfrac{1}{(n-1)!^{d'}}\right]\circ {\rm{Eval}}_z\circ \mathcal{T}_{\bold{c}} \bigcirc_{j=1}^{n-1} B\left(t \tfrac{d}{dt}+j\right)\left(H_{\ell}(t)\right)\enspace, \label{Pl}\\
&P_{\ell,i,{{s}}}(z)=\varphi_{i,{s}}\left(\dfrac{P_{\ell}(z)-P_{\ell}(t)}{z-t}\right) \ \text{for} \ 1\le i \le m, {{1 \le s \le d}}\enspace, \label{Plis}
\end{align}
where $\mathcal{T}_{\boldsymbol{c}}$ and $\varphi_{i,s}$ are defined in Definition $\ref{defope}$.
Then, $(P_{\ell}(z),P_{\ell,i,{s}}(z))_{1\le i \le m, {{1 \le s \le d}}}$ form a weight $(n,\ldots,n)\in \N^{dm}$ and degree $dmn+\ell$ Pad\'{e}-type approximant of
$(F_{{s}}(\alpha_i/z))_{1\le i \le m, {{1 \le s \le d}}}$.
\end{proposition}
\begin{proof}
By the definition of $P_{\ell}(z)$, 
$$
{\rm{deg}}\,P_{\ell}(z)=dmn+\ell\enspace.
$$
Hence the required condition on the degree is verified.  

Let $k$ be an integer with $0\le k \le n-1$. 
Using \cite[Lemma $2.3$]{Kaw}, it is sufficient to prove 
\begin{equation} \label{zero}
\varphi_{i,{s}}\left(t^kP_{\ell}(t)\right)=0 \ \ \text{for} \ \ 1\le i \le m, \, 1 \le s \le d \enspace.\tag{*}
\end{equation}

To ease notation, set:
$$
\begin{array}{llllll}\displaystyle
\mathcal{A} & = & \displaystyle \bigcirc_{j=1}^kA\left(t\tfrac{d}{dt}-j\right), &\displaystyle \mathcal{B}&= &\displaystyle\bigcirc_{j=1}^{n-1-k}B\left(t \tfrac{d}{dt}+j\right)\enspace, \kern20pt \\\rule{0mm}{8mm}\mathcal {C}& = & \bigcirc_{j=1}^{n-1}B\left(t\tfrac{d}{dt}+j\right),  & \mathcal{D} & =&  \bigcirc_{j=0}^{k-1}B\left(t \tfrac{d}{dt}-j\right)\enspace,
\end{array} $$
and note that $$\bigcirc_{j=1}^{n-1}B\left(t \tfrac{d}{dt}+(j-k)\right)\circ \mathcal{D}^{-1}=\mathcal{B}\enspace.$$
By Lemma~$\ref{difference}$~$({\rm{ii}})$,   
\begin{align*}
(n-1)!^{d'}t^kP_{\ell}(t)&=[t^k]\circ {\mathcal{T}}_{\bold{c}}\circ \mathcal{C}\left(t^{\ell}\prod_{i=1}^m(t-\alpha_i)^{dn}\right) \\
&={\mathcal{T}}_{\bold{c}} \circ \mathcal{A}\circ\mathcal{D}^{-1}\circ [t^k]\circ \mathcal{C}\left(t^{\ell}\prod_{i=1}^m(t-\alpha_i)^{dn}\right)\enspace.\end{align*} 
Now, taking into account Facts~\ref{elem}~$({\rm{iii}})$ applied to $[t^k]\circ \mathcal{C}$,
\begin{align*} 
(n-1)!^{d'}t^kP_{\ell}(t)&={\mathcal{T}}_{\bold{c}}\circ \mathcal{A}\circ \mathcal{D}^{-1} \circ \bigcirc_{j=1}^{n-1}B\left(t \tfrac{d}{dt}+(j-k)\right)\left(t^{\ell+k}\prod_{i=1}^m(t-\alpha_i)^{dn}\right)\\ & ={\mathcal{T}}_{\bold{c}}\circ \mathcal{A}\circ \mathcal{B}\left(t^{\ell+k}\prod_{i=1}^m(t-\alpha_i)^{dn}\right) \enspace,
\end{align*}

Therefore, taking into account Definition~\ref{defope}, and Lemma~\ref{difference} $({\rm{i}})$ for the last equality:
{\footnotesize{
\begin{align}
\varphi_{i,{s}}((n-1)!^{d'}t^kP_{\ell}(t))
&=\varphi_{i,{s}}\circ {\mathcal{T}}_{\bold{c}}\circ\mathcal{A}\circ \mathcal{B} \left(t^{\ell+k}\prod_{i=1}^m(t-\alpha_i)^{dn}\right) \nonumber\\
&=\varphi_{i,d}\circ\bigcirc_{j=1}^{{d-s}} \left(t \tfrac{d}{dt}+\gamma_j\right) \circ  {\mathcal{T}}_{\bold{c}}\circ\mathcal{A}\circ\mathcal{B}\left(t^{\ell+k}\prod_{i=1}^m(t-\alpha_i)^{dn}\right) \nonumber\\
&=[\alpha_i]\circ\ev_{\alpha_i}\circ\mathcal{T}_{\bold{c}}^{-1}\bigcirc_{j=1}^{{d-s}} \left(t\tfrac{d}{dt}+\gamma_j\right) \circ  {\mathcal{T}}_{\bold{c}}\circ\mathcal{A}\circ\mathcal{B}\left(t^{\ell+k}\prod_{i=1}^m(t-\alpha_i)^{dn}\right) \nonumber\\
&=[\alpha_i]\circ {\rm{Eval}}_{\alpha_i}\bigcirc_{j^=1}^{{d-s}} \left(t \tfrac{d}{dt}+\gamma_{j}\right)\circ\mathcal{A}\circ\mathcal{B}\left(t^{\ell+k}\prod_{i=1}^m(t-\alpha_i)^{dn}\right)\enspace. \nonumber
\end{align} 
}}
Since
\begin{align*}
{\rm{deg}} \left(\prod_{j^{\prime \prime}=1}^{{d-s}} (X+\gamma_{j^{\prime \prime}})\prod_{j^{\prime}=1}^kA(X-j^{\prime})\prod_{j=1}^{n-1-k}B(X+j)\right)
&\le {{d-s}}+dk+d(n-1-k)\\
&\le dn-1\enspace,
\end{align*}
thanks to the Leibniz rule,  the differential operator $\bigcirc_{j^=1}^{{d-s}} \left(t \tfrac{d}{dt}+\gamma_{j}\right)\circ\mathcal{A}\circ\mathcal{B}$ is of order at most $dn-1$, hence,  the polynomial 
$$\bigcirc_{j^{\prime \prime}=1}^{{d-s}} (t\tfrac{d}{dt}+\gamma_{j^{\prime \prime}})\circ\mathcal{A}\circ\mathcal{B}\left(t^{\ell+k}\prod_{i=1}^m(t-\alpha_i)^{dn}\right)$$ 
belongs to the ideal $(t-\alpha_i)={\rm{ker}}\, {\rm{Eval}}_{\alpha_i}$. 
Consequently we have \eqref{zero}, hence completing the proof of the proposition\footnote{Note that a similar construction was also considered by D.~V.~Chudnovsky and G.~V.~Chudnovsky in \cite[Theorem~$5.5$]{ch brothers},
but without arithmetic application. See also a related work by Matala-aho \cite{Matala}.}.
\end{proof}

\begin{remark}
The polynomial $P_{\ell}(z)$ does not depend on the choice of $\gamma_1,\ldots,\gamma_{d-1}\in K$.
By contrast, the polynomials $P_{\ell,i,{s}}(z)$ depend on them.
\end{remark}
\begin{remark}
Let $d,m$ be strictly positive integers. Let $x\in K$, supposed to be non-negative integer and $\alpha_1,\ldots,\alpha_m\in K\setminus\{0\}$ be pairwise distinct.
Put $A(X)=B(X)=(X+x+1)^d$ and $c_k=1/(k+x+1)^d$. Then, $$c_{k+1}=c_k\cdot\dfrac{A(k)}{B(k+1)}\enspace.$$
Put $\gamma_1=\cdots=\gamma_{d-1}=x+1$. This gives us 
\begin{align} \label{Fis}
F_{{s}}(\alpha_i/z)=\sum_{k=0}^{\infty}\dfrac{1}{(k+x+1)^{s}}\cdot\dfrac{\alpha^{k+1}_i}{z^{k+1}}=\Phi_{s}(x,\alpha_i/z) \ \ (1\le i \le m, {1 \le s \le d})\enspace,
\end{align}
where $\Phi_{{s}}(x,1/z)$ is the $s$-th Lerch function (generalized polylogarithmic function, {\it confer} \cite{DHK3}). 
In this case, we have ${\mathcal{T}}_{\bold{c}}=(t\tfrac{d}{dt}+x+1)^d/(x+1)^d$ and 
$$P_{\ell}(z)=\left[\dfrac{1}{(x+1)^d \cdot (n-1)!^d}\right]\circ {\rm{Eval}}_z \bigcirc_{j=1}^n(t\tfrac{d}{dt}+x+j)^d\left(t^{\ell}\prod_{i=1}^m(t-\alpha_i)^{dn}\right)\enspace.$$
The polynomial $(x+1)^d/n^dP_{\ell}(z)$ gives Pad\'{e}-type approximant of Lerch functions in \cite[Theorem~$3.8$]{DHK2}.
\end{remark}

\section{Non-vanishing of the generalized Wronskian of Hermite type}\label{nonvan} 
Throughout this section, we consider the following setting: {$K$ is a field of characteristic $0$ and $A,B\in K[X]$ are monic polynomials satisfying $\eqref{AB}$ with $\min\{{\rm{deg}}\,A, {\rm{deg}}\,B\}>0$.
Put $$ \max\{{\rm{deg}}\, A, {\rm{deg}}\,B\}=d, \quad \deg\,A=d'', \quad \deg\,B=d' \enspace. $$
{{By replacing $K$ with an appropriate finite extension, we may assume that $A(X)$ and $B(X)$ are decomposable in $K$ and put
\begin{align*}
A(X)=(X+\eta_1)\cdots (X+\eta_{d''}), \quad B(X)=(X+\zeta_1)\cdots (X+\zeta_{d'})\enspace,
\end{align*}
where $\eta_i,\zeta_j\in K\setminus\Z_{\le 0}$.

\medskip

We fix the sequence $\boldsymbol{c}=(c_k)_{k\ge0}$ satisfying 
\eqref{recurrence 1} for the given polynomials $A(X)$ and $B(X)$.
Let $\boldsymbol{\alpha}=(\alpha_1,\ldots,\alpha_m)\in (K\setminus\{0\})^m$ 
and $\boldsymbol{\gamma}=(\gamma_1,\ldots,\gamma_{d-1})\in K^{d-1}$.
Let us fix a positive integer $n$. 
For a non-negative integer $\ell$ with $ 0\le \ell \le dm$, recall the polynomials $P_{\ell}(z), P_{\ell,i,{s}}(z)$ defined in Proposition~\ref{GHG pade}} for {these choices of $A, B,\boldsymbol{\alpha}$ and $\boldsymbol{\gamma}$}.
We define column vectors $\vec{p}_{\ell}(z)\in K[z]^{dm+1}$ by
\begin{align*}
&\vec{p}_{\ell}(z)={}^t\Biggl(P_{\ell}(z),
{P_{\ell,1,{1}}(z),\ldots, P_{\ell,1,{d}}(z)}, \ldots, {P_{\ell,m,{1}}(z),\ldots, P_{\ell,m,{d}}(z)}\Biggr)\enspace,
\end{align*}
and put
$$
\Delta_n(z)=\Delta(z)=
{\rm{det}}  {\begin{pmatrix}
\vec{p}_{0}(z) \ \cdots \ \vec{p}_{ dm}(z)
\end{pmatrix}}\enspace.
$$
The aim of this section is to prove the following proposition.

\begin{proposition} \label{non zero det}
{{Assume $\alpha_1,\ldots,\alpha_m$ are pairwise distinct and 
$$\eta_i-\zeta_j  \ \ \ (1\leq i\leq d'', 1\leq j\leq d')\medskip,$$ is not a positive integer.}}
Then $\Delta(z)\in K\setminus\{0\}$.
\end{proposition}
\begin{remark} Proposition \ref{non zero det} together with the general criterion proven by Marcovecchio in \cite{marc3}
can lead to relax metric constraints on the pair $(\boldsymbol{\alpha},\beta)$ in Theorem~\ref{hypergeometric}. 
However, the cost of the improvement is that one gets only ``almost linear independence'', instead of the
linear independence. \footnote{
It should be noted that in fact these determinants can actually be computed albeit not with the proof provided here. This requires a generalisation of the approach followed in our previous paper \cite{DHK4}; some minors can be computed.} 
\end{remark} 

{{\begin{remark} \label{change Delta}
In this remark, we emphasize {that though the choices of $A,B$ are crucial, since each set defines fundamentally different special functions, as is the choice of $\boldsymbol{\alpha}$ that fixes the special values studied, in contrast the choice of $\boldsymbol{\gamma}$ is not that significative, indeed the differential operators that $\boldsymbol{\gamma}$ define are linked by simple linear transformations. In particular,}  the non-vanishing of $\Delta(z)$ {\emph{does not}} depend on the choice of $\boldsymbol{\gamma}$. 

Take $\tilde{\boldsymbol{\gamma}}=(\tilde{\gamma}_1,\ldots,\tilde{\gamma}_{d-1})\in K^{d-1}$. Denote the $K$-morphism $\varphi_{F_s(\tilde{\boldsymbol{\gamma}},\alpha_i/z)}$ by $\tilde{\varphi}_{i,s}$, the polynomial 
$\tilde{\varphi}_{i,s}\left((P_{\ell}(z)-P_{\ell}(t))/(z-t)\right)$ by $\tilde{P}_{\ell,i,s}(z)$ for $1\le i \le m$ and $1\le s \le d$, 
$$
\tilde{p}_{\ell}(z)={}^t\biggl({P}_{\ell}(z),
{\tilde{P}_{\ell,1,{1}}(z),\ldots, \tilde{P}_{\ell,1,{d}}(z)}, \ldots, {\tilde{P}_{\ell,m,{1}}(z),\ldots, \tilde{P}_{\ell,m,{d}}(z)}\biggr)\enspace,
$$
$$
\tilde{\Delta}(z)=
{\rm{det}}  {\begin{pmatrix}
\tilde{p}_{0}(z) \ \cdots \ \tilde{p}_{dm}(z)
\end{pmatrix}}\enspace.
$$
{Set\footnote{When no confusion can occur, we omit the subscript $\boldsymbol{\gamma}$.}  for $1\leq s\leq d$, $D_{s,\boldsymbol{\gamma}}=(X+\gamma_{1})\cdots (X+\gamma_{d-s})$ with empty product (for $s=d$) equal to 1,}
There exists a $(d\times d)$ upper triangular matrix $A(\boldsymbol{\gamma},\tilde{\boldsymbol{\gamma}})$ with all the diagonal entries are $1$ such that 
\begin{align} \label{base change}
{}^{t}(D_{1,\tilde{\boldsymbol{\gamma}}},\ldots, D_{d,\tilde{\boldsymbol{\gamma}}})=A(\boldsymbol{\gamma},\tilde{\boldsymbol{\gamma}})\cdot {}^t(D_{1,\boldsymbol{\gamma}}, \ldots, D_{d,\boldsymbol{\gamma}})\enspace.
\end{align} 
Put the $(dm+1)\times (dm+1)$ upper triangular matrix with all diagonal entries are $1$ by
$$
{{B(\boldsymbol{\gamma},\tilde{\boldsymbol{\gamma}})=}}\begin{pmatrix}
1 & 0 & 0 & \cdots & 0 \\
0 & A(\boldsymbol{\gamma},\tilde{\boldsymbol{\gamma}}) & 0 & \cdots  & 0\\
\vdots & \cdots & \ddots  & \cdots & \vdots \\ 
0 & 0 & \vdots & \ddots & \vdots \\
0 & 0 & \cdots  & \cdots & A(\boldsymbol{\gamma},\tilde{\boldsymbol{\gamma}})
\end{pmatrix}\enspace. 
$$
Equation~\eqref{base change} implies 
$${}^t(\tilde{\varphi}_{i,1},\ldots,\tilde{\varphi}_{i,d})={{B}}(\boldsymbol{\gamma},\tilde{\boldsymbol{\gamma}})\cdot {}^t(\varphi_{i,1},\ldots,\varphi_{i,d}) \ \ \ 1\le i \le m\enspace,$$
therefore $\tilde{p}_{\ell}(z)={{B}}(\boldsymbol{\gamma},\tilde{\boldsymbol{\gamma}})\vec{p}_{\ell}(z)$ for any $0\le \ell \le dm$. This yields the equality: 
$$\tilde{\Delta}(z)={\rm{det}}\,{{B}}(\boldsymbol{\gamma},\tilde{\boldsymbol{\gamma}})\cdot \Delta(z)=\Delta(z)\enspace.$$
\end{remark}}}

\subsection{First Step}
In this subsection, we establish the determinant satisfies $\Delta(z)\in K$. 
Define column vectors ${\vec{q}}_{\ell}\in K^{dm}$ by
\begin{align*}
\vec{q}_{\ell}={}^t\Biggl(\varphi_{{1,{1}}}(t^nP_{\ell}(t)),\ldots, \varphi_{{1,{d}}}(t^nP_{\ell}(t)),\ldots, \varphi_{{m,{1}}}(t^nP_{\ell}(t)),\ldots, \varphi_{{m,{d}}}(t^nP_{\ell}(t))\Biggr)\enspace
\end{align*}
for $0\le \ell \le dm-1$ 
and a determinant
\begin{align} \label{Theta}
\Theta= 
{\rm{det}}  {\begin{pmatrix}
\vec{q}_{0} \ \cdots \ \vec{q}_{dm-1}
\end{pmatrix}}\enspace.
\end{align}
\begin{lemma}  $(${\it confer} {\rm{\cite[Lemma~$4.2$]{DHK4}}}$)$.\label{sufficient condition}
There exists a non-zero element $c\in K$ with $\Delta(z)=c\cdot \Theta$.
\end{lemma}
\begin{proof}
Put $R_{\ell,i,{s}}(z)=P_{\ell}(z)F_{{s}}(\alpha_i/z)-P_{\ell,i,{s}}(z)$ and
$$c=\dfrac{1}{(dm(n+1))!}\left(\dfrac{d}{dz}\right)^{dm(n+1)}\kern-7ptP_{dm}(z)\enspace,$$ be the coefficient of highest degree ($=dmn+dm$) of the polynomial $P_{dm}(z)$. 
Consider
{\footnotesize{$$\left(\kern-5pt\begin{array}{ccccc}
1 & 0 & \cdots & \cdots & 0\\
F_{{1}}(\alpha_1/z)  & -1 & 0 & \cdots & \vdots \\
F_{{2}}(\alpha_1/z)  & 0 & -1 & \cdots &0 \\ 
\vdots & \vdots & \vdots & \ddots & \vdots \\ 
F_{{d}}(\alpha_1/z)  & 0 & \cdots & 0 & -1
\end{array}\kern-5pt\right)\kern-3pt\left(\kern-5pt\begin{array}{ccc} \vec{p}_0(z) & \cdots & \vec{p}_{dm}(z)\\	
\end{array}\kern-5pt\right)=\kern-2pt\left(\kern-5pt\begin{array}{ccc} P_0(z)& \ldots & P_{dm}(z)\\
R_{0,1,{1}}(z)&\cdots &R_{dm,1,{d}}(z)\\ \vdots & \ddots & \vdots\\ R _{0,m,{1}}(z) & \cdots & R_{dm,m,{d}}(z)\kern-5pt\end{array}\right)\enspace.$$}}
Since the entries of the first line are (by definition on $P_{\ell}(z)$) polynomials of degree $dmn+dm$ and entries of the other lines (by Proposition~$\ref{GHG pade}$) are of valuation at least $n+1$, 
we can apply \cite[Lemma~$3.11$ $({\rm{ii}})$]{DHK4}.  
We need only to check the coefficients of highest degree (for the first line) and of minimal valuation  (for all the other lines). 
By construction, the vector of highest degree ($=dmn+dm$) for the first line is $(0,\ldots,0,c)z^{dmn+dm}$ and
\begin{align*}
R_{\ell,i,{s}}(z)=\sum_{k=n}^{\infty}\dfrac{{\varphi_{i,{s}}}(t^kP_{\ell}(t))}{z^{k+1}}\enspace, 
\end{align*} for $0\le \ell \le dm, \ 1\le i \le m$ and ${{1 \le s \le d}}$. 
So, by \cite[Lemma~$3.11$ $({\rm{ii}})$]{DHK4}
$$\Delta(z)=\pm c\cdot \det\left(\varphi_{i,{s}}(t^nP_{\ell}(t))\right)_{\substack{0\leq l\leq dm-1\\ 1\leq i\leq m, {1\leq s \leq d}}}\enspace,$$
as claimed. 
\end{proof}

\subsection{Second step}\label{second}
Relying on Lemma~$\ref{sufficient condition}$, we study here the value
$\Theta$ defined in \eqref{Theta}. 
From this subsection, we specify the choice of $\gamma_1, \ldots, \gamma_{d-1} \in K$ and {{take $\gamma_d\in K$}} as follows. 
{{\begin{choice} \label{gamma}
We fix $\gamma_i = \zeta_i$ for $1 \leq i \leq d'$, and choose $\gamma_{d'+1}, \ldots, \gamma_{d}$ arbitrarily (if $d'<d$). 
\end{choice}
We recall that Proposition~$\ref{non zero det}$ does not depend on this choice (see Remark~$\ref{change Delta}$).}}

Let the $dm$ by $dm$ matrix
\begin{align} \label{Mn}
\mathcal{M}_n:=
\begin{pmatrix}
{\rm{Eval}}_{\alpha_i} \bigcirc_{w={d-s+1}}^{{{d'}}}\left(t \tfrac{d}{dt}+\gamma_{w}\right)^{-1} (t^nH_{\ell}(t))\\
{\rm{Eval}}_{\alpha_{i}} \circ\left(t \tfrac{d}{dt}\right)^{dn+{s'-1}} (t^nH_{\ell}(t))
\end{pmatrix}_{\substack{0\le \ell \le dm-1 \\ 1\le i \le m \\ {d-d'+1 \le s \le d} \\ {1 \le s' \le d-d'}}}\enspace,
\end{align}
where $H_{\ell}(t)$ is defined in Proposition~\ref{GHG pade}.
}}
We then simplify the determinant $\Theta$ to prove its non-vanishing property.
\begin{lemma} \label{simplify}
There exist elements $a_{s,0}\in K$ for $d-d'+1\leq s \leq d$ such that
\begin{align*}
\Theta=\dfrac{\prod_{i=1}^m\alpha^d_i{\prod_{s=d-d'+1}^{d}}a_{s,0}^m}{(n-1)!^{d^2m}}\cdot {\rm{det}}\, \mathcal{M}_n\enspace.
\end{align*}
\end{lemma}
\begin{proof} {Put $K$-endomorphisms 
$$\begin{array}{llllll}
\mathcal{D}_s & = & \bigcirc_{j=1}^{d-s}\left(t\tfrac{d}{dt}+\gamma_j\right)\enspace,
& \mathcal{A}& =& \bigcirc_{j=1}^nA\left(t \tfrac{d}{dt}-j\right)\enspace, \\ \rule{0mm}{8mm} \mathcal{B} & =& \bigcirc_{j=0}^{n-1}B\left(t \tfrac{d}{dt}-j\right)^{-1}\enspace, & \mathcal{B}'&=&\bigcirc_{j=1}^{n-1}B\left(t\frac{d}{dt}+j\right)\enspace,\end{array}$$
{{with the convention $\mathcal{D}_d=\id$.}} We have
{\small{$$\begin{array}{lcl} (n-1)!^{d'}\varphi_{i,s}(t^nP_{\ell}) & =& \varphi_{i,d}\circ \mathcal{D} _s\circ [t^n]\circ  \mathcal{T}_{\bold{c}} \circ\mathcal{B}'\left(H_{\ell}\right)
\\ & =& \rule{0mm}{8mm} \varphi_{i,d}\circ \mathcal{D}_s\circ \mathcal{T}_{\bold{c}}\circ\mathcal{A}\circ\mathcal{B}\circ [t^n]\circ\mathcal{B}'(H_{\ell})
\\ & =& \rule{0mm}{8mm} [\alpha_i]\circ \ev_{\alpha_i}\circ\mathcal{T}_{\bold{c}}^{-1}\circ \mathcal{D}_s\circ \mathcal{T}_{\bold{c}}\circ\mathcal{A}\circ\mathcal{B}\circ [t^n]\circ\mathcal{B}'(H_{\ell})\\ & =& 
 \rule{0mm}{8mm} [\alpha_i]\circ \ev_{\alpha_i}\circ \mathcal{D}_s\circ\mathcal{A}\circ\mathcal{B}\circ \bigcirc_{j=1}^{n-1}B\left(t\frac{d}{dt}+(j-n)\right)(t^nH_{\ell})\\ 
&=& \rule{0mm}{8mm} [\alpha_i]\circ \ev_{\alpha_i}\circ \mathcal{D}_s\circ\mathcal{A}\circ {{B}}\left(t\frac{d}{dt}\right)^{-1}(t^nH_{\ell})\enspace.
\end{array}$$
}}
We now consider the euclidean division of {{$\mathcal{A}$ by $B(t\tfrac{d}{dt})/\mathcal{D}_s$ in $K[t,\tfrac{d}{dt}]$ with the convention $B(t\tfrac{d}{dt})/\mathcal{D}_s=B(t\tfrac{d}{dt})$ if $d'<d-s$}}:
\begin{align*}
{{\mathcal{A}={\mathcal{Q}}_s\circ  B(t \tfrac{d}{dt})/\mathcal{D}_s+\mathcal{R}_s\enspace,}}
\end{align*}
so that
$$ \varphi_{i,s}(t^nP_{\ell})=\left[\dfrac{\alpha_i}{(n-1)!^{d'}}\right]\circ \ev_{\alpha_i}\circ \left[\mathcal{Q}_s+ \mathcal{D}_s\circ \mathcal{R}_s\circ B\left(t\frac{d}{dt}\right)^{-1}\right](t^nH_{\ell})\enspace.$$
Note that 
$${\rm{ord}}(\mathcal{Q}_s)=n\deg(A)+{\rm{ord}}(\mathcal{D}_s)-\deg(B)=d''n-d'+d-s\enspace,$$
and distinguish two cases:

{\textbf Case I}: $d-s< d'$. In this case,  ${\rm{ord}}(\mathcal{Q}_s)< d''n\leq dn$, so $\mathcal{Q}_s$ is a differential operator of order $\leq dn-1$, so, by Leibniz rule, 
$$\mathcal{Q}_s(t^nH_{\ell})\in (t-\alpha_i)$$ since $t^nH_{\ell}$ belongs to the ideal $(t-\alpha_i)^{dn}$ and since $\ker\,\ev_{\alpha_i}=(t-\alpha_i)$,
$${\rm{Eval}}_{\alpha_i}\circ \mathcal{Q}_s(t^nH_{\ell})=0$$ and we can simplify
$$ \varphi_{i,s}(t^nP_{\ell})=\left[\dfrac{\alpha_i}{(n-1)!^{d'}}\right]\circ \ev_{\alpha_i}\circ \left[\mathcal{D}_s\circ \mathcal{R}_s\circ B\left(t\frac{d}{dt}\right)^{-1}\right](t^nH_{\ell})\enspace.$$

{{Put the polynomials $D_s(X)=\prod_{j=1}^{d-s}(X+\gamma_j)$ and $R_s(X)\in K[X]$ such that $R_s(t \tfrac{d}{dt})=\mathcal{R}_s$.}}
We now choose the following natural $K$-basis for the quotient $K[X]/(B/D_s)$:
\begin{align*}
e_0&=1\enspace,\\
e_1&=(X+\gamma_{{d-s+1}}),\\
     &\vdots\\
 e_j&=\prod_{\ell=1}^{j}(X+\gamma_{d-s+\ell})\enspace, \\
      & \vdots \\
e_{d'+s-d-1}&=(X+\gamma_{{d-s+1}})\ldots (X+\gamma_{d'-1})\enspace,
\end{align*}
and write $R_s$ in this basis:
$$R_s=\sum_{j=0}^{d'+s-d-1}a_{s,j}e_j$$
so that 
$$D_sR_sB^{-1}=\sum_{j=0}^{d'+s-d-1}a_{s,j}\left[\prod_{\ell={{d-s+j+1}}}^{d'}(X+\gamma_{\ell})\right]^{-1}\enspace.$$

Define a column vector:
$$C=\left(\begin{array}{c}
(X+\gamma_{d'})^{-1}\\
\left[(X+\gamma_{d'-1}) (X+\gamma_{d'})\right]^{-1}\\
\vdots\\
\left[\prod_{\ell=j+1}^{d'}(X+\gamma_{\ell})\right]^{-1}\\\vdots\\
\left[(X+\gamma_{1})\ldots (X+\gamma_{d'})\right]^{-1}
\end{array}\right)\enspace.$$
Then,  the column vector
$$(D_sR_sB^{-1})_{d-d'+1\leq s\leq d}=\left(\begin{array}{ccccc} a_{{d-d'+1},0}&0 & 0& \ldots & 0\\ a_{d-d'+2,1} & a_{d-d'+2,0} & 0& \cdots & 0
\\ \vdots & \ddots & \ddots &\cdots &\vdots\\ 
a_{d,d'-1} & \cdots& \cdots &\cdots & a_{d,0}
\end{array}\right)\cdot C\enspace.$$
}
{Denote by $T$ the above upper triangular matrix and note that the diagonal coefficients satisfy:
\begin{align} \label{as0}
a_{s,0}=R_s(-\gamma_{d-s+1})=\prod_{j=0}^{n-1}A(-\gamma_{d-s+1}-j)\enspace.
\end{align}
We can now compute  taking into account the fact that $T$ is a scalar matrix, hence commutes with $[\alpha_i]\circ\ev_{\alpha_i}$,
\begin{align*}
\Theta&=\det(\varphi_{i,s}(t^nP_{\ell}))_{i,s}\\
           &=\frac{\prod_{i=1}^m\alpha_i^d\prod_{s=d-d'+1}^da^m_{s,0}}{(n-1)!^{dd'm}}
\begin{pmatrix}{\rm{Eval}}_{\alpha_i} \bigcirc_{{{w}}={d-s+1}}^{{d}}\left(t \tfrac{d}{dt}+\gamma_{w}\right)^{-1} (t^nH_{\ell}(t))\\
\end{pmatrix}_{\substack{0\le \ell \le dm-1 \\ 1\le i \le m \\ {1\le s \le d}}}\enspace.
\end{align*}

{\textbf{Case II}}: ${d-s\ge d}'$. 
\medskip
Note that this case can occur only if $d=d''>d'$.  In this case, the polynomial $B$ divides $D_s$, hence, the operator
$\mathcal{A}\circ \mathcal{D}_s\circ B(t\tfrac{d}{dt})^{-1} {{\in K[t\tfrac{d}{dt}]}}$ that we can readily write in the basis {{$((t\tfrac{d}{dt})^n)_{n\ge 0}$}}, 
$$\varphi_{i,s}(t^nP_{\ell})=\mathcal{A}\mathcal{D}_s\mathcal{B}^{-1}(t^nH_{\ell}(t))=\sum_{k=0}^{dn+d-d'-s}b_{s,k} \left(t\frac{d}{dt}\right)^k(t^nH_{\ell}(t))\enspace.$$

Again, the terms of degree $\leq dn-1$ lie in the kernel of $\ev_{\alpha_i}$ and this simplifies in
$$\varphi_{i,s}(t^nP_{\ell})=\sum_{k=dn}^{dn+d-d'-s}b_{s,k}\left(t\frac{d}{dt}\right)^k(t^nH_{\ell}(t))\enspace.$$}
{Since all the polynomials $\mathcal{A},\mathcal{D}_s,B(t\tfrac{d}{dt})$ are monic, 
we again can transform the expression via a triangular matrix (but this time with diagonal entries $1$ since the terms of highest degree $b_{s,dn+d-d'+s}=1$). This yields the lemma in case II.}
\end{proof}

We now make the following. 
\begin{choice}\label{choicegamma}
Assume that for any $1\leq i\leq d''$ and any $1\leq j\leq d'$,
$$\eta_i-\zeta_j$$ is not a positive integer.
\end{choice}
\begin{corollary} \label{a0s 0}
Under Choice $\ref{gamma}$ and $\ref{choicegamma}$ the factor $\prod_{s=d-d'+1}^{d}a^m_{s,0}$ appearing in Lemma~$\ref{simplify}$ is non-zero.
\end{corollary}
\begin{proof} By Equation~\eqref{as0}, $a_{s,0}=\prod_{j=1}^{n-1}A(-\gamma_{d-s+1}-j)$ for $d-d'+1\le s \le d$. 
Therefore the choice of $\gamma_i$ implies the assertion.
\end{proof}

\subsection{Final step}
{{We keep the notation of Subsection $\ref{second}$ and assume $\alpha_1,\ldots,\alpha_m$ are pairwise distinct.
We now prove the non-vanishing of ${\rm{det}}\,\mathcal{M}_n$ without explicitly computing its exact value. 
The strategy we employ is a differential analogue of the approach used in \cite[Section $8$]{KP2}. We continue Choice $\ref{gamma}$ for $\gamma_w$, in particular we have $\gamma_w\notin \Z_{\le 0}$ for $1\le w \le d'$.}}

Denote the $K$-morphisms in the definition of $\mathcal{M}_n$ (see Equation~\eqref{Mn}) by
\begin{align*}
\psi_{i,s}:K[t]\longrightarrow K; \ \ t^k \mapsto
\begin{cases}
\prod_{w=d-s+1}^{d'}\dfrac{\alpha_i^k}{k+\gamma_w}      & \ d-d'+1\le s \le d\\
k^{dn+s-1}\alpha_i^k  & \ 1\le s\le d-d'.
\end{cases}
\end{align*}
Notice 
\begin{align} \label{psi}
\psi_{i,s}=\begin{cases}
{\rm{Eval}}_{\alpha} \bigcirc_{w=d-s+1}^{d'}\left(t\dfrac{d}{dt}+\gamma_w\right)^{-1}      & \ d-d'+1\le s \le d\\
{\rm{Eval}}_{\alpha}\circ \left(t\dfrac{d}{dt}\right)^{dn+s-1}  & \ 1\le s\le d-d'.
\end{cases}
\end{align}
Let $\boldsymbol{q}={}^t (q_{\ell})_{0\le \ell \le dm-1}\in K^{dm}$ be a vector satisfying 
\begin{align} \label{kernel}
\mathcal{M}_n\cdot \boldsymbol{q}=0.
\end{align} 
Put 
\begin{align} \label{poly Q}
Q(t):=t^n\sum_{\ell=0}^{dm-1} q_{\ell} H_{\ell}(t).
\end{align}
Using linearity of the morphisms $\psi_{i,s}$, Equation~\eqref{kernel} implies $$Q(t)\in \bigcap_{i=1}^m \bigcap_{s=1}^d {\rm{ker}}\,\psi_{i,s}.$$ 
The regularity of the matrix $\mathcal{M}_n$ is equivalent to the following statement: 
\begin{proposition} \label{Q=0}
We have $Q(t)=0$.
\end{proposition}
To prove Proposition~$\ref{Q=0}$, we study the kernel of $\psi_{i,s}$.
First we consider the case $d-d'+1\le s \le d$.
\begin{lemma} \label{basis}
Let $r\in \N$, $\alpha\in K\setminus \{0\}$ and $\gamma_1,\ldots,\gamma_r\in K\setminus \Z_{\le 0}$. Then the following identity holds.
$$\bigcap_{s=1}^r\bigcirc_{w=1}^s \left(t\dfrac{d}{dt}+\gamma_w\right)\circ[t-\alpha](K[t])=\bigcirc_{w=1}^r\left(t\dfrac{d}{dt}+\gamma_w\right)\circ[(t-\alpha)^r](K[t])\enspace.$$
\end{lemma}
\begin{proof}
The Leibniz formula yields $$\bigcirc_{w=1}^r\left(t\dfrac{d}{dt}+\gamma_w\right)\circ[(t-\alpha)^r](K[t])\subseteq \bigcap_{s=1}^r\bigcirc_{w=1}^s \left(t\dfrac{d}{dt}+\gamma_w\right)\circ[t-\alpha](K[t])\enspace.$$
Let us show the opposite inclusion. 
Let $P(t)\in K[t]$. Assume there exist polynomials $P_s(t)$ for $1\le s \le r$ such that 
\begin{align} \label{rep P}
P(t)=\bigcirc_{w=1}^s \left(t\dfrac{d}{dt}+\gamma_w\right)\circ[t-\alpha](P_s(t)) \ \ (1\le s \le r)\enspace.
\end{align}
It is sufficient to prove 
\begin{align} \label{divide}
P_s(t)\in (t-\alpha)^{s-1}K[t] \ \ (1\le s \le r)\enspace,
\end{align}
since above relation implies 
$$P(t)=\bigcirc_{w=1}^{r} \left(t\dfrac{d}{dt}+\gamma_w\right)\circ[t-\alpha](P_{r}(t))\in \bigcirc_{w=1}^{r} \left(t\dfrac{d}{dt}+\gamma_w\right)\circ[(t-\alpha)^{r}](K[t])\enspace.$$

Let us show Equation~\eqref{divide} by induction on $r$.
There is nothing to prove for $r=1$. 
Assume Equation~\eqref{divide} holds for $r\ge 1$.
Let us take $r+1$. Then the induction hypothesis for 
\begin{align*}
&P(t)=\bigcirc_{w=1}^s \left(t\dfrac{d}{dt}+\gamma_w\right)\circ[t-\alpha](P_s(t)) \ \ (1\le s \le r)\enspace,\\
&\left(t\dfrac{d}{dt}+\gamma_1\right)^{-1}(P(t))=\bigcirc_{w=2}^s \left(t\dfrac{d}{dt}+\gamma_w\right)\circ[t-\alpha](P_s(t)) \ \ (2\le s \le r+1)\enspace,
\end{align*}
assert that $P_{r}(t),P_{r+1}(t)\in (t-\alpha)^{r-1}K[t]$. Put $$P_{r}(t)=(t-\alpha)^{r-1}\tilde{P}_{r}(t) \ \ \text{and}  \ \ P_{r+1}(t)=(t-\alpha)^{r-1}\tilde{P}_{r+1}(t)\enspace.$$
Equation~\eqref{rep P} for $r+1$ implies that 
\begin{align*}
(t-\alpha)^{r}\tilde{P}_{r}(t)&=\left(t\dfrac{d}{dt}+\gamma_{r+1}\right)\circ [(t-\alpha)^{r}](\tilde{P}_{r+1}(t))\\
                                            &=rt(t-\alpha)^{r-1}\tilde{P}_{r+1}(t)+(t-\alpha)^r(t\tilde{P}'_{r+1}(t)+\gamma_r)\enspace.
\end{align*}
Since $\alpha\neq 0$, this allows us to get $\tilde{P}_{r+1}(t)$ is divisible by $(t-\alpha)$ and thus we get $P_{r+1}(t)\in (t-\alpha)^{r}K[t]$. This completes the proof of Equation~\eqref{divide}.
We complete the proof of the lemma.
\end{proof}

\begin{lemma} \label{ker psi}
Let $d-d'+1\le s \le d$ be an integer.Then we have 
$$
{\rm{ker}}\,\psi_{i,s}=\bigcirc_{w=d-s+1}^{d'} \left(t\dfrac{d}{dt}+\gamma_w\right)(t-\alpha_i)K[t] \enspace.
$$
\end{lemma}
\begin{proof}
It is easy to see that 
\begin{align}
\bigcirc_{w=d-s+1}^{d'} \left(t\dfrac{d}{dt}+\gamma_w\right)(t-\alpha_i)K[t] \subset {\rm{ker}}\,\psi_{i,s}\enspace.
\end{align} 
Let us take $P(t)\in {\rm{ker}}\,\psi_{i,s}$.
Since $\gamma_w\notin \Z_{\le 0}$, we notice $\left(t\tfrac{d}{dt}+\gamma_w\right)\in {\rm{Aut}}_K(K[t])$. This shows there exists a polynomial $\tilde{P}(t)$ with 
$$P(t)=\bigcirc_{w=d-s+1}^{d'} \left(t\dfrac{d}{dt}+\gamma_w\right)(\tilde{P}(t))\enspace.$$ By definition of $\psi_{i,s}$, we have 
$$0=\psi_{i,s}(P(t))={\rm{Eval}}_{\alpha_i}(\tilde{P}(t))\enspace.$$ 
This implies $\tilde{P}(t)\in (t-\alpha_i)K[t]$ and thus, we obtain the desire equality. 
\end{proof}
\begin{corollary} \label{key1} 
The following equalities hold.
$$\bigcap_{i=1}^m\bigcap_{s=d-d'+1}^d {\rm{ker}}\,\psi_{i,s}= \bigcirc_{w=1}^{d'} \left(t\dfrac{d}{dt}+\gamma_w\right)\circ \left[\prod_{i=1}^m(t-\alpha_i)^{d'}\right](K[t])\enspace.$$
\end{corollary}
\begin{proof}
By combining Lemma~$\ref{basis}$ $({\rm{i}})$ and Lemma~$\ref{ker psi}$ $({\rm{i}})$, it is sufficient to prove 
{\small{
$$\bigcap_{i=1}^m  \bigcirc_{w=1}^{d'}\left(t\dfrac{d}{dt}+\gamma_w\right)\circ[(t-\alpha_i)^{d'}](K[t])=\bigcirc_{w=1}^{d'}\left(t\dfrac{d}{dt}+\gamma_w\right)\circ\left[\prod_{i=1}^m(t-\alpha_i)^{d'}\right](K[t])\enspace.$$
}}
By definition we see that the right hand side is contained in the left hand side. 
Let $P\in \bigcap_{i=1}^m  \bigcirc_{w=1}^{d'}\left(t\tfrac{d}{dt}+\gamma_w\right)\circ[(t-\alpha_i)^{d'}](K[t])$.
Then there exist polynomials $P_{i}\in K[t]$ for $1\le i \le m$ such that 
$$P= \bigcirc_{w=1}^{d'}\left(t\dfrac{d}{dt}+\gamma_w\right)((t-\alpha_i)^{d'}P_i(t))\enspace.$$
Since $\bigcirc_{w=1}^{d'}\left(t\dfrac{d}{dt}+\gamma_w\right)\in {\rm{Aut}}_K(K[t])$, the above equality implies 
$$(t-\alpha_1)^{d'}P_1(t)=\cdots =(t-\alpha_m)^{d'}P_m(t)\enspace,$$ and thus 
$P_i(t)\in \prod_{j\neq i}(t-\alpha_j)^{d'}K[t]$ for $1\le i \le m$. This leads us to get $$P\in \bigcirc_{w=1}^{d'}\left(t\dfrac{d}{dt}+\gamma_w\right)\circ\left[\prod_{i=1}^m(t-\alpha_i)^{d'}\right](K[t])\enspace.$$
\end{proof}

\begin{lemma} \label{key2}
Let $P(t),Q(t)\in K[t]$ and $r,n\in \N$. Let $\alpha,\gamma_1,\ldots,\gamma_{r}\in K\setminus \{0\}$ and $\beta\in \{0,\alpha\}$.
Assume $${Q}(t)(t-\beta)^{n}=\bigcirc_{w=1}^{r} \left(t\dfrac{d}{dt}+\gamma_w\right)^{r}\circ[(t-\alpha)^{r}](P(t))\enspace.$$
Then $P(t)\in (t-\beta)^{n}K[t]$.
\end{lemma}
\begin{proof}
Let us prove the lemma by induction on $n$. Firstly we consider the case $\beta=0$.
Let $n=1$. Then, since $t\tfrac{d}{dt}(P(t))\in tK[t]$, we have 
$${Q}(t)t=\bigcirc_{w=1}^{r} \left(t\dfrac{d}{dt}+\gamma_w\right)^{r}\circ[(t-\alpha)^{r}](P(t))\in \prod_{w=1}^r\gamma_w (t-\alpha)^rP(t)+tK[t]\enspace.$$
Since $\alpha\in K\setminus\{0\}$, the above equality leads us to get $P(t)\in tK[t]$.
Assume the statement holds for $n\ge 1$. Let us take $n+1$. The induction hypothesis implies $P(t)\in t^nK[t]$. Put $P(t)=t^n\tilde{P}(t)$. This allows to show 
{\small{
$$Q(t)t^{n+1}=\bigcirc_{w=1}^{r} \left(t\dfrac{d}{dt}+\gamma_w\right)^{r}\circ[(t-\alpha)^{r}](P(t))\in \prod_{w=1}^r\gamma_w(t-\alpha)^nt^n\tilde{P}(t)+t^{n+1}K[t]\enspace.$$
}}
Since $\alpha\in K\setminus\{0\}$, the above equality leads us to get $\tilde{P}(t)\in tK[t]$ and therefore $P(t)\in t^{n+1}K[t]$.

\medskip

Next we consider the case $\beta=\alpha$. Let $n=1$. A straightforward computation yields
$${Q}(t)(t-\alpha)=\bigcirc_{w=1}^{r} \left(t\dfrac{d}{dt}+\gamma_w\right)^{r}\circ[(t-\alpha)^{r}](P(t))\in r!t^rP(t)+(t-\alpha)K[t]\enspace.$$
This allows us to get $P(t)\in (t-\alpha)K[t]$. Assume the statement holds for $n\ge 1$. Let us take $n+1$. The induction hypothesis implies $P(t)\in (t-\alpha)^nK[t]$. 
Put $P(t)=(t-\alpha)^n\tilde{P}(t)$.
We thus obtain
\begin{align*}
Q(t)(t-\alpha)^{n+1}&=\bigcirc_{w=1}^{r} \left(t\dfrac{d}{dt}+\gamma_w\right)^{r}\circ[(t-\alpha)^{r+n}](\tilde{P}(t))\\
                                 &\in (n+1)_rt^r(t-\alpha)^n\tilde{P}(t)+(t-\alpha)^{n+1}K[t]\enspace.
\end{align*}
The above equality implies $\tilde{P}(t)\in (t-\alpha)K[t]$ and thus $P(t)\in (t-\alpha)^{n+1}K[t]$. This completes the proof of the lemma. 
\end{proof}

Next, we consider the kernel of $\varphi_{i,s}$ for $1 \leq s \leq d - d'$.
\begin{lemma} \label{P is contained in power of t-a}
Let $N, r$ be positive integers and $\alpha \in K \setminus \{0\}$. Then we have
\[
\bigcap_{s=1}^r \ker\, {\rm{Eval}}_{\alpha} \circ \left(t \dfrac{d}{dt}\right)^{N+s}\cap (t - \alpha)^N K[t] \subseteq (t - \alpha)^{N+r} K[t] \enspace.
\]
\end{lemma}
\begin{proof}
It suffices to show that
\begin{align} \label{P=0}
P(t) = \sum_{j=0}^{r-1} p_j (t - \alpha)^{N + j} \in \bigcap_{s=1}^r \ker\, {\rm{Eval}}_{\alpha} \circ \left(t \dfrac{d}{dt}\right)^{N + s - 1}\Rightarrow P(t) = 0 \enspace.
\end{align}
We prove $p_j = 0$ by induction on $j$. Applying ${\rm{Eval}}_{\alpha} \circ \left(t \tfrac{d}{dt}\right)^N$ to $P(t)$ and using the Leibniz rule, we obtain
\[
{\rm{Eval}}_{\alpha} \circ \left(t \dfrac{d}{dt}\right)^N(P) = p_0 \alpha^N N! = 0\enspace,
\]
so $p_0 = 0$.  
Now fix $1 \leq j < r-1$ and assume $p_0 = p_1 = \cdots = p_j = 0$. Then applying ${\rm{Eval}}_{\alpha} \circ \left(t \tfrac{d}{dt}\right)^{N + j + 1}$ to $P(t)$ yields
\[
{\rm{Eval}}_{\alpha} \circ \left(t \dfrac{d}{dt}\right)^{N + j + 1}(P) = p_{j+1} \alpha^{N + j + 1} (N + j + 1)! = 0\enspace,
\]
so $p_{j+1} = 0$. This completes the induction and proves \eqref{P=0}.
\end{proof}
\begin{corollary} \label{divisible by power of t-a}
Let $P(t) \in \bigcap_{i=1}^m \bigcap_{s=1}^{d - d'} \ker \psi_{i,s}$. Assume $P(t)$ is divisible by $\prod_{i=1}^m (t - \alpha_i)^{dn}$. Then we have
\[
P(t) \in \prod_{i=1}^m (t - \alpha_i)^{dn + d - d'} K[t]\enspace.
\]
\end{corollary}
\begin{proof}
From \eqref{psi}, we have
\[
P \in \bigcap_{s=1}^{d - d'} \ker\, {\rm{Eval}}_{\alpha_i} \circ \left(t \dfrac{d}{dt}\right)^{dn + s - 1}
\]
for each $1 \leq i \leq m$. By Lemma~\ref{P is contained in power of t-a}, it follows that
\[
P \in (t - \alpha_i)^{dn + d - d'} K[t] \quad \text{for all } i = 1, \dotsc, m\enspace.
\]
Since the $\alpha_i$ are pairwise distinct, we conclude the assertion.
\end{proof}

\begin{proof}[{\textbf{Proof of Proposition~$\ref{Q=0}$}}]
We recall the polynomial $Q(t) \in \bigcap_{i=1}^m \bigcap_{s=1}^d \ker\,\psi_{i,s}$ (see Equation~\eqref{poly Q}).
By definition, $Q(t)$ is divisible by $t^n$.
Applying Corollary~\ref{divisible by power of t-a} to $Q(t) \in \bigcap_{i=1}^m \bigcap_{s=1}^{d - d'} \ker\,\psi_{i,s}$, we deduce that $Q(t)$ is divisible by
\[
t^n \prod_{i=1}^m (t - \alpha_i)^{dn + d - d'}\enspace.
\]
Corollary~\ref{key1}, combined with the above divisibility, implies that
{\small{\begin{align} \label{relation}
Q(t) \in t^n \prod_{i=1}^m (t - \alpha_i)^{dn + d - d'} K[t] \cap 
\bigcirc_{w=1}^{d'} \left(t \dfrac{d}{dt} + \gamma_w\right) \circ \left[\prod_{i=1}^m (t - \alpha_i)^{d'}\right](K[t])\enspace.
\end{align}}}
From \eqref{relation}, there exist polynomials $P(t), \widetilde{Q}(t) \in K[t]$ such that
{\small{\[
Q(t) = t^n \prod_{i=1}^m (t - \alpha_i)^{dn + d - d'} \widetilde{Q}(t) = 
\bigcirc_{w=1}^{d'} \left(t \dfrac{d}{dt} + \gamma_w\right) \circ 
\left[\prod_{i=1}^m (t - \alpha_i)^{d'}\right](P(t))\enspace.
\]}}
Applying Lemma~\ref{key2} repeatedly with $r = d'$ and using the fact that the $(t - \alpha_i)$ are pairwise coprime (since the $\alpha_i$ are distinct), we deduce that
\[
P(t) \in t^n \prod_{i=1}^m (t - \alpha_i)^{dn + d - d'} K[t]\enspace.
\]
Hence, if $Q\neq 0$, the degree of $Q$ satisfies
\[
\deg Q \ge n + (dn + d - d')m + d' m = dm(n + 1) + n\enspace.
\]
On the other hand, by the definition of $Q$ (see Equation~\eqref{poly Q}), we have
\[
\deg Q \le n + (dm - 1) + dmn = dm(n + 1) + n - 1\enspace.
\]
This contradiction implies $Q(t) = 0$.
\end{proof}

We now finish the proof of Proposition~$\ref{non zero det}$.
\begin{proof}[\textbf{Proof of Proposition~$\ref{non zero det}$}]
Combining Lemmas $\ref{sufficient condition}$ and $\ref{simplify}$ yields that there exist $c\in K\setminus\{0\}$ and $a_{0,s}\in K$ for $d-d'+1\le s \le d$ such that  
\begin{align*}
\Delta(z)=c \cdot \dfrac{\prod_{i=1}^m\alpha^d_i{\prod_{s=d-d'+1}^{d}}a_{s,0}^m}{(n-1)!^{d^2m}}\cdot {\rm{det}}\, \mathcal{M}_n\enspace.
\end{align*}
Corollary $\ref{a0s 0}$ ensures that the non-vanishing of the term except for ${\rm{det}}\,\mathcal{M}_n$.
Finally, since the $\alpha_i$ are pairwise distinct, Proposition~$\ref{Q=0}$ ensures ${\rm{det}}\,\mathcal{M}_n\neq 0$.
\end{proof}

{{\section{Estimates} \label{estimate}
{We keep the notations of Section~\ref{nonvan}. {We further assume that $K$ is a number field and $\eta_1,\ldots,\eta_{{d''}},\zeta_1,\ldots,\zeta_{d'}$ be rational numbers which are not negative integers with $\eta_i-\zeta_j\notin \N$ for $1 \le i \le d'', 1\le j \le d'$.
Assume $$A(X)=(X+\eta_1)\cdots (X+\eta_{{d''}}), \ \ \ B(X)=(X+\zeta_1)\cdots (X+\zeta_{d'})$$
where $d'd''>0$. 
Let $\boldsymbol{c}=(c_k)_{k\ge0}$ be {the} sequence satisfying $c_0=1$ and \eqref{recurrence 1} for the given polynomials $A(X)$ and $B(X)$. 
Let $\boldsymbol{\gamma}=(\gamma_1,\ldots,\gamma_{d-1})\in K^{d-1}$. Unless otherwise stated, let us treat $\alpha_1,\ldots,\alpha_m$ {as variables.
For a non-negative integer $\ell$ with $0\le \ell \le dm$, recall the polynomials $P_{\ell}(z), P_{\ell,i,s}(z)$ defined in Proposition~\ref{GHG pade} for the given data.

We start this section with elementary considerations, useful during intermediate estimates for the norm of our auxiliary polynomials.
\begin{lemma} \label{elementary} 
Let $k$ be a positive integer and $\eta,\zeta$ be strictly positive rational numbers.

\medskip

$({\rm{i}})$ One has
$${\frac{1}{(\eta)_k}\leq \dfrac{(\lfloor \eta\rfloor+k)\lfloor\eta\rfloor!}{\eta\cdot (\lfloor\eta\rfloor+k)!}\enspace.}$$

$({\rm{ii}})$ One has
$$k!\le (1+\zeta)_k \le \dfrac{(\lceil \zeta \rceil + k)!}{\lceil \zeta \rceil !} \enspace.$$

$({\rm{iii}})$ For any positive integer $a$, one has
$$\frac{1}{(a+k)!}\leq \frac{1}{k!k^a}\enspace.$$
\end{lemma}

\begin{lemma} \label{denominator}
Let $n,k$ be positive integers and $\eta,\zeta$ be non-zero rational numbers. Recall $\mu_n(\zeta)={\rm{den}}(\zeta)^n\cdot \prod_{\substack{q:\rm{prime} \\ q\mid {\rm{den}}(\zeta_j)}}q^{\lfloor \tfrac{n}{q-1} \rfloor}$, $\mu(\zeta)=\den(\zeta)\prod_{\substack{q:\rm{prime} \\ q\mid {\rm{den}}(\zeta_j)}}q^{\frac{1}{q-1}}$.

\medskip

$({\rm{i}})$ 
One has
$$\mu_n(\zeta)\cdot \dfrac{(\zeta)_k}{k!}\in \Z \ \ \ \text{for} \ \ \ 0\le k \le n\enspace.$$

$({\rm{ii}})$ One has 
\begin{align*} 
\mu_n(\zeta)=\mu_n(\zeta+k) \ \ \ \text{and} \ \ \ \mu_{n+k}(\zeta) \ \ \text{is divisible by} \ \ \mu_{n}(\zeta)\mu_{k}(\zeta)\enspace.
\end{align*}

$({\rm{iii}})$ 
For a non-negative integer $n$, put $$D_n={\rm{den}}\left(\dfrac{(\eta)_0}{(\zeta)_0},\ldots,\dfrac{(\eta)_n}{(\zeta)_n}\right)\enspace.$$
One has
$$\limsup_{n\to \infty}\dfrac{1}{n}\log\,D_n\le \log\,\mu(\eta)+{{{\den}(\zeta)}}
\footnote{Using Dirichlet's prime number theorem on arithmetic progression (see \cite{Betal}), we may improve the upper bound as 
$$\limsup_{n\to \infty}\dfrac{1}{n}\log\,D_n\le \log\,\mu(\eta)+\dfrac{{\den}(\zeta)}{\varphi({\rm{den}}(\zeta))}\sum_{\substack{j=1 \\ (j,{\rm{den}}(\zeta))=1}}^{{\rm{den}(\zeta)}}\dfrac{1}{j}\enspace,$$ where $\varphi$ denotes Euler's totient function.}
\enspace.$$

\medskip

$({\rm{iv}})$ Put $\zeta=c/d$ where $c,d$ are coprime integers with $d>0$. Put $N_n=c(c+d)\cdots (c+d(n-1))$. Let $p$ be a prime number with $p \mid N_n$.
One has
$$1\le \left|\dfrac{n!}{(\zeta)_n}\right|_p \le |c|+d(n-1)\enspace.$$
\end{lemma}
\begin{proof}
$({\rm{i}})$ This property is proven in \cite[Lemma~$2.2$]{B}.

\medskip

$({\rm{ii}})$ We directly obtain the assertion by the definition of $\mu_n(\zeta)$. 

\medskip

$({\rm{iii}})$ Put
$$D_{1,n}={\rm{den}}\left(\dfrac{(\eta)_0}{0!},\ldots,\dfrac{(\eta)_n}{n!}\right), \ \ \ D_{2,n}={\rm{den}}\left(\dfrac{0!}{(\zeta)_0},\ldots,\dfrac{n!}{(\zeta)_n}\right)\enspace.$$
Since inequality $D_n\le D_{1,n}D_{2,n}$ holds, the assertion is deduced from
$$\limsup_{n\to \infty}\, \dfrac{1}{n}\log\,D_{1,n}\le \log\, \mu(\eta), \ \ \ \limsup_{n\to \infty}\,\dfrac{1}{n}\log\,D_{2,n}\le {{\rm{den}(\zeta)}} 
\enspace.$$

The first inequality is a consequence of $({\rm{i}})$.
Second inequality is shown in \cite[Lemma~$4.1$]{KP}, however, we explain here this proof in an abbreviated form, to let our article be self-contained. 
This proof is originally indicated by Siegel \cite[p.57,58]{Siegel}.
Put $d={\rm{den}}(\zeta)$, $c=d\cdot \zeta$.
    We set $N_k=c(c+d)\cdots(c+(k-1)d)$.
    Let $p$ be a prime number with $p\mid N_k$. The following three properties hold.

    \medskip

    $({\rm{a}})$  We have ${\rm{GCD}}(p,d)=1$. For any integers $i, \ell$ with $\ell>0$, there exists exactly one integer $\nu$ with $0\le \nu \le p^{\ell}-1$ and such that $p^{\ell}\mid c+(i+\nu)d$.

    \medskip

    $({\rm{b}})$ Let $\ell$ be a strictly positive integer with $|c|+(k-1)d<p^{\ell}$. Then, $c+id$ for $0\le i \le k-1$ is not divisible by $p^{\ell}$.

    \medskip

    $({\rm{c}})$ Set $C_{p,k}=\lfloor \log(|c|+(k-1)d)/\log(p)\rfloor $. Then,
    \[
        v_p(k!)=\sum_{\ell=1}^{C_{p,k}}\left \lfloor \dfrac{k}{p^{\ell}}  \right\rfloor\le v_p(N_k)\le \sum_{\ell=1}^{C_{p,k}}\left(1+\left \lfloor \dfrac{k}{p^{\ell}} \right\rfloor  \right)=v_p(k!)+C_{p,k}\enspace,
    \]
    where $v_p$ denotes the $p$-adic valuation. This allows us 
    \begin{align} \label{D2n}
        \log\,D_{2,n}&=\sum_{p \mid N_n}\max_{0\le k \le n} \log\,\left|\dfrac{k!}{(\zeta)_k}\right|_p 
\le \log(|c|+d(n-1))\sum_{{{p\le c+d(n-1)}}}1\enspace.
    \end{align}
    Denote  $\pi(x)=\#\{p:\text{prime} \mid \ p<x\}$ for $x>0$. Then by prime number theorem
    \[
        \limsup_{n\to \infty}  \dfrac{\log(|c|+d(n-1))\pi(c+d(n-1))}{n}=d\enspace,
    \]
    and we deduce desire inequality ({\it confer} \cite{RS}).

\medskip

$({\rm{iv}})$ We keep the notation in the proof of $({\rm{iii}})$. The assertion $(c)$ in the proof of $({\rm{iii}})$ yields $$1\le \left|\dfrac{n!}{N_n}\right|_p \le p^{C_{p,n}}\enspace,$$ as claimed.
\end{proof}

Throughout the section, the small $o$-symbol $o(1)$ and $o(n)$ and the large $O$-symbol $O(n)$ refer when $n$ tends to infinity. 
Put $\varepsilon_v=1$ if $ v\mid \infty$ and $0$ otherwise.

Let $I$ be a non-empty finite set of indices, $R=K[X_i]_{i\in I}$ be the polynomial ring over $K$  in indeterminate $X_i$. 
We set $\|P\|_v=\max\{\vert a\vert_v\}$ where $a$ runs in the coefficients of $P$ for any place $v\in\mathfrak{M}_K$.
The degree of an element of ${{R}}$ is as usual the total degree. 

\medskip

{{We begin by estimating Pad\'{e} approximants, but the method differs from the previous ones in \cite[Lemma~$5.4$]{DHK3} and \cite[Lemma~$5.2$]{DHK4}. 
The previous method involved estimating  the norm of the operator using submultiplicativity, while this time we are estimating the norm taking advantage of the fact that most of the operators involved in the construction of $P_{\ell}$, $P_{\ell,i,s}, R_{\ell, i, s}$ and related polynomials are defined via linear operators acting diagonally (see remark 4 of \cite{DHK3}). Though not necessary for $P_{\ell}$, this becomes necessary for the others.

Recall our kernel polynomial {(defined in Proposition~\ref{GHG pade})}
$$H_{\ell}(t)=t^{\ell}\prod_{i=1}^m(t-\alpha_i)^{dn}\enspace,$$
and
$$P_{\ell}(t)={\mathcal{T}}_{\bold{c}}\bigcirc_{j=1}^{d'} S_{n-1,\zeta_j}\left(H_{\ell}(t)\right)\enspace,$$
where $S_{n-1,\zeta_j}=\frac{1}{(n-1)!}\prod_{l=1}^{n-1}\left(t\frac{d}{dt}+\zeta_j+l\right)$.

Also recall that by definition, $\mathcal{T}_{\bold{c}}$ and $S_{n-1,\zeta_j}$ both act diagonally on the standard basis of $K[t]$
with eigenvalues respectively 
$$\lambda_{\mathcal T_{\bold{c}}}(k)=\frac{\prod_{j=1}^{d'}(\zeta_j+1)_k}{\prod_{j=1}^{d''}(\eta_j)_k}\enspace,\kern20pt\lambda_{S_{n-1},\zeta_j}(k)=\frac{(k+\zeta_j+1)_{n-1}}{(n-1)!} \enspace.$$

\begin{definition}\label{consterreur} An admissible error term is a sequence $(c_{{n},v})_{v\in \mathfrak{M}_K, {n}\in\enn}$ of positive real numbers such that:
$$\frac{1}{n} \sum_{v\in \mathfrak{M}_K}\log(c_{{n},v})\raisebox{-4pt}{$\stackrel{\displaystyle\longrightarrow}{\scriptscriptstyle n \longmapsto\infty} $}0\enspace.$$
 
The product of two admissible error terms $(c_{{n},v})\cdot (c'_{{n},v})$ is done term by term and is equal to $(c_{{n},v}\cdot c'_{{n},v})_{{n},v}$. The product of finitely many admissible error terms is still admissible, and if $c_{n,v}=1$ for all but finitely many $v\in \mathfrak{M}_K$ {\it and} is a $o(n)$ for those places where it is not $=1$, it is obviously admissible.
\end{definition}

We now define the constant 
\begin{align} \label{Cmr}
B(d,d',m)=B=dm \log(2)+d'\left(\log (dm+1)+dm \log\left(\frac{dm+1}{dm}\right)\right)\enspace,
\end{align} which depends on $d,d'$ and $m$.

\begin{lemma}\label{normepl} 
There is an admissible effectively computable error term  $(c_{n,v}=c_{n,v}(\boldsymbol{\zeta},\boldsymbol{\eta}))_{v\in \mathfrak{M}_K,n\in \N}$ depending only on the given data $\boldsymbol{\zeta},\boldsymbol{\eta}$ $($also depends implicitly on the parameter $m)$ such that 
\begin{itemize}
\item[$({\rm{i}})$] Assume $v$ is Archimedean, then $$\max\{\Vert P_{\ell}\Vert_v,\Vert P_{\ell,i,s}\Vert_v\}\leq c_{n,v}
\exp(B\tfrac{[K_v:\R]}{[K:\Q]}n) \vert (dmn)!\vert_v^{\max\{0,d'-d''\}}\enspace .$$

\item[$({\rm{ii}})$] Assume $v$ is ultrametric, then, setting

$$
D_{k,l,n}(\boldsymbol{\zeta},\boldsymbol{\eta})= {\dfrac{\prod_{j=1}^{d'}(\zeta_j+l+1)_{k-l+n-1}}{(n-1)!^{d'}\prod_{j=1}^{d''}(\eta_j+l)_{k-l}}}
\enspace,
$$
and 
$$
A_{n,v}(\boldsymbol{\zeta},\boldsymbol{\eta})=\max\{\left \vert D_{k,l,n}(\boldsymbol{\zeta},\boldsymbol{\eta})\right\vert_v;\kern5pt 0\leq {l\leq k-1, 0\leq} k \leq dmn+dm\}\enspace.
$$
One has$:$
$$\max\{\Vert P_{\ell}\Vert_v,\Vert P_{\ell,i,s}\Vert_v\}\leq {\erreurv} A_{n,v}(\boldsymbol{\zeta},\boldsymbol{\eta})\enspace.$$ 
\end{itemize}
\end{lemma}
\begin{proof} 
We start with the case of $P_{\ell}$.
Now, by the observation above, the eigenvalue of 
$\mathcal{T}_{\mathbf{c}} \circ \bigcirc_{j=1}^{d'} S_{n-1,\zeta_j}$
corresponding to the monomial $t^k$ is:
$$\frac{\prod_{j=1}^{d'}(\zeta_j+1)_{k+n-1}}{(n-1)!^{d'}\prod_{j=1}^{d''}(\eta_j)_k}\enspace.$$

Hence, $$\Vert P_{\ell}\Vert_v\leq  \max_{0\le k\le dmn+l}\left\vert\frac{\prod_{j=1}^{d'}(\zeta_j+1)_{k+n-1}}{(n-1)!^{d'}\prod_{j=1}^{d''}(\eta_j)_k}\right\vert_v\Vert H_{\ell}\Vert_v\enspace.$$

Since for finite places, $\Vert H_{\ell}\Vert_v\leq 1$, (ii) is proven for $P_{\ell}$, with $c_{n,v}=1$. We now turn to the numerators of the Pad\'e approximation system {\it i.\kern3pte.} the polynomials $P_{\ell,i,s}$.

By the definition, we have $P_{\ell,i,s}(z)=\varphi_{i,s}\left(\tfrac{P_{\ell}(z)-P_{\ell}(t)}{z-t}\right)$ and $$\varphi_{i,s}=[\alpha_i]\circ {\rm{Eval}}_{t=\alpha_i}\circ \mathcal{T}^{-1}_{\boldsymbol{c}} \bigcirc_{u=1}^{d-s}(t\tfrac{d}{dt}+\gamma_{u})\enspace.$$
 Setting $\Gamma(Q)=\frac{Q(z)-Q(t)}{z-t}$, we have 
 $$P_{\ell,i,s}(z)=[\alpha_i]\circ {\rm{Eval}}_{t=\alpha_i}\circ \mathcal{T}^{-1}_{\boldsymbol{c}} \bigcirc_{u=1}^{d-s}(t\tfrac{d}{dt}+\gamma_{u})\circ \Gamma(P_{\ell})\enspace.$$
 Note that the operators $[\alpha_i]$, $\ev_{t=\alpha_i}$ are isometries (right shift, substitution of variables). 
 As for the operator $\Gamma$, it takes the monomial $t^{\delta}$ to $\sum_{j=0}^{\delta-1} t^{\delta-j-1}z^j$ and is of norm $1$ (\cite{DHK3}, Lemma~$5.2~(iii)$).
 The action of the operators  on the monomial $t^k$
 $$\mathcal{T}^{-1}_{\boldsymbol{c}} \bigcirc_{u=1}^{d-s}(t\tfrac{d}{dt}+\gamma_{u})\Gamma\circ\mathcal{T}_{\bold{c}}\circ\bigcirc_{j=1}^{d'}S_{n-1,\zeta_j}$$ is thus given by their eigenvalue
 $$\prod_{j=1}^{d'}\frac{(k+\zeta_j+1)_{n-1}}{(n-1)!}\lambda_{\mathcal{T}_{\boldsymbol{c}}}(k)\prod_{u=1}^{d-s}({l}+\gamma_u)\lambda^{-1}_{\mathcal{T}_{\boldsymbol{c}}}(l)$$ where $l$ varies between $0$ and $k-1$. Without prejudice, one can multiply 
$c_{n,v}$
by $$\max\Biggl\{\left\vert\prod_{u=1}^{d-s}(l+\gamma_u)\right\vert_v; \kern3pt 0\leq {l\leq dmn+dm-1}\Biggr\}$$ which is admissible since $\gamma_u\in K$ are fixed and the height of $l$ grows linearly with $\log n$.

We are thus left with estimating:
$$\frac{\prod_{j=1}^{d'}(\zeta_j+1)_{k+n-1}}{(n-1)!^{d'}\prod_{j=1}^{d''}(\eta_j)_k}\frac{\prod_{j=1}^{d''}(\eta_j)_l}{\prod_{j=1}^{d'}(\zeta_j+1)_l}
=\frac{\prod_{j=1}^{d'}(\zeta_j+l+1)_{k-l+n-1}}{(n-1)!^{d'}\prod_{j=1}^{d''}(\eta_j+l)_{k-l}}\enspace.
$$

{By definition of $D_{k,l,n}(\boldsymbol{\zeta},\boldsymbol{\eta})$ and taking $c_{n,v}=\prod_{w=1}^d\max\{1,\left|\gamma_{w}\right|_v\}$ to account for the neglected  factors, (ii) is now completely proven.

For now on, we assume $v$ is archimedean and  temporarily assume that $\eta_j+l,\zeta_j+l$ are all $>0$. 

By Lemma~\ref{elementary} (i),
$$\frac{1}{(\eta+l)_{k-l}}\leq   \frac{(\lfloor\eta\rfloor+k)(\lfloor\eta\rfloor+l)!}{(\eta+l)(\lfloor\eta\rfloor+k)!}\enspace$$ and up to an admissible error term (polynomial in $n$), the right hand side is equal to $\frac{l!}{k!}$.
Similarly, using this time Lemma~\ref{elementary} (ii),
$$\frac{(\zeta+l+1)_{k-l+n-1}}{(n-1)!}\leq  \frac{(\lceil\zeta\rceil+k+n-1)!(\lceil \zeta\rceil+k)!}{(n-1)!(\lceil \zeta\rceil+k)!(\lceil\zeta\rceil+l)!}$$
and the right hand side is up to an admissible error term $\binom{k+n}{n}\frac{k!}{l!}$.

Putting together and taking the product over $\zeta_j,\eta_j$ one deduces that
$$\Theta(k,l)=\frac{\prod_{j=1}^{d'}(\zeta_j+l+1)_{k-l+n-1}}{(n-1)!^{d'}\prod_{j=1}^{d''}(\eta_j+l)_{k-l}}\leq  \exp(o(n))\binom{k+n}{n}^{d'}\frac{k!^{d'-d''}}{l!^{d'-d''}}$$
When $d'\geq d''$, the right hand side is maximal for $l=0,k=dmn+dm$, whence if $d''\geq d'$ it is maximal for $l=k-1$ and $k=dmn+dm$.

Taking into account  (since $k\leq dmn+dm$)
$$\log\binom{ k+n}{n}\leq n\left(\log (dm+1)+dm\log \left(\frac{dm+1}{dm}\right)\right)+o(n)\enspace, $$
one gets
\begin{align*}
\log\Theta(k,l) \leq &  \max\{0,d'-d''\}\log(dmn)! \\
&+ nd'\left(\log (dm+1)+dm\log \left(\frac{dm+1}{dm}\right)\right)+o(n)\enspace.
\end{align*}
To finish the proof, we turn to the case where all the $\zeta_j,\eta_j$ might not be strictly positive. Let $k^*=\max\{-\lfloor\zeta_j\rfloor,-\lfloor\eta_j\rfloor\}$. If $k>k^*$, note that 
$$(\zeta+l+1)_{k-l+n-1}=\left((\zeta+l)\ldots(\zeta+l+k^*-1)\right)(\zeta+l+k^*+1)_{k-l-k^*+n-1}$$
and
$$(\eta+l)_{k-l}=\left((\eta+l)\ldots(\eta+l+k^*-1)\right)(\eta+l+k^*)_{k-l-k^*}\enspace.$$
So, upto  finitely many factors (depending only on $\boldsymbol{\zeta},\boldsymbol{\eta}$), the above bound is valid, the maximum of these finitely many factors being admissible. 

Finally, using Lemma~$5.2$~$({\rm{iv}})$ of \cite{DHK3}, one gets that $$\Vert H_{\ell}\Vert_v\leq \left(2^{dmn+dm}(dmn+dm)^{m}\right)^{[K_v:\R]/[K:\qu]}\enspace.$$

Therefore, one gets
$$\max\{\Vert P_{\ell}\Vert_v,\Vert P_{\ell,i,s}\Vert_v\}\leq c_{n,v}
\exp(B\tfrac{[K_v:\R]}{[K:\Q]}n) \vert (dmn)!\vert_v^{\max\{0,d'-d''\}}\enspace ,$$
and part (i) of the lemma is also proven.

We start with $({\rm{i}})$ for the case $P_{\ell}$.
Using Lemma~$5.2$~$({\rm{iv}})$ of \cite{DHK3}, one gets that $$\Vert H_{\ell}\Vert_v\leq \left(2^{dmn+dm}(dmn+dm)^{m}\right)^{[K_v:\R]/[K:\qu]}\enspace.$$
Now, the $k$th eigenvalue of $\mathcal{T}_{\bold{c}}\circ\bigcirc_{j=1}^{d'}S_{n-1,\zeta_j}$ is by the observation above:
\begin{align*}
\frac{\prod_{j=1}^{d'}(\zeta_j+1)_{k+n-1}}{(n-1)!^{d'}\prod_{j=1}^{d''}(\eta_j)_k}
=&\prod_{j=1}^{d'}\frac{(\lceil |\zeta_j| \rceil+k+n-1)!(\zeta_j+1)_{k+n-1}(\lceil |\zeta_j| \rceil+k)!}{(\lceil |\zeta_j| \rceil+k+n-1)!(n-1)!(\lceil |\zeta_j| \rceil+k)!}\\
&\prod_{j=1}^{d''}\frac{(\lceil |\eta_j| \rceil+k)!}{(\lceil |\eta_j| \rceil+k)!(\eta_j)_k}
\enspace.
\end{align*}
We now make use of the fact that $v$ is Archimedean and get using the previous Lemma~\ref{elementary}, $({\rm{i}})$ and then $({\rm{iii}})$:
\begin{align*}
\frac{(\lceil |\eta| \rceil+k)!}{(\lceil |\eta| \rceil+k)!(\eta)_k}&\leq \frac{k}{|\eta|}\binom{k+\lceil |\eta| \rceil}{\lceil |\eta| \rceil} \cdot \frac{1}{(\lceil |\eta| \rceil+k)!}\\
&\leq \frac{k}{|\eta|}\binom{k+\lceil |\eta| \rceil}{k} \cdot \frac{1}{k!\max\{1,k\}^{\lceil |\eta| \rceil}}
\end{align*}
and {by Lemma~\ref{elementary} (ii)}
{\small{\begin{align*}
\frac{(\lceil |\zeta| \rceil+k+n-1)!(\zeta+1)_{k+n-1}(\lceil |\zeta| \rceil+k)!}{(\lceil |\zeta| \rceil+k+n-1)!(n-1)!(\lceil |\zeta| \rceil+k)!}
& \leq  \displaystyle \binom{\lceil |\zeta| \rceil + k+n-1}{n-1} \frac{(\lceil |\zeta|\rceil +k)!}{\lceil |\zeta| \rceil!}\\ 
& =  \binom{\lceil |\zeta| \rceil + k+n-1}{n-1} k! \binom{\lceil |\zeta| \rceil+k}{\lceil |\zeta| \rceil}\enspace.
\end{align*}
}}

{Putting together and recalling that $0\leq k\leq dmn+dm$, one gets and taking into account the fact
that the absolute value $\vert\cdot\vert_v$ is normalized so that it coincides with the power $\vert\cdot\vert^{[K_v:\R]/[K:\Q]}$ of the usual absolute value in $\qu$
\begin{align*}
\Vert \mathcal{T}_{\bold{c}}\circ\bigcirc_{j=1}^{d'}S_{n-1,\zeta_j}\Vert_v&= \max_{0\leq k\leq dmn+dm} \biggl\{\vert\lambda_{\mathcal T_{\bold{c}}}(k)\vert_v\prod_{j=1}^{d'}\vert\lambda_{S_{n-1},\zeta_j}(k)\vert_v\biggr\}\enspace,
\end{align*}
and so
$$\begin{array}{lcl}\displaystyle
\Vert P_{\ell}\Vert_v& \leq& \displaystyle \max_{0\leq k\leq dmn+dm}\left\{{{\left|\frac{\prod_{j=1}^{d'}(\zeta_j+1)_{k+n-1}}{(n-1)!^{d'}\prod_{j=1}^{d''}(\eta_j)_k}\right|}}\right\}^{{{\tfrac{[K_v:\R]}{[K:\Q]}}}}\Vert H_{\ell}\Vert_v\enspace.
\end{array}$$
We estimate the max for $k$ for the product over all the $\eta_j,\zeta_j$ of the upper bounds above:
\begin{align*}
A=&\max_{0\leq k\leq dmn+dm} \Biggl\{\prod_{j=1}^{d'}\left[\binom{\lceil |\zeta_j| \rceil + k+n-1}{n-1} k! \binom{\lceil |\zeta_j| \rceil+k}{\lceil |\zeta_j| \rceil}\right]\\
     &\prod_{j=1}^{d''}\left[\frac{k}{|\eta_j|}\binom{k+\lceil |\eta_j| \rceil}{k} \cdot \frac{1}{k!\max\{1,k\}^{\lceil |\eta_j| \rceil}}\Biggr\}\right]
\end{align*}
and observe the maximum for each term is either obtained for $k=0$ or $k=dmn+dm$.
Thus,
\begin{align*}
A& \leq  (dmn+dm)!^{\max\{0,d'-d''\}} \prod_{j=1}^{d'}\binom{\lceil |\zeta_j| \rceil + dmn+dm+n-1}{n-1}\\
   &\prod_{j=1}^{d'}\left[\binom{\lceil |\zeta_j| \rceil+dmn+dm}{\lceil |\zeta_j| \rceil}\right]\prod_{j=1}^{d''}\left[\frac{dmn+dm}{|\eta_j|}\binom{dmn+dm+\lceil |\eta_j| \rceil}{dmn+dm} \right]\enspace.
\end{align*}
}

Notice the standard Stirling formula implies
{\small{\begin{align} \label{Stirling}
\log\binom{(dm+1)n+dm+\lceil |\zeta|\rceil}{n-1} =n\left(\log (dm+1)+dm\log \left(\frac{dm+1}{dm}\right)\right)+o(n) \enspace.
\end{align}
}}
We now define 
{\small{\begin{align*}
c'_{n,v}&= \prod_{j=1}^{d''}\frac{dmn+dm}{|\eta_j|} \prod_{j=1}^{d''}\binom{dmn+dm+\lceil |\eta_j| \rceil}{\lceil |\eta_j| \rceil} \prod_{j=1}^{d'}\binom{dmn+dm+\lceil |\zeta_j| \rceil}{\lceil |\zeta_j| \rceil}\\
&\cdot\left(\frac{(dmn+dm)!}{(dmn)!}\right)^{\max\{0,d'-d''\}}\cdot e^{o(1)}\enspace.
\end{align*}}}
We deduce,
\begin{align*}
\Vert P_{\ell}\Vert_v&\leq A^{{^{\tfrac{[K_v:\R]}{[K:\Q]}}}}\Vert H_{\ell}\Vert_v\\
                            &\leq c'_{n,v}\exp\left(n\left(\log (dm+1)+dm\log \left(\frac{dm+1}{dm}\right)\right){{\dfrac{[K_v:\R]}{[K:\Q]}}}\right)\\
                            &\cdot {{|(dmn)!|_v^{\max\{0,d'-d''\}}}}\Vert H_{\ell}\Vert_v\enspace.
\end{align*}
In putting the upper bound for the norm of $H_{\ell}$, one first simplifies the constants, defining for Archimedean $v$:
\begin{align*}
c_{n,v} = \left(c'_{n,v}2^{dm}(dmn+dm)^m\right)^{\tfrac{[K_v:\R]}{[K:\Q]}}\enspace,
\end{align*}
then, consequently, we get
$$
\Vert P_{\ell}\Vert_v\leq  c_{n,v}\exp(B\tfrac{[K_v:\R]}{[K:\Q]}n) \vert (dmn)!\vert_v^{\max\{0,d'-d''\}}\enspace.
$$
The same estimates provide for ultrametric bounds;  the fact that the eigenvalues of $\mathcal{T}_{\mathbf{c}} \circ \bigcirc_{j=1}^{d'} S_{n-1,\zeta_j}$ are
$$\frac{\prod_{j=1}^{d'}(\zeta_j+1)_{k+n-1}}{(n-1)!^{d'}\prod_{j=1}^{d''}(\eta_j)_k}
\enspace,$$
and taking into account the fact that $\Vert H_{\ell}\Vert_v=1$, one gets the ultrametric part of the lemma\footnote{In view of Lemma~\ref{elementary}, this is a bound in $n$ very similar to the Archimedean case, however, it would be premature to simplify the quantity since then one would loose convergence ensured by the product formula when summing over all places.}. One notes that for finite places, it is enough to choose $c_{n,v}=1$ at this stage.

Then, since $c_{n,v}$ is by definition polynomial in $n$,
\begin{align*}&
\frac{1}{n}\sum_{v\in \mathfrak{M}_K}\log\left(c_{n,v}\right)=
\frac{1}{n}\sum_{v\in \mathfrak{M}^{\infty}_{K}}\log\left(c_{n,v}\right)
\raisebox{-4pt}{$\stackrel{\displaystyle\longrightarrow}{\scriptscriptstyle n \longmapsto\infty} $}0\enspace.
\end{align*}
Hence $c_{n,v}$ is an admissible error term and the lemma is completely proven for $P_{\ell}$.
\footnote{to be recomputed with the corrections Taking into account the above consideration, we conclude 
{\scriptsize{$$c_{n,v}
=
\begin{cases}
e^{o(1)}\left(2^{dm}\cdot (dmn)^{\max_{ j',j^{''}} \{\lceil |\eta_{j'}| \rceil, \lceil |\zeta_{j^{''}}| \rceil\}(d'+d'')+m+d+1+dm\max\{0,d'-d''\}}\right)^{[K_v:\R]/[K:\Q]} & \text{if} \ v\mid \infty\enspace,\\
\left|\prod_{w=1}^d\gamma_{w}\right|_v& \text{if} \ v\nmid \infty\enspace.
\end{cases}
$$
}}
}}
\end{proof}
We now turn to the remainder term.
\begin{lemma}\label{estireste} 
Let $u\geq 0$ and integer. {{Let $(c_{n,v})_{v\in \mathfrak{M}_K,n\in \N}$ be the admissible error terms defined in Lemma~$\ref{normepl}$.}}
Then there exists a rational function $A(n,u)$ in $n$ and $u$ depending on $\boldsymbol{\zeta}$ and $\boldsymbol{\eta}$ such that$:$
\begin{itemize}
\item[$({\rm{i}})$] If $v$ is Archimedean and $d''\le d'$,
$$\Vert \varphi_{i,s}(t^{n+u}P_{\ell}(t))\Vert_v\leq c_{n,v} |A(n,u)|_v\exp(B\tfrac{[K_v:\R]}{[K:\Q]}n) \left\vert \dfrac{1}{n!u!}\right\vert^{d'-d''}_v\enspace.$$
\item[$({\rm{ii}})$] If $v$ is ultrametric,
{\small{
$$\Vert\varphi_{i,s}(t^{n+u}P_{\ell}(t))\Vert_v\leq c_{n,v}
\vert n!\vert_v^{-d'}\max\Biggl\{\left|\frac{\prod_{j=1}^{d''}(\eta_j+k)_{n+u}}{\prod_{j=1}^{d'}(\zeta_j+k+n)_{u+1}}\right|_v;\, 0\leq k\leq dmn+dm\Biggr\}\enspace.$$
}}
\end{itemize}
\end{lemma}
\begin{proof} 
We can follow the same approach, as
\begin{align*}
&\varphi_{i,s}(t^{n+u}P_{\ell}(t))=\\
&[\alpha_i]\circ {\rm{Eval}}_{t=\alpha_i}\circ \mathcal{T}^{-1}_{\boldsymbol{c}} \bigcirc_{u=1}^{d-s}(t\tfrac{d}{dt}+\gamma_{u})\circ [t^{n+u}]\circ{\mathcal{T}}_{\bold{c}}\bigcirc_{j=1}^{d'} S_{n-1, \zeta_j}\left(H_{\ell}(t)\right)\enspace.
\end{align*}
This time, we are only left with diagonally acting operators, substitutions of variables (morphisms $\ev$) and shifts (multiplication by $[\alpha_i]$ and $[t^{n+u}]$) and we have a direct estimation of the norm via the eigenvalues.

We need to estimate 
\begin{align} \label{main term}
\max_{0\le k \le dmn+dm}\Biggl\{{\left|\frac{\prod_{j=1}^{d'}(\zeta_j+1)_{k+n-1}}{(n-1)!^{d'}\prod_{j=1}^{d''}(\eta_j)_k}\frac{\prod_{j=1}^{d''}(\eta_j)_{k+n+u}}{\prod_{j=1}^{d'}(\zeta_j+1)_{k+n+u}}\prod_{w=1}^{d-s}(k+n+u+\gamma_w)\right|_v}\Biggr\}\enspace.
\end{align}
Notice
\begin{align} \label{compute factorials}
\frac{\prod_{j=1}^{d'}(\zeta_j+1)_{k+n-1}}{(n-1)!^{d'}\prod_{j=1}^{d''}(\eta_j)_k}\frac{\prod_{j=1}^{d''}(\eta_j)_{k+n+u}}{\prod_{j=1}^{d'}(\zeta_j+1)_{k+n+u}}=
\frac{1}{(n-1)!^{d'}}\frac{\prod_{j=1}^{d''}(\eta_j+k)_{n+u}}{\prod_{j=1}^{d'}(\zeta_j+k+n)_{u+1}}\enspace.
\end{align}

Now, we start the proof of $({\rm{i}})$. 
In Equation~\eqref{main term}, the factor $\prod_{w=1}^{d-s}(k+n+u+\gamma_w)$ is a polynomial in $n$ and $u$ and can conveniently be absorbed into the definition of the error term $A(n,u)$.
Using Lemma~\ref{elementary}~$({\rm{ii}})$ again,
$$|(\eta+k)_{n+u}|\leq \frac{(\lceil |\eta| \rceil+k+n+u-1)!}{(\lceil |\eta| \rceil+k-1)!}\enspace.$$

By definition of $\lfloor\cdot\rfloor$
$${{\left|\frac{1}{(\zeta+k+n)_{u+1}}\right|}}\leq  \frac{(\lfloor {{|}} \zeta {{|}} \rfloor+k+n-1)!}{(\lfloor {{|}} \zeta {{|}}\rfloor+k+n+u)!}\enspace.$$

Now the terms $(\lfloor  {{|}} \zeta {{|}} \rfloor+k+n-1)!$, $(\lfloor  {{|}} \zeta {{|}} \rfloor+k+n+u)!$ are $(k+n)!$ and $(k+n+u)!$ respectively upto an error term and similarly $({{\lceil |}}\eta {{|\rceil}}+k+n+u-1)!$ and $({{\lceil |}}\eta {{|\rceil}}+k-1)!$ are $(k+n+u)!$ and $k!$ respectively which can be put in the definition of $A(n,u)$, and moreover $(n-1)!$ is $n!$ upto an admissible error.

Taking into account the fact that the norm of $H_{\ell}$ is bounded by $2^{dmn}$ up to an admissible error, we deduce that, for Archimedean places, we have:
\begin{align*}
\Vert\varphi_{i,s}(t^{n+u}P_{\ell}(t))\Vert_v&\leq c_{n,v}|A(n,u)|_v |2|^{dmn}_v\\
&\cdot \max\left\{\left\vert\frac{(k+n)!^{d'}}{(k+n+u)!^{d'-d''}n!^{d'}k!^{d''}}\right\vert_v;\, 0\leq k\leq dmn+dm\right\}\enspace.
\end{align*}
Notice that
$$\frac{ (k+n)!^{d'}}{(k+n+u)!^{d'-d''} n!^{d'}k!^{d''}}=\binom{k+n}{n}^{d'}\binom{k+n+u}{k}^{d''-d'}(n+u)!^{d''-d'}\enspace.$$
From this expression, it is evident that $\binom{k+n}{n}^{d'}$ is maximal for $k=dmn+dm$ and, since $d'\geq d''$,  $\binom{k+n+u}{k}^{{d''-d'}}$ is maximal for $k=0$, and replacing this with 
$k=dmn$ introduces only an admissible error; moreover $(n+u)!\geq n!u!$. Thus, the maximum becomes upto admissible error:
$$
\left\vert\binom{(dm+1)n}{n}^{d'} \left(\dfrac{1}{n!u!}\right)^{d'-d''}\right\vert_v \enspace.$$ Taking into account Equation~\eqref{Cmr}, part $({\rm{i}})$ of the lemma follows. 

\medskip

{{We prove $({\rm{ii}})$. Assume $v$ is ultrametric. 
Relying on~\eqref{compute factorials}, Equation~\eqref{main term} is bounded by 
$$\prod_{w=1}^d|\gamma_w|_v|n!|^{-d'}_v \max\Biggl\{\left|\frac{\prod_{j=1}^{d''}(\eta_j+k)_{n+u}}{\prod_{j=1}^{d'}(\zeta_j+k+n)_{u+1}}\right|_v;\kern5pt 0\leq k\leq dmn+dm\Biggr\}\enspace.$$
Since the norm of  $H_{\ell}(t)$ is equal to $1$, the above estimate yields the desired conclusion for part $({\rm{ii}})$.}}
\end{proof}

Recall that if $P$ is a homogeneous polynomial in some variables $y_i,i\in I$, for any point $\boldsymbol{\alpha}=(\alpha_i)_{i\in I}\in K^{\card(I)}$ where $I$ is any finite set, and $\Vert\cdot\Vert_v$ stands for the sup norm in $K_v^{\card(I)}$, with $$C_v(P)= (\deg(P)+1)^{\frac{\varepsilon_v{{[K_v:\R]}}(\card(I))}{{{[K:\Q]}}}}\enspace,$$
one has
\begin{equation}\label{estimhomo}
\vert P(\boldsymbol{\alpha})\vert_v\leq C_v(P) \Vert P\Vert_v\cdot {\rm{H}}_v({\boldsymbol{\alpha}})^{\deg(P)}\enspace.
\end{equation}
So, the preceding lemma yields trivially estimates for the $v$-adic norm of the above given polynomials. Moreover, since all the polynomials involved here ($P_{\ell}, P_{\ell,i,s},\varphi_{i,s}(t^{n+u}P_{\ell}))$ are of degreee polynomial in $n$, the Archimedean error term $C_{v}(P)$ above is an admissible error in $n$ independent {of $u$\footnote{The number of non-zero terms is at most the number of coefficients of $H_{\ell}$ for $\varphi_{i,s}(t^{n+u}P_{\ell})$ and is thus independent of  $u$ since $\varphi_{i,s}(t^{n+u}P_{\ell})$ factors by $\alpha_i^{n+u}$.}.}

Recall that by Proposition~$3.9$ and Definition~\ref{defope} the polynomials $P_{\ell}, P_{\ell,i,s}$ are of degree at most  $dmn+\ell$ and $dmn+\ell-1$ respectively and $\varphi_{i,s}(t^{n+u}P_{\ell})$ is of degree at most $dmn+\ell+n+u+1$.

We now turn to the issue of convergence.
For $v\in \mathfrak{M}_K$, we denote the embedding $K$ into $K_v$, and the extension to the Laurent series ring by $$\sigma_v: K[[1/z]]\longrightarrow K_v[[1/z]]; \ f(z)\mapsto f_v(z):=\sigma_v(f(z))\enspace.$$

\begin{lemma} \label{upper jyouyonew}
Let $\boldsymbol{\alpha}=(\alpha_1,\ldots,\alpha_m)\in (K\setminus\{0\})^m$, $\beta\in K\setminus\{0\}$ and $v\in \mathfrak{M}_K$. 
{{Let $(c_{n,v}(\boldsymbol{\zeta},\boldsymbol{\eta}))_{v\in \mathfrak{M}_K,n\in \N}$ be the admissible error terms defined in Lemma~$\ref{normepl}$.}}
Let $i,\ell,s$ be integers such that $1\le i\leq m,0\le \ell\leq dm, {1 \le s\leq d}$.
Recall ${\rm{deg}}\,A=d'',{\rm{deg}}\,B=d'$ and $\max\{d',d''\}=d$.  
Then\footnote{One may note that the series does not converge at Archimedean places if $d'<d''$ and that if $d'>d''$, the ultrametric series do not provide valuable enough information to offset the norm of the Pad\'e approximants. Hence, there is no loss of generality to restrict ourvselves to these cases from now one. Also, the behavior of the functions differ fundamentally depending on $d'>d''$, $d'=d''$ and $d'<d''$ and it makes sense to distinguish cases from now on.} 

\medskip

$({\rm{i}})$ Assume $d'=d''$, $v\in \mathfrak{M}^{\infty}_K$ and $|\alpha_i|_v<|\beta|_v$.
Then the series $R_{\ell,i,s,v}(z)$ converges to an element of $K_v$ at $z=\beta$ and there exists an admissible constant $c_{n,v}$ such that
$$
|R_{\ell,i,s,v}(\beta)|_v \leq  c_{n,v} {\rm{H}}_v(\boldsymbol{\alpha})^{(dm+1)n{+\ell+1}}\exp\left(B\tfrac{[K_v:\R]}{[K:\Q]}n\right)\vert \beta\vert_v^{-n-1}\enspace.$$

\medskip

$({\rm{ii}})$ Assume $d''<d'$ and $v\in \mathfrak{M}^{\infty}_K$. 
Then the series $R_{\ell,i,s,v}(z)$ converges to an element of $K_v$ at $z=1$ and 
\begin{align*}
|R_{\ell,i,s,v}(1)|_v &\leq \exp(C_vn) \left\vert \dfrac{1}{n!}\right\vert^{d'-d''}_v\enspace,
\end{align*}
where $C_v$ is a constant depending on $\boldsymbol{\alpha}, \boldsymbol{\eta}, \boldsymbol{\zeta}$ and $v$.

\medskip

$({\rm{iii}})$ Assume $d''\ge d'$, $v\in \mathfrak{M}^{f}_K$ and 
\begin{align} \label{convergence cond}
\left|\dfrac{\alpha_i}{\beta}\right|_v<\prod_{j=1}^{d''}\mu_v(\eta_j)
 |p|^{\tfrac{d''-d'}{p-1}}_v\enspace,
\end{align} 
where $p$ is the prime below $v$. Then the series $R_{\ell,i,s,v}(z)$ converges to an element of $K_v$ at $z=\beta$.

\medskip

$({\rm{iv}})$ Assume $d''=d'$, $v\in \mathfrak{M}^{f}_K$ and Equation~\eqref{convergence cond}. 
Then {{there exists an admissible constant $c_{n,v}$ such that}}
\begin{align*}
|R_{\ell,i,s,v}(\beta)|_v &\le e^{o(n)}c_{n,v}{\rm{H}}_v(\boldsymbol{\alpha})^{(dm+1)n} |\beta|^{-n}_v \prod_{j=1}^{d''}|\mu_n(\eta_{j})|^{-1}_v\enspace.
\end{align*}

\medskip

$({\rm{v}})$ Assume $d''>d'$ and each $\alpha_i$ is an algebraic integer. 
Let $v\in \mathfrak{M}^{f}_K$ with $|\eta_j|_v\le 1$ for any $1\le j \le d'$. 
Denote $p$ the prime number below $v$. Assume $p\ge e^{\tfrac{d^{''}}{d''-d'}}$ and 
\[
p\le {\Delta_n:=}\max_{1\le j \le d'} \biggl\{{{|}}{\rm{den}}(\zeta_j)\zeta_j{{|}}+{\rm{den}}(\zeta_j)((dm+1)n+dm)\biggr\}\enspace.
\] 
Then there exists an admissible constant $c_{n,v}$ such that
\begin{align*}
|R_{\ell,i,s,v}(1)|_v &\le c_{n,v} \delta_v(n)p^{d'\tfrac{[K_v:\Q_p]}{[K:\Q]}}
\prod_{j=1}^{d''}|\mu_n(\eta_{j})|^{-1}_v\cdot |n!|^{d''-d'}_v\enspace,
\end{align*}
where {{$$\delta_v(n)= \prod_{j=1}^{d'} \left(|{\rm{den}}(\zeta_j)\zeta_j|+{\rm{den}}(\zeta_j)((dm+1)n+dm)\right)^{\tfrac{[K_v:\Q_p]}{[K:\Q]}}\enspace.$$}}
\end{lemma}  
\begin{proof}
Recall we have
\begin{align} \label{R}
R_{\ell,i,s}(z)=\sum_{u=0}^{\infty}\frac{\varphi_{i,s}(t^{u+n}P_{\ell}(t))}{z^{u+n+1}} \enspace.
\end{align}

We now start the proof of $({\rm{i}})$. 
Let $v$ be an Archimedean place.
{{Combining Lemma~\ref{estireste} (i) with $d'=d''$ and Equation~\eqref{estimhomo} yields
\begin{align}
|\varphi_{i,s}(t^{n+u}P_{\ell}(t))|_v&\leq \left\Vert\varphi_{i,s}\left(t^{n+u}P_{\ell}(t)\right)\right\Vert_v \vert{\alpha_i}\vert_v^{dmn+\ell+n+u+1} \nonumber\\
&\leq c_{n,v}(\boldsymbol{\zeta},\boldsymbol{\eta}) \exp\left(B\tfrac{[K_v:\R]}{[K:\Q]}n\right) {\rm{H}}_v(\boldsymbol{\alpha})^{dmn+\ell}|\alpha_i|^{u+n+1}_v
\vert A(n,u)\vert_v \label{numerator estimate} 
\end{align}}}
Since $A(n,u)$ is a rational function with respect $n$ and $u$, 
\begin{align} \label{|A|}
|A(n,u)|_v\le C\cdot |n|^{{\rm{deg}}_n A}_v  \cdot |u|^{{\rm{deg}}_u A}_v
\end{align}
where $C$ is a constant. 
Since $\vert\alpha_i\vert_v<\vert\beta\vert_v$, Equation~\eqref{numerator estimate}, together with above inequality, yields:
\begin{align*}
\vert R_{\ell,i,s,v}(\beta)\vert_v&\leq  c_{n,v}(\boldsymbol{\zeta},\boldsymbol{\eta})\exp\left(B\tfrac{[K_v:\R]}{[K:\Q]}n\right) {\rm{H}}_v(\boldsymbol{\alpha})^{dmn+\ell} 
\sum_{u=0}^{\infty}|A(n,u)|_v \left|\dfrac{\alpha_i}{\beta}\right|^{n+u+1}\\
&\leq c_{n,v}{\rm{H}}_v(\boldsymbol{\alpha})^{(dm+1)n+\ell+1}\exp\left(B\tfrac{[K_v:\R]}{[K:\Q]}n\right)\vert \beta\vert_v^{-n-1}\enspace. 
\end{align*}
This proves $({\rm{i}})$.  

\medskip

$({\rm{ii}})$ Let $v$ be an Archimedean place. Since $F_{s,v}(z)$ are entire function on $K_v$, the series $R_{\ell,i,s,v}(z)$ converges to an element of $K_v$ at $z=1$. 
Combining Lemma~\ref{estireste} $({\rm{i}})$ and Equation~\eqref{estimhomo} together with~\eqref{|A|} yields
{\small{\begin{align*}
\vert R_{\ell,i,s,v}(1)\vert_v&\leq \sum_{u=0}^{\infty} c_{n,v}(\boldsymbol{\zeta},\boldsymbol{\eta}){\rm{H}}_v(\boldsymbol{\alpha})^{dmn{+\ell}} |A(n,u)|_v\exp\left(B\tfrac{[K_v:\R]}{[K:\Q]}n\right)  
\left\vert \dfrac{1}{n!u!}\right\vert^{d'-d''}_v |\alpha_i|^{n+u+1}_v \\
&\leq c_{n,v}(\boldsymbol{\zeta},\boldsymbol{\eta}) {\rm{H}}_v(\boldsymbol{\alpha})^{(dm+1)n{+\ell+1}}|n|^{{\rm{deg}}_n A}_v {\exp\left(B\tfrac{[K_v:\ru]}{[K:\qu]}n\right)}\\
& \cdot \left\vert \dfrac{1}{n!}\right\vert^{d'-d''}_v\sum_{u=0}^{\infty} |u|^{{\rm{deg}}_u(A)}_v \left|\dfrac{\alpha_i^u}{u!^{d'-d''}}\right|_v
\enspace.
\end{align*}}}
Since the series $\sum_{u=0}^{\infty} |u|^{{\rm{deg}}_u(A)}_v \left|\alpha_i^u/u!^{d'-d''}\right|_v$ converges in $\R$,  
we see that there exists a constant $C_v$ such that  
$$\vert R_{\ell,i,s,v}(1)\vert_v \le \exp(C_vn) \left\vert \dfrac{1}{n!}\right\vert^{d'-d''}_v\enspace.$$
This completes the proof of $({\rm{ii}})$.

\medskip

$({\rm{iii}})$ Let $v$ be a non-Archimedean place.  Denote by $p$ the rational prime lying below $v$. 
Let us show $\displaystyle{\lim_{u\to \infty}} \left|\varphi_{i,s}(t^{n+u}P_{\ell}(t))/\beta^{n+u+1}\right|_v=0$ under the assumption \eqref{convergence cond}. 
Combining Lemma~$\ref{estireste}$ $({\rm{ii}})$ and Equation~\eqref{estimhomo} yields that 
\begin{align*}
\left|\dfrac{\varphi_{i,s}(t^{n+u}P_{\ell}(t))}{\beta^{n+u+1}}\right|_v&\le  c_{n,v}(\boldsymbol{\zeta},\boldsymbol{\eta}){\rm{H}}_{v}(\boldsymbol{\alpha})^{dmn+\ell}\\
&\cdot \max_{0\le k \le dmn+dm}\Bigg\{\left|\frac{1}{(n)!^{d'}}\frac{\prod_{j=1}^{d''}(\eta_j+k)_{n+u}}{\prod_{j=1}^{d'}(\zeta_j+k+n)_{u+1}}\left(\dfrac{\alpha_i}{\beta}\right)^{n+u+1}\right|_v\Biggr\} \enspace.
\end{align*}
By Lemma~\ref{denominator} (i) and (ii),
$$\vert (\eta+k)_{n+u}\vert_v\leq \vert \mu_{n+u}(\eta)\vert_v^{-1}\vert(n+u)!\vert_v\enspace,$$
and, by Lemma~\ref{denominator} (iv), if $p \mid\prod_{l=0}^{k+n+u} (\den(\zeta)\zeta+\den(\zeta)l)$
$$\left\vert\frac{(u+1)!}{(\zeta+k+n)_{u+1}}\right\vert_v\leq \left(|\den(\zeta)\zeta|+\den(\zeta)(k+n+u)\right)^{\frac{[K_v:\qu_p]}{[K:\qu]}}\enspace,$$
and of course if $v$ does not divide the above quantity, $\left\vert\frac{(u+1)!}{(\zeta+k+n)_{u+1}}\right\vert_v\leq 1$.

We now set
{{$$\delta_v(n,u)= \prod_{j=1}^{d'} \left(|{\rm{den}}(\zeta_j)\zeta_j|+{\rm{den}}(\zeta_j)((dm+1)n+dm+u)\right)^{\tfrac{[K_v:\Q_p]}{[K:\Q]}}\enspace.$$}}
Taking the product over all $\eta_j,\zeta_j$, one deduces (using $\mu_n(\zeta)$ is increasing in $n$)
{\footnotesize{
$$\max_{0\le k \le dmn+dm}\Bigg\{\left|\frac{1}{n!^{d'}}\frac{\prod_{j=1}^{d''}(\eta_j+k)_{n+u}}{\prod_{j=1}^{d'}(\zeta_j+k+n)_{u+1}}\right|_v\Biggr\}\leq  \frac{\delta_v(n,u)\prod_{j=1}^{d''}\vert\mu_{n+u}(\eta_j) \vert_v^{-1}\vert (n+u)!\vert_v^{d''}}{\vert n!\vert_v^{d'}\vert(u+1)!\vert_v^{d'}}\enspace.$$
}}

We now simplify the combinatorial factors
$$\frac{(n+u)!^{d''}}{n!^{d'}(u+1)!^{d'}}=\frac{1}{n^{d'}}\binom{n+u}{u+1}^{d'}\binom{n+u}{u}^{d''-d'}n!^{d''-d'}u!^{d''-d'}\enspace.$$

We conclude for all place $v$ above a prime number $p$ satisfying $$p\mid \prod_{j=1}^{d'} \prod_{k=0}^{dmn+dm} \left[{\rm{den}}(\zeta_j)\zeta_j+{\rm{den}}(\zeta_j)(k+n+u)\right]\enspace,$$ 
then we have
\begin{align} \label{main estimate}
\left|\dfrac{\varphi_{i,s}(t^{n+u}P_{\ell}(t))}{\beta^{n+u+1}}\right|_v\le &c_{n,v}(\boldsymbol{\zeta},\boldsymbol{\eta})|n!|^{d''-d'}_v \prod_{j=1}^{d'}|\mu_{n}(\eta_j)|_v^{-1} {\rm{H}}_{v}(\boldsymbol{\alpha})^{{{(dm+1)n}}+\ell+1} \left|{\beta}\right|^{-n-1}_v  \\
&\cdot \left(\vert u!\vert_v ^{d''-d'}\prod_{j=1}^{d''}|\mu_{u}(\eta_j)|^{-1}_v  \delta_v(n,u)\left|\dfrac{\alpha_i}{\beta}\right|^{u}_v\right)\enspace. \nonumber
\end{align}
Now, since $\delta_v(n,u)$ is a polynomial in $n,u$, assuming $\vert {\alpha_i}/{\beta}\vert_v{{<}} \vert p\vert_v^{\frac{d''-d'}{p-1}}\prod_{j=1}^{d''}\mu_v(\eta_j)$, 
one deduces 
$$\lim_{u\rightarrow\infty}\left(\vert u!\vert_v ^{d''-d'}\prod_{j=1}^{d''}|\mu_{u}(\eta_j)|^{-1}_v  \delta_v(n,u)\left|\dfrac{\alpha_i}{\beta}\right|^{u}_v\right)=0\enspace.$$

This shows that $R_{\ell,i,s,v}(z)$ converges to an element of $K_v$ at $z=\beta$.

The same argument works if $v$ is not above a prime dividing $$p\mid \prod_{j=1}^{d'} \prod_{k=0}^{dmn+dm} \left[{\rm{den}}(\zeta_j)\zeta_j+{\rm{den}}(\zeta_j)(k+n+u)\right]\enspace.$$
We now turn to bounding the series.

\medskip

$({\rm{iv}})$ Under the assumption~\eqref{convergence cond}, since $\delta_v(n,u)$ is a polynomial in $n,u$, we have
$$
\max_{{0\le u}}\Biggl\{\vert u!\vert_v ^{d''-d'}\prod_{j=1}^{d''}|\mu_{u}(\eta_j)|^{-1}_v  \delta_v(n,u)\left|\dfrac{\alpha_i}{\beta}\right|^{u}_v\Biggr\}=e^{o(n)}\enspace.
$$
Combining \eqref{main estimate} {{with $d'=d''$}} and above equality yields
$$
|R_{\ell,i,s,v}(\beta)|_v\le \max_{0 \le u}\Biggl\{\left|\dfrac{\varphi_{i,s}(t^{n+u}P_{\ell}(t))}{\beta^{n+u+1}}\right|_v\Biggr\}\le c_{n,v} {\rm{H}}_v(\boldsymbol{\alpha})^{(dm+1)n} |\beta|^{-n}_v \prod_{j=1}^{d''}|\mu_n(\eta_{j})|^{-1}_v\enspace.
$$

\medskip

$({\rm{v}})$ The definition of $\delta_v(n,u)$ yields
$$
\delta_v(n,u)\leq 2\delta_v(n)\max\{1,u\}^{d'\tfrac{[K_v:\Q_p]}{[K:\Q]}}\enspace.
$$ 
Hence, at places where $\vert {\alpha_i}\vert_v\le 1$ and $|\den(\eta_j)|_v=1$,
{\small{$$
\max_{0\le u}\Biggl\{\vert u!\vert_v ^{d''-d'}\prod_{j=1}^{d''}|\mu_{u}(\eta_j)|^{-1}_v \delta_v(n,u)\left|\dfrac{\alpha_i}{\beta}\right|^{u}_v\Biggr\}\leq 2 \delta_v(n) \max_{1\le u}\biggl\{u^{d'\tfrac{[K_v:\Q_p]}{[K:\Q]}} \vert p\vert_v^{(d''-d')v_p(u!)}\biggr\}\enspace.
$$}}
Under the assumption $p\ge e^{\tfrac{d^{''}}{d''-d'}}$, let us show  
\begin{align} \label{trivial upper bound}
\max_{1\le u}\biggl\{u^{d'\tfrac{[K_v:\Q_p]}{[K:\Q]}} \vert p\vert_v^{(d''-d')v_p(u!)}\biggr\}\leq p^{d'\tfrac{[K_v:\Q_p]}{[K:\Q]}}\enspace.
\end{align}
Notice, by taking $u=p-1$, $$(p-1)^{d'\tfrac{[K_v:\Q_p]}{[K:\Q]}}\le \max_{1\le u}\biggl\{u^{d'\tfrac{[K_v:\Q_p]}{[K:\Q]}} \vert p\vert_v^{(d''-d')v_p(u!)}\biggr\}\enspace.$$ 
We now consider $u\ge p-1$. 
Denote the $p$-adic expansion of $u$ by $\sum_{i=0}^{\ell}a_ip^i$ with $0\le a_i <p$ and $S_u=\sum_{i=0}^{\ell}a_i$.
Using the identity $v_p(u!)=\tfrac{u-S_u}{p-1}$ with  $S_u\le (p-1)\log_p(u){+\frac{(p-1)}{u\log p}}$, we obtain the estimate
\begin{align*}
u^{d'}\vert p\vert_p^{(d''-d')v_p(u!)}&\le u^{d^{''}} p^{-\tfrac{(d''-d')u}{p-1}}\exp\left(\frac{d''-d'}{u}\right)\\
&=\exp\left(d^{''}\log u-\log p\dfrac{(d''-d')u}{p-1}+\frac{d''-d'}{u}\right)\enspace.
\end{align*}
The function $u\longmapsto 1/u$ is maximal at $u=1$ and $u\longmapsto d^{''}\log u-\log p\tfrac{(d''-d')u}{p-1}$ achieves its maximum on $u>0$ at $$u=\dfrac{(p-1)d^{''}}{(d''-d')\log p}\enspace,$$ and both functions are decreasing right of their maximal value.
The assumption $p\ge e^{\tfrac{d^{''}}{d''-d'}}$ yields $$\dfrac{(p-1)d^{''}}{(d''-d')\log p}\le p-1\enspace.$$ 
In the range $u\leq p-1$, we have already seen the desired bound~\eqref{trivial upper bound} holds. Thus, we can assume $u=p$ and readily check the bound also holds.
Combining above considerations concludes  
$$|R_{\ell,i,s,v}(1)|_v\le  c_{n,v} \delta_v(n)p^{d'\tfrac{[K_v:\Q_p]}{[K:\Q]}}
\prod_{j=1}^{d''}|\mu_n(\eta_{j})|^{-1}_v\cdot |n!|^{d''-d'}_v\enspace.$$
\end{proof}}}
}

{{\section{Proof of Main theorems}\label{maintheoremssection}
{{In this section, we keep notations of Section $\ref{estimate}$.
Recall $K$ is a number field and $\eta_1,\ldots,\eta_{{d''}},\zeta_1,\ldots,\zeta_{d'}$, which are rational numbers that are not negative integers with $\eta_i-\zeta_j\notin \N$ for $1 \le i \le d'', 1\le j \le d'$.
Put $d=\max\{d',d''\}$.
Assume $$A(X)=(X+\eta_1)\cdots (X+\eta_{{d''}}), \ \ \ B(X)=(X+\zeta_1)\cdots (X+\zeta_{d'})$$
 where $d'd''>0$. Let $\boldsymbol{\alpha}=(\alpha_1,\ldots,\alpha_m)\in (K\setminus\{0\})^m$ whose coordinates are pairwise distinct.
We also recall the Pad\'{e} approximants  $P_{\ell}(z),P_{\ell,i,s}(z)$ defined in Proposition~$\ref{GHG pade}$ for the above data. 
We now prove our main theorems stated in Section~$\ref{main results}$ by considering the cases for the relationship between $d''$ and $d'$, namely $d''=d'$, $d''<d'$ or $d''>d'$.}}
\subsection{Proof of Theorem~$\ref{hypergeometric}$}\label{proof} 
We consider the $G$-function case that is $d'=d''=d$, which is the situation of Theorem~$\ref{hypergeometric}$.
We want to apply the qualitative linear independence criterion \cite[Proposition~$5.6$]{DHK3}.

\medskip

Let us define notation. 
For $\beta\in K\setminus \{0\}$, we put the $dm+1$ by $dm+1$ matrices ${\rm{M}}_n$ by
$${\rm{M}}_n=\begin{pmatrix} P_{\ell}(\beta) \\ P_{\ell,i,s}(\beta)\end{pmatrix}.$$
Note that Proposition~$\ref{non zero det}$ ensures ${\rm{M}}_n\in {\rm{GL}}_{dm+1}(K)$. 
Recall $$B=dm\log(2)+d\left(\log (dm+1)+dm \log\left(\frac{dm+1}{dm}\right)\right)\enspace,$$
and{, for $v\in\mathfrak{M}_K$ and $p\in\mathfrak{M}_{\Q}$ the prime below $v$, define the functions $F_v:\N\longrightarrow \R$ by 
{\small{\begin{align*}
F_v(n)&=n\left(\varepsilon_{v}B\dfrac{[K_v:\qu_p]}{[K:\Q]}+\left(dm+\dfrac{dm{+1}}{n}\right){\rm{h}}_v(\boldsymbol{\alpha},\beta)+(1-\varepsilon_{v})\log A_{n,v}(\boldsymbol{\zeta},\boldsymbol{\eta})\right)\\
         &+\log c_{n,v}\enspace,
\end{align*}
}}
where the real numbers $A_{n,v}(\boldsymbol{\zeta},\boldsymbol{\eta}),c_{n,v}$ are defined in Lemma~$\ref{normepl}$.
So combining Lemma~$\ref{normepl}$ and Equation~\eqref{estimhomo} yields, $$||{\rm{M}}_n||_v\le e^{F_v(n)} \ \ \text{for} \ \ v\in\mathfrak{M}_K\enspace.$$

We now choose a place $v_0$ of $K$ and define a real number 
$$\mathbb{A}_{v_0}(\beta)=\log|\beta|_{v_0}-(dm+1) {\rm{h}}_{v_0}(\boldsymbol{\alpha})-\varepsilon_{v_0} B\dfrac{[K_{v_0}:{\qu_p}]}{[K:\Q]}+{{\sum_{j=1}^d\log\mu_{v_0}(\eta_j)}}\enspace.$$
Then Lemma~$\ref{upper jyouyonew}$ $({\rm{i}})$ and $({\rm{iv}})$ imply
$$\log |R_{\ell,i,s}(\beta)|_{v_0}\le -\mathbb{A}_{v_0}(\beta)n+o(n)\enspace.$$

\medskip

We check the condition of  the qualitative linear independence criterion \cite[Proposition~$5.6$]{DHK3}.

\begin{lemma}\label{sommenorme} One has
$$\sum_{v\in \mathfrak{M}^f_K\setminus\{v_0\}}\log A_{n,v}(\boldsymbol{\zeta},\boldsymbol{\eta})\leq dm\sum_{j=1}^d{{{\rm{den}}(\eta_j)}}+(dm+1)\sum_{j=1}^d\log \mu(\zeta_j)\enspace.$$
\end{lemma}
\begin{proof}
By Lemma~\ref{denominator} (ii) joined with (i)
$\mu_{n-1}(\zeta)\frac{(\zeta+l+1)_{n-1}}{(n-1)!}$ is an integer and  by Lemma~\ref{denominator} (ii) joined with  (iii)
$$D_{k,l}={\rm den} \left(\frac{(\zeta+l+n+1)_{k-l}}{(\eta+l)_{k-l}}\right)_{0\leq l\leq k-1, 0\leq k\leq dmn+dm}$$
satisfies
$$\limsup_{n}\frac{1}{n}\log(D_{k,l})\leq dm\log \mu(\zeta)+dm\den(\eta)\enspace.$$
Putting together, we deduce
\begin{align*}
\limsup_n\dfrac{1}{n}
\sum_{\substack{v\in \mathfrak{M}^{f}_K \setminus\{v_0\}}}\log A_{n,v}(\boldsymbol{\zeta},\boldsymbol{\eta})
\le  dm\sum_{j=1}^d{{{\rm{den}}(\eta_j)}}+(dm+1)\sum_{j=1}^d\log \mu(\zeta_j)
\enspace.
\end{align*}
\end{proof}
\begin{lemma} \footnote{We easily see that the criterion \cite[Proposition~$5.6$]{DHK3} is also verified replacing $\lim_n \tfrac{1}{n}\sum_{v}F_v(n)<\infty$ by $\limsup_n \tfrac{1}{n}\sum_{v}F_v(n)<\infty$.}
We have
{\small{\begin{align*}
\limsup_n\dfrac{1}{n}\sum_{v\neq v_0}F_v(n)\le  \mathbb{B}_{v_0}(\beta):&=dm({\rm{h}}(\boldsymbol{\alpha},\beta)-{\rm{h}}_{v_0}(\boldsymbol{\alpha},\beta))+B\left(1-\varepsilon_{v_0}\dfrac{[K_{v_0}:\qu_p]}{[K:\Q]}\right)\\
&+dm\sum_{j=1}^d{{{\rm{den}}(\eta_j)}}+(dm+1)\sum_{j=1}^d\log \mu(\zeta_j)
\enspace.
\end{align*}}}
\end{lemma}
\begin{proof}
By the definition of $F_v(n)$, and Lemma~\ref{sommenorme} and summation over all places, one gets the statement.

\end{proof}
Define 
\begin{align*}
V_{v_0}(\beta):&=\mathbb{A}_{v_0}(\beta)-\mathbb{B}_{v_0}(\beta)\\
&=\log|\beta|_{v_0}+dm({\rm{h}}_{v_0}(\boldsymbol{\alpha})-{\rm{h}}(\boldsymbol{\alpha},\beta))-(dm+1){\rm{h}}_{v_0}(\boldsymbol{\alpha})-B\\
&-\Biggl\{dm\sum_{j=1}^d{{\rm{den}}(\eta_j)}+(dm+1)\sum_{j=1}^d\log \mu(\zeta_j)+
\sum_{j=1}^d\left(\log\mu_{v_0}(\eta_j)
\right)
\Biggr\}\enspace.
\end{align*}
By direct application of \cite[Proposition~$5.6$]{DHK3} implies the following result.
\begin{theorem} \label{thm 2}
Let $v_0\in \mathfrak{M}_K$
such that
$V_{v_0}(\beta)>0$.
For any $\boldsymbol{\gamma}=(\gamma_1,\ldots,\gamma_{d-1})\in K^{d-1}$, the functions $F_s(\boldsymbol{\gamma},z), \,\, {1 \leq s \leq d}$ converge around $\alpha_i/\beta$ in $K_{v_0}$,  $1\leq i \leq m$ and
for any positive number $\varepsilon$ with $\varepsilon<V_{v_0}(\beta)$, there exists an effectively computable positive number $H_0$ depending on $\varepsilon$ and the given data such that the following property holds.
For any ${{\boldsymbol{\lambda}}}=({{\lambda_0}},{{\lambda_{i,s}}})_{\substack{1\le i \le m \\ {1 \le s \le d}}} \in K^{dm+1} \setminus \{ \bold{0} \}$ satisfying $H_0\le {\mathrm{H}}({{\boldsymbol{\lambda}}})$, then 
\begin{align*}
\left|{{\lambda_0}}+\sum_{i=1}^m\sum_{s={1}}^{{d}}{{\lambda_{i,s}}}F_{s}(\boldsymbol{\gamma},\alpha_i/\beta)\right|_{v_0}>C(\beta,\varepsilon){\mathrm{H}}_{v_0}({{\boldsymbol{\lambda}}}) {\mathrm{H}}({{\boldsymbol{\lambda}}})^{-\mu(\beta,\varepsilon)}\enspace,
\end{align*}
where 
\begin{align*}
&\mu(\beta,\varepsilon)=\dfrac{\mathbb{A}_{v_0}(\beta)+{{U}}_{v_0}(\beta)}{V_{v_0}(\beta)-\epsilon} \enspace,\\
&C(\beta,\varepsilon)=\exp\left(-{{\left(\frac{\log(2)}{V_{v_0}(\beta)-\varepsilon}+1\right)}}(\mathbb{A}_{v_0}(\beta)+{{U}}_{v_0}(\beta)\right)\enspace,\\
&U_{v_0}(\beta)=\limsup_n \dfrac{1}{n}F_{v_0}(n)\enspace.
\end{align*}
\end{theorem} 
\begin{proof}[\textbf{Proof of Theorem~$\ref{hypergeometric}$.}] 
In Theorem~$\ref{thm 2}$, we take $\boldsymbol{\eta}=(a_1,\ldots,a_d),\boldsymbol{\zeta}=(b_1-1,\ldots,b_{d-1}-1,0)$ and $\boldsymbol{\gamma}=(a_{d-1},\ldots,a_{1})$. 
Then $V_{v_0}(\boldsymbol{\alpha},\beta)\le V_{v_0}(\beta)$. 
Combining Equation~\eqref{Fs p>} and Theorem~$\ref{thm 2}$ yields the assertion of Theorem~$\ref{hypergeometric}$.
\end{proof}

{{\begin{example}
Applying Theorem~$\ref{hypergeometric}$ for $d=2$, $(a_1,a_2)=(-1/2,1/2)$ and $b_1=1$ yields a linear independence criterion concerning the following solutions of the Gauss-Manin connection for the Legendre family of elliptic curves $(${\it confer} \cite[$7.1$]{andre}$)$:$${}_{2}F_{1} \Biggl(\begin{matrix} \tfrac{1}{2}, \tfrac{1}{2}\\ 1 \end{matrix} \biggm| z \Biggr),  {}_{2}F_{1} \biggl(\begin{matrix} \tfrac{1}{2}, -\tfrac{1}{2}\\ 1 \end{matrix} \biggm| z \biggr)\enspace.$$
Here we consider $K=\Q$ and take $v_0$ a place of $\Q$, $m=10$ and $\boldsymbol{\alpha}=(1,2,\ldots,10)$.
Now, let us establish a sufficient condition for $\beta\in \mathbb{Q}$, considering the $21$-elements
$$1, {}_{2}F_{1} \biggl(\begin{matrix}\tfrac{1}{2},\tfrac{1}{2}\\ 1\end{matrix} \biggm| \frac{j}{\beta}\biggr), {}_{2}F_{1} \biggl(\begin{matrix} \tfrac{1}{2},-\tfrac{1}{2}\\ 1 \end{matrix} \biggm| \frac{j}{\beta}\biggr) \in \Q_{v_0} \ \ \ \text{for} \ \ \ 1\le j \le 10$$
are linearly independent over $\mathbb{Q}$.
In case of $v_0=\infty$ and $\beta\in \mathbb{Z}\setminus\{0\}$, we have $V_{\infty}(\boldsymbol{\alpha},\beta)> \log |\beta|-150.2579$. Theorem~$\ref{hypergeometric}$ implies that if $\beta$ satisfies 
$|\beta|\ge e^{150.2579}$ then
$$1, {}_{2}F_{1} \biggl(\begin{matrix} \tfrac{1}{2},\tfrac{1}{2}\\ 1 \end{matrix} \biggm| \frac{j}{\beta}\biggr), {}_{2}F_{1} \biggl(\begin{matrix} \tfrac{1}{2},-\tfrac{1}{2}\\ 1 \end{matrix} \biggm| \frac{j}{\beta}\biggr)\in \mathbb{R}
 \ \ \ \text{for} \ \ \ 1\le j \le 10$$
are linearly independent  over $\mathbb{Q}$.

In case of $v_0=p$ where $p$ is a rational prime and $\beta=p^{-k}$ for a positive integer $k$, we have $V_p(\boldsymbol{\alpha}, p^{-k})> k\log p-150.2579+4\log |2|_p$.
Theorem~$\ref{hypergeometric}$ asserts that if $p$ and $k$ satisfy either $p\ge e^{150.2579-4\log |2|_p}$ and $k=1$ or 
{\scriptsize$$\begin{array}{|c| c |c |c | c|c|c|c|c|c|c|c|c|c|c|c|c|c|c|c|c}
\hline
  p & 2 & 3 & 5 & 7 & 11 & 13& 17& 19 &23 & 29 & 31&37 &41&43&47&53&59&61 &67\\
\hline
 k\ge & 221 & 137 & 94 & 78 & 63 & 59 & 54 & 52 &48 & 45 & 44& 42 & 41& 40 & 40& 38 &37 &37 &36\\
\hline
\end{array}
$$
}
then 
$$1, {}_{2}F_{1} \biggl(\begin{matrix} \tfrac{1}{2},\tfrac{1}{2}\\ 1 \end{matrix} \biggm| jp^k\biggr), {}_{2}F_{1} \biggl(\begin{matrix} \tfrac{1}{2},-\tfrac{1}{2}\\ 1 \end{matrix} \biggm| jp^k\biggr)\in \mathbb{Q}_p
 \ \ \ \text{for} \ \ \ 1\le j \le 10$$
are linearly independent over $\mathbb{Q}$.
\end{example}}}

\subsection{Proof of Theorem~$\ref{hypergeometric2}$}
We assume ${\rm{deg}}\,A=d''<{\rm{deg}}\,B=d$. We now fix an embedding $\overline{\Q}\hookrightarrow \C$. For $v\in \mathfrak{M}^{\infty}_K$, denote the embedding corresponds to $v$ by $\sigma_v:K\hookrightarrow \C$.
\begin{theorem} \label{general 2}
Let $K$ be an algebraic number field and $\boldsymbol{\alpha}=(\alpha_1,\ldots,\alpha_m)\in (K\setminus\{0\})^m$ whose coordinates are pairwise distinct.
Then, the $dm+1$ complex numbers 
$$1, F_s(\boldsymbol{\gamma}, \alpha_i) \ \ \ \text{for} \ \ 1\le i \le m, 1\le s \le d$$
are linearly independent over $\overline{\Q}$.
\end{theorem}

To prove Theorem~$\ref{general 2}$, we rely on the following remarkable results of F.~Beukers \cite[Proposition 4.1]{Beu} and Fischler and Rivoal \cite[Proposition 1]{F-R}, which are grounded in the theory of $E$-operators developed by Andr\'{e} \cite{An2}.
\begin{proposition} $(${\rm{{\it confer} \cite[Proposition~$4.1$]{Beu}}}$)$ \label{Beu j}
Let $j$ be a negative integer. Let $f(z)\in K[[z]]$ be an arithmetic Gevrey series of order $j$ and $\xi\in \overline{\Q}\setminus\{0\}$ such that $f(\xi)=0$. Then 
$f(z)/(z-\xi)$ is again an arithmetic Gevrey series of order $j$. 
\end{proposition} 
\begin{proof}
Assuming $f$ has rational coefficients, by applying \cite[Theorem~$3.4.1$]{An3} and using the same arguments as in the proof of \cite[Corollary~$2.2$]{Beu}, we can ensure Proposition~$\ref{Beu j}$. For the general case, Proposition~$\ref{Beu j}$ is proven by the same arguments as in \cite[Proposition|$4.1$]{Beu}.
\end{proof} 
{{For a power series $f(z)\in \overline{\Q}[[z]]$ and an embedding $\sigma:\overline{\Q}\hookrightarrow \C$, we denote the image of $f$ of the natural extension of $\sigma$ to $\overline{\Q}[[z]]$ by $f^{\sigma}$.}}
\begin{proposition} $(${\rm{{\it confer} \cite[Proposition~$1$]{F-R}}}$)$\label{F-R j}
Let $j$ be a negative integer. Let $f\in K[[z]]$ be an arithmetic Gevrey series of order $j$ and $\xi\in \overline{\Q}\setminus\{0\}$. Then the following assertions are equivalent$:$

\medskip

$({\rm{i}})$ $f$ vanishes at $\xi$.

\medskip

$({\rm{ii}})$ There exists $\sigma\in {\rm{Gal}}(\overline{\Q}/\Q)$ such that $f^{\sigma}$ vanishes at $\sigma(\xi)$.

\medskip

$({\rm{iii}})$ For any $\sigma\in {\rm{Gal}}(\overline{\Q}/\Q)$ such that $f^{\sigma}$ vanishes at $\sigma(\xi)$.

\medskip

$({\rm{iv}})$ There exists an arithmetic Gevrey series $g$ of order $j$ with coefficients in $K$ such that $$f(z)=D(z)g(z) \ \ \text{where} \ D \ \text{is the minimal polynomial of} \ \xi \ \text{over} \ K\enspace.$$
\end{proposition}
\begin{proof}
This proposition is proven using the same arguments as in \cite[Proposition 1]{F-R}, with Proposition~$\ref{Beu j}$ used in place of \cite[Proposition 4.1]{Beu}.
\end{proof}

\begin{corollary} \label{conjugate 0}
Let $(b_0,b_{i,s})_{1\le i \le m, 1\le s \le q}\in K^{dm+1}\setminus\{\boldsymbol{0}\}$. Assume 
\[
b_0+\sum_{i,s}b_{i,s}F_s(\boldsymbol{\gamma},\alpha_i)=0\enspace.
\]
Then for any $\sigma\in {\rm{Gal}}(\overline{\Q}/\Q)$, we have 
$$\sigma(b_0)+\sum_{i,s}\sigma(b_{i,s})F^{\sigma}_s(\boldsymbol{\gamma}, \sigma(\alpha_i))=0\enspace.$$
\end{corollary}
\begin{proof}
Notice $F_s(\boldsymbol{\gamma},\alpha_i z)\in K[[z]]$ are arithmetic Gevrey series of exact order $d''-d'<0$ for any $1\le i \le m, 1\le s \le q$.
This yields the power series $$f(z):=b_0+\sum_{i,s}b_{i,s}F_s(\boldsymbol{\gamma},\alpha_i z)$$ is also an arithmetic Gevrey series of order $d''-d'<0$.
Using Proposition~$\ref{F-R j}$ for the above $f$ and $\xi=1$, we obtain the assertion. 
\end{proof}

\begin{proof} [\textbf{Proof of Theorem~$\ref{general 2}$}]
Let $P_{\ell}(z),P_{\ell,i,s}(z)$ be polynomials defined in \eqref{Pl}, \eqref{Plis} respectively for $F_s(\boldsymbol{\gamma},\alpha_i/z)$.
Set $$a^{(n)}_{\ell}=P_{\ell}(1), \ \ \ a^{(n)}_{\ell,i,s}=P_{\ell,i,s}(1) \ \ \text{for} \ 0\le \ell\le dm, \ 1\le i \le m, \ 1\le s \le d\enspace.$$
By Proposition~$\ref{non zero det}$ and Remark~$\ref{change Delta}$, the matrix 
${{\mathrm{M}}}_n=\begin{pmatrix}
a^{(n)}_{\ell}\\
a^{(n)}_{\ell,i,s}
\end{pmatrix}$ 
is invertible. 
Assume $1,F_s(\boldsymbol{\gamma},\alpha_i)$ are linearly dependent over $K$. 
Then there exists a non-zero vector $\boldsymbol{b}=(b_0,b_{i,s})_{i,s}\in K^{dm+1}$ such that $$b_0+\sum_{i,s}b_{i,s}F_s(\boldsymbol{\gamma},\alpha_i)=0\enspace.$$
Corollary $\ref{conjugate 0}$ ensures, for any $v\in \mathfrak{M}^{\infty}_K$, 
\begin{align} \label{sigmav 0}
\sigma_v(b_0)+\sum_{i,s}\sigma_v(b_{i,s})F^{\sigma_v}_s(\boldsymbol{\gamma},\sigma_v(\alpha_i))=0\enspace.
\end{align}
Since ${\rm{M}}_n$ is non-singular, there exists $0\le \ell_n \le dm$ such that $$B_{\ell_n}:=b_0a^{(n)}_{\ell_n}+\sum_{i,s}b_{i,s}a^{(n)}_{\ell_n,i,s}\in K\setminus\{0\}\enspace.$$
Notice \eqref{sigmav 0} implies 
\begin{align} \label{B infinite}
\sigma_v(B_{\ell_n})=-\sum_{i,s}\sigma_v(b_{i,s})(\sigma_v(a^{(n)}_{\ell_n})F^{\sigma_v}_s(\boldsymbol{\gamma},\sigma_v(\alpha_i))-\sigma_v(a^{(n)}_{\ell_n,i,s})) \ \ \text{for} \ \ v\in \mathfrak{M}^{\infty}_K\enspace.
\end{align}
Lemma~\ref{upper jyouyonew} $({\rm{ii}})$ ensures
\begin{align*}
\max_{\substack{0\le \ell \le dm \\ 1\le i \le m, 1\le s \le d}}\{|\sigma_v(a^{(n)}_{\ell})F^{\sigma_v}_s(\boldsymbol{\gamma},\sigma_v(\alpha_i))-\sigma_v(a^{(n)}_{\ell,i,s})|_v\} &\le e^{{C_vn}}(n!)^{\tfrac{(d''-d)[K_v:\R]}{[K:\Q]}}
\end{align*}
for $v\in \mathfrak{M}^{\infty}_K$.  
Thus, combining above inequality and Equation~\eqref{B infinite} leads us 
\begin{align} \label{infinite}
|B_{\ell_n}|_v\le e^{{C_vn}}(n!)^{\tfrac{(d''-d)[K_v:\R]}{[K:\Q]}} \ \ \text{for} \ \ v\in \mathfrak{M}^{\infty}_K \enspace. 
\end{align}
\begin{lemma} \label{E type finite}
There is a constant $C>0$ depending only on $\boldsymbol{\zeta},\boldsymbol{\eta}$ such that
$$\frac{1}{n}\sum_{v\in\mathfrak{M}_K^f}\log A_{n,v}(\boldsymbol{\zeta},\boldsymbol{\eta}) \leq C\enspace,$$
where $A_{n,v}(\boldsymbol{\zeta},\boldsymbol{\eta})$ is defined in Lemma~\ref{normepl} .
\end{lemma}
\begin{proof} Let $0\leq k\leq dmn+dm$ and $0\leq l\leq k-1$ then recall (Lemma~\ref{normepl})
\begin{align*}
D_{k,l,n}(\boldsymbol{\zeta},\boldsymbol{\eta})&= \dfrac{\prod_{j=1}^{d'}(\zeta_j+l+1)_{k-l+n-1}}{(n-1)!^{d'}\prod_{j=1}^{d''}(\eta_j+l)_{k-l}}\\
&=\prod_{j=1}^{d'}\frac{(\zeta_j+l+1)_{n-1}}{(n-1)!}\prod_{j=1}^{d''}\frac{(\zeta_j+l+n)_{k-l}}{(\eta_j+l)_{k-l}}\prod_{j=d''+1}^{d'}(\zeta_j+l+n)_{k-l}
\enspace.
\end{align*}
Then, Lemma~\ref{denominator} $({\rm{i}})$ ensures
$$\sum_{v\in\mathfrak{M}_K^f}\log\left\vert \prod_{j=1}^{d'}\frac{(\zeta_j+l+1)_{n-1}}{(n-1)!}\right\vert_v\leq n\sum_{j=1}^{d'}\log \mu(\zeta_j) \enspace,$$
similarly, Lemma~\ref{denominator} $({\rm{iii}})$ ensures
$$\sum_{v\in\mathfrak{M}_K^f}\log\left\vert \prod_{j=1}^{d''}\frac{(\zeta_j+l+n)_{k-l}}{(\eta_j+l)_{k-l}}\right\vert_v\leq dmn\left(\sum_{j=1}^{d''}\left(\log \mu(\zeta_j)+\den(\eta_j)\right)\right)+o(n)$$ and trivially,
$$\sum_{v\in\mathfrak{M}_K^f}\log\left\vert\prod_{j=d''+1}^{d'}(\zeta_j+l+n)_{k-l}\right\vert_v\leq  dmn\sum_{j=d''+1}^{d'}\den(\zeta_j)\enspace.$$
This completes the proof of Lemma~\ref{E type finite}.
\end{proof}
From this lemma, one deduces
\begin{align} \label{finite}
\sum_{v\in\mathfrak{M}_K^f}\log\left\vert B_{\ell_n}\right\vert_v\leq  Cn\enspace.
\end{align}
The product formula for $B_{\ell_n}\in K\setminus \{0\}$ together with Equations \eqref{infinite}, \eqref{finite} implies
$$1=\prod_{v\in \mathfrak{M}_K} |B_{\ell_n}|_v \le e^{O(n)}(n!)^{d''-d}\to 0 \ \ \ (n\to \infty)\enspace.$$
This gives a contradiction. Since we may take any algebraic number field $K$ containing $\alpha_i$, we get the assertion.
\end{proof}

\begin{proof} [\textbf{Proof of Theorem~$\ref{hypergeometric2}$}.]
We use the same notation as in Theorem~$\ref{hypergeometric2}$. In Theorem~$\ref{general 2}$, we take $d''=p$, $d=q+1$, $\boldsymbol{\eta}=(a_1,\ldots,a_p)$, $\boldsymbol{\zeta}=(b_1-1,\ldots,b_{q}-1,0)$ and $\boldsymbol{\gamma}=(1,b_q,\ldots,b_2)$. 
Then, Theorem~$\ref{general 2}$ and Equation~\eqref{Fs} allow us to obtain the assertion of Theorem~$\ref{hypergeometric2}$.
\end{proof}

\subsection{Proof of Theorem~$\ref{hypergeometric3}$}
We now consider the case ${\rm{deg}}\,B=d'<{\rm{deg}}\,A=d$.
Let $\mathcal{O}_K$ denote the ring of integers of $K$.
For each $v\in \mathfrak{M}^{f}_K$, by Lemma~\ref{upper jyouyonew} $({\rm{iii}})$, for any algebraic integer $\alpha \in K$ satisfying 
\begin{align} \label{converge cond}
\left|\alpha\right|_v<\prod_{j=1}^{d}\mu_v(\eta_j)|p|^{\tfrac{d-d'}{p-1}}_v\enspace,
\end{align}
where $p$ is the rational prime below $v$, the series $F_{s,v}(\boldsymbol{\gamma}, \alpha)$ converges to an element of $K_v$.

To establish Theorem~$\ref{hypergeometric3}$, we first prove a slightly more general result:

\begin{theorem} \label{general 3}
Let $\boldsymbol{\alpha} = (\alpha_1, \ldots, \alpha_m) \in (\mathcal{O}_K \setminus \{0\})^m$ have distinct coordinates, each $\alpha_i$ satisfying \eqref{converge cond}.  
Let $\boldsymbol{\gamma} = (\gamma_1, \ldots, \gamma_{d-1}) \in K^{d-1}$ and $\boldsymbol{\lambda} = (\lambda_0, \lambda_{i,s})_{1 \le i \le m, \, 1 \le s \le d} \in \mathcal{O}_K^{dm+1} \setminus \{\boldsymbol{0}\}$.  
Then there exists an effectively computable positive real number $H_0$ such that, whenever $H(\boldsymbol{\lambda}) \ge H_0$, for any $H \ge H(\boldsymbol{\lambda})$, there exists a prime 
\[
p \in \left] \left(\frac{3dm \log H}{(d-d') \log \log H}\right)^{\tfrac{1}{8dm}}, \, \frac{12dm \, \max_{1 \leq j \leq d'} \{\den(\zeta_j)\} \, \log H}{(d-d') \log \log H} \right[
\]
and a place $v \in \mathfrak{M}^{f}_K$ lying above $p$ such that the following linear form in hypergeometric values is nonzero in $K_v$:
\[
\lambda_0 + \sum_{i=1}^m \sum_{s=1}^d \lambda_{i,s} F_{s,v}(\boldsymbol{\gamma}, \alpha_i) \neq 0 \,.
\]
\end{theorem}
Before proving Theorem~$\ref{general 3}$, we show how it implies Theorem~$\ref{hypergeometric3}$.
\begin{proof} [\textbf{Proof of Theorem~$\ref{hypergeometric3}$}.]
We keep the notation from Theorem~$\ref{hypergeometric3}$. 
In Theorem~$\ref{general 3}$, take $d'=q+1$, $d=p$, $\boldsymbol{\eta}=(a_1,\ldots,a_p)$, $\boldsymbol{\zeta}=(b_1-1,\ldots,b_{q}-1,0)$, and $\boldsymbol{\gamma}=(a_{p-1},\ldots,a_{1})$.
Then, Theorem~$\ref{general 3}$ and Equation~\eqref{Fs p>} allow us to obtain the assertion of Theorem~$\ref{hypergeometric3}$.
\end{proof}

We now proceed to the proof of Theorem~$\ref{general 3}$. We begin with some preparatory lemmas.
\begin{lemma}{\rm{\cite{RS}}} \label{log p/p}
Let $x>1$ be a real number. Then
$$ \sum_{\substack{p:\text{prime} \\ p\le x}}\dfrac{\log p}{p}<\log x\enspace.$$
\end{lemma}

\begin{lemma}{\rm{\cite[Corollary 1, Theorem 9]{RS}}} \label{pi theta}
Let $x>1$ be a real number. Define
$$ \pi(x) = \sum_{\substack{p:\text{prime} \\ p\le x}}1, \quad \vartheta(x) = \sum_{\substack{p:\text{prime} \\ p\le x}}\log p\enspace.$$
Then$:$

$({\rm{i}})$ $\pi(x) < \dfrac{1.25506\, x}{\log x}$.

$({\rm{ii}})$ $\vartheta(x) < 1.01624\, x$.
\end{lemma}

\subsubsection{Proof of Theorem~$\ref{general 3}$} \label{product formula}
Let $P_{\ell}(z), P_{\ell,i,s}(z)$ be the polynomials defined in \eqref{Pl}, \eqref{Plis}, respectively, for $F_s(\boldsymbol{\gamma},\alpha_i/z)$.

Define
\begin{align*}
a^{(n)}_{\ell} = P_{\ell}(1), \quad a^{(n)}_{\ell,i,s} =  P_{\ell,i,s}(1)\enspace.
\end{align*}
Moreover, Proposition~$\ref{non zero det}$ and Remark~$\ref{change Delta}$ assert that the matrix
\[
\mathrm{M}_n = 
\begin{pmatrix}
a^{(n)}_{\ell} \\
a^{(n)}_{\ell,i,s}
\end{pmatrix}
\]
is invertible.
Let $\boldsymbol{\lambda} = (\lambda_0, \lambda_{i,s})_{1 \le i \le m, 1 \le s \le d} \in \mathcal{O}_K^{dm+1} \setminus \{ \boldsymbol{0} \}$. 
Due to the invertibility of $\mathrm{M}_n$, there exists $0 \le \nu_n \le dm$ such that
\[
B_{\nu_n} := a^{(n)}_{\nu_n} \lambda_0 + \sum_{i=1}^m \sum_{s=1}^d a^{(n)}_{\nu_n,i,s} \lambda_{i,s} \neq 0\enspace.
\]
{We denote by $\boldsymbol{a}_{\nu_n}$ the vector $\left(a^{(n)}_{\nu_n}, a^{(n)}_{\nu_n,i,s} \right)$.
Let us now estimate $|B_{\nu_n}|_v$ for $v \in \mathfrak{M}_K$.

{We separate cases with $\mathfrak{M}_K=\mathfrak{M}_{K}^{\infty} \sqcup S_1\sqcup S_2\sqcup S_3$, where 
\begin{align*}
S_1&=\{v\in \mathfrak{M}_K^f; v\mid p, \, p\leq (dmn)^{1/8dm}\}\medskip, \\
S_2&=\{v\in \mathfrak{M}_K^f; v\mid p, \, (dmn)^{1/8dm}<p <4\max_{1\leq j\leq d'}\{\den(\zeta_j)\} dmn\}\enspace,\\
S_3&=\{v\in \mathfrak{M}_K^f; v\mid p, \, p\geq 4\max_{1\leq j\leq d'}\{\den(\zeta_j)\}dmn\}\enspace.
\end{align*}
We start with a lemma to apply Lemma~\ref{normepl}. 
\begin{lemma}\label{valnormepl}
One has, provided $n\geq 2$,
\begin{align*}
\sum_{v\in S_1\sqcup S_3}\log A_{n,v}(\boldsymbol{\zeta},\boldsymbol{\eta}) &\leq 3dmn\left(\sum_{j=1}^{d'}\log\mu(\zeta_j)+\sum_{j=1}^{d}\den(\eta_j)\right)\\
                                                                                                                                &+(d-d')\sum_{p\leq (dmn)^{\frac{dm}{8}}}\log\left\vert (dm(n+1))!\right\vert_v^{-1}\enspace,
\end{align*}
where $A_{n,v}(\boldsymbol{\zeta},\boldsymbol{\eta})$ is defined in Lemma~\ref{normepl}. 
\end{lemma}
\begin{proof}
Let $v\in \mathfrak{M}^{f}_K$. Recall 
\[
A_{n,v}(\boldsymbol{\zeta},\boldsymbol{\eta})=\max\{\left \vert D_{k,l,n}(\boldsymbol{\zeta},\boldsymbol{\eta})\right\vert_v;\kern5pt 0\leq {l\leq k-1, 0\leq} k \leq dmn+dm\}
\]
with
\begin{align*}
D_{k,l,n}(\boldsymbol{\zeta},\boldsymbol{\eta})&= {\dfrac{\prod_{j=1}^{d'}(\zeta_j+l+1)_{k-l+n-1}}{(n-1)!^{d'}\prod_{j=1}^{d}(\eta_j+l)_{k-l}}}\\
                                                                &=\prod_{j=1}^{d'}\frac{(\zeta_j+l+1)_{n-1}}{(n-1)!}\prod_{j=1}^{d'}\frac{(\zeta_j+l+n)_{k-l}}{(\eta_j+l)_{k-l}}\\
                                                                &\cdot \prod_{j=d'+1}^{d}\frac{(k-l)!}{(\eta_j+l)_{k-l}}\cdot\frac{1}{(k-l)!^{d-d'}}\enspace.
\end{align*}
By Lemma~\ref{denominator} (i),
$$\sum_{v\in \mathfrak{M}_K^f}\log^+\left\vert \prod_{j=1}^{d'}\frac{(\zeta_j+l+1)_{n-1}}{(n-1)!}\right\vert_v\leq  \sum_{j=1}^{d'}\log\mu_n(\zeta_j)\leq n\sum_{j=1}^{d'}\log \mu(\zeta_j)\enspace.$$
By Lemma~\ref{denominator}, (iii) combined with Lemma~\ref{log p/p} (i) we also have, provided $n\geq 2$
$$\sum_{v\in \mathfrak{M}_K^f}\log^+\left\vert\prod_{j=1}^{d'}\frac{(\zeta_j+l+n)_{k-l}}{(\eta_j+l)_{k-l}}\right\vert_v \leq 2dmn\left(\sum_{j=1}^{d'}\log\mu(\zeta_j)+\den(\eta_j)\right)
$$
and using the same argument, for $n\geq 2$
$$\sum_{v\in \mathfrak{M}_K^f}\log^+\left\vert\prod_{j=d'+1}^{d}\frac{(k-l)!}{(\eta_j+l)_{k-l}}\right\vert_v\leq 2dmn\left(\sum_{j=d'+1}^{d}\den(\eta_j)\right)
\enspace.$$
Summing up the three terms,
one gets (since $\vert (k-l)!\vert_v=1$ for $v\in S_3$)
\begin{align*}
\sum_{v\in S_1\sqcup S_3}\log A_{n,v}(\boldsymbol{\zeta},\boldsymbol{\eta}) &\leq 3dmn\left(\sum_{j=1}^{d'}\log\mu(\zeta_j)+\sum_{j=1}^{d}\den(\eta_j)\right)\\
                                                                                                                                &+(d-d')\sum_{p\leq (dmn)^{\frac{dm}{8}}}\log\left\vert (dm(n+1))!\right\vert_v^{-1}\enspace.
\end{align*}
This completes the proof of Lemma~\ref{valnormepl}.
\end{proof}
\begin{proof} [\textbf{Proof of Theorem~\ref{general 3}}]
We now assume by contradiction 
\begin{align} \label{v in S_2 = 0}
\lambda_0+\sum_{i=1}^m\sum_{s=1}^d\lambda_{i,s}F_{s,v}(\boldsymbol{\gamma}, \alpha_i)= 0 \ \text{for all} \ v\in S_2\enspace.
\end{align}
We are going to bound trivially $\vert B_{\nu_n}\vert_v$ by $(dm+1)^{\varepsilon_v [K_v:\Q_p]/[K:\Q]} \Vert \boldsymbol{a}_{\nu_n}\Vert_v\Vert \boldsymbol{\lambda}\Vert_v$ for 
$v\in \mathfrak{M}_K^{\infty}\sqcup S_1\sqcup S_3$ and use Lemma~\ref{upper jyouyonew} $(iv)$ for primes lying in $S_2$.

By Lemma~\ref{normepl} (i) together with Equation~\eqref{estimhomo}, one has
$$
\sum_{v\in\mathfrak{M}_K^{\infty}}\log^+\Vert \boldsymbol{a}_{\nu_n}\Vert_v\leq \sum_{v\in \mathfrak{M}_K^{\infty}}\log c_{n,v}+Bn+dmn{\rm{h}}_{\infty}(\boldsymbol{\alpha})\enspace,
$$
and thus
\begin{align} \label{v in infty}
\sum_{v\in\mathfrak{M}_K^{\infty}}\log|B_{\nu_n}|_v\le \sum_{v\in \mathfrak{M}_K^{\infty}}\log c_{n,v}+Bn+dmn{\rm{h}}_{\infty}(\boldsymbol{\alpha})+{\rm{h}}_{\infty}(\boldsymbol{\lambda})\enspace.
\end{align}
By Lemma~\ref{normepl} (ii) combined with Lemma~\ref{valnormepl}, taking into account that $\alpha_i\in \mathcal{O}_K$, one has
\begin{align}
\sum_{v\in S_1\sqcup S_3}\log^+\Vert \boldsymbol{a}_{\nu_n}\Vert_v&\leq \sum_{v\in  S_1\sqcup S_3}\log c_{n,v}+3dmn\left(\sum_{j=1}^{d'}\log\mu(\zeta_j)+\sum_{j=1}^{d}\den(\eta_j)\right) \nonumber\\
                                                                                        &+(d-d')\sum_{p\leq (dmn)^{1/8dm}}\log\vert(dm(n+1))!\vert^{-1}_p\enspace. \label{v in S1 S3}
\end{align}
Using $v_p(k!)\leq k/(p-1)$ and Lemma~\ref{log p/p}, one gets assuming $n\geq dm$
\begin{align*}
\sum_{p\leq (dmn)^{1/8dm}}\log\vert(dm(n+1))!\vert^{-1}_p&\leq 2dm(n+1)\sum_{p\leq (dmn)^{1/8dm}}\frac{\log p}{p}\\
                                                                                 &\leq \frac{2dm(n+1)}{8dm}\log(dmn)\leq n\log n\enspace.
\end{align*}
Combining Equation~\eqref{v in S1 S3} and above inequality gives
{\small{\begin{align}\label{v in S1 S3 II}
&\sum_{v\in S_1\sqcup S_3}\log|B_{\nu_n}|_v\le \\
      &\sum_{v\in  S_1\sqcup S_3}(\log c_{n,v}+{\rm{h}}_v(\boldsymbol{\lambda}))+3dmn\left(\sum_{j=1}^{d'}\log\mu(\zeta_j)+\sum_{j=1}^{d}\den(\eta_j)\right)+(d-d')n\log n\enspace. \nonumber
\end{align}}}

\medskip
We now turn to primes in $S_2$. Notice, by~\eqref{v in S_2 = 0}, a straightforward computation gives
\begin{align} \label{B 3}
B_{\nu_n} = - \sum_{i=1}^m \sum_{s=1}^d \lambda_{i,s} r^{(n)}_{i,s} \in K_v \enspace,
\end{align}
where $$r^{(n)}_{i,s} = a^{(n)}_{\nu_n} F_{s,v}(\boldsymbol{\gamma}, \alpha_i) - a^{(n)}_{\nu_n,i,s}\enspace.$$

\medskip

We assume further $n\geq \exp(d)$ to ensure $p\ge e^{\tfrac{d}{d-d'}}$ if $p$ is a prime below $v\in S_2$ 
and deduce from Equation~\eqref{B 3} and Lemma~\ref{upper jyouyonew} ({\rm{iv}}) that
\begin{align*}
|B_{\nu_n}\vert_v&\le c_{n,v} \delta_v(n)p^{d'\tfrac{[K_v:\Q_p]}{[K:\Q]}}\prod_{j=1}^{d}|\mu_n(\eta_{j})|^{-1}_v\cdot |n!|^{d-d'}_v\cdot {\rm{H}}_v(\boldsymbol{\lambda})\enspace,
\end{align*}
where {{$$\delta_v(n)= \prod_{j=1}^{d'} \left(|{\rm{den}}(\zeta_j)\zeta_j|+{\rm{den}}(\zeta_j)((dm+1)n+dm)\right)^{\tfrac{[K_v:\Q_p]}{[K:\Q]}}\enspace.$$}}
One deduces
{\small{\begin{align} 
\sum_{v\in S_2}\log|B_{\nu_n}\vert_v & \leq  \sum_{v\in S_2}\log c_{n,v}+d'\sum_{v\in S_2, p|v} \log p+n\sum_{j=1}^{d}\log\mu(\eta_j) \label{B S_2}\\  
 &+\displaystyle\log\prod_{j=1}^{d'} \left(|{\rm{den}}(\zeta_j)\zeta_j|+{\rm{den}}(\zeta_j)((dm+1)n+dm)\right)\sum_{v\in S_2, p|v}1 \nonumber\\
 &+ \displaystyle (d-d')\sum_{v\in S_2, p|v}\log \vert n!\vert_p\enspace. \nonumber
\end{align}
}}
By Lemma~\ref{pi theta}  
\begin{align}
&\sum_{v\in S_2,  p|v}\log p\leq  8\max_{1\leq j\leq d}\{\den(\zeta_j)\} dmn\enspace, \label{S_2 log p}\\
&\sum_{v\in S_2,  p|v}1\leq 8\dfrac{\max_{1\leq j\leq d}\{\den(\zeta_j)\}dmn}{\log n}\enspace, \label{S_2 1}\\
&\sum_{p\in S_1}\log\vert n!\vert_p\geq  -n\sum_{p\leq (dmn)^{1/(8dm)}}\log(p)/p\enspace, \nonumber
\end{align}
and by formula  3.24 of  \cite{RS}
$$-n\sum_{p\leq (dmn)^{1/(8dm)}}\log p/p\geq -\frac{n}{8dm}\log(dmn)=-\frac{n\log n}{8dm}+O(n)\enspace.$$
It follows since $n\leq 4\max_{1\leq j\leq d'}\{\den(\zeta_j)\}dmn$
\begin{align*}
\sum_{p\in S_2}\log\vert n!\vert_p&=\sum_{p\in S_1\cup S_2}\log\vert n!\vert_p-\sum_{p\in S_1}\log\vert n!\vert_p\\
                                                         &=\log n!-\sum_{p\in S_1}\log\vert n!\vert_p\geq -n\log n+\frac{n\log n}{8dm}+O(n)
\end{align*}
and summing up
$$\sum_{v\in S_2, p|v} \log\vert n!\vert_p\leq -\dfrac{7}{8}n\log n+O(n)\enspace.$$
Putting inequalities~\eqref{S_2 log p},~\eqref{S_2 1} and the above one in \eqref{B S_2}, we conclude
\begin{align} \label{v in S2}
\sum_{v\in S_2}\log|B_{\nu_n}\vert_v\leq -\dfrac{7}{8}(d-d')n\log n+ C_1n\enspace,
\end{align} where $C_1$ is some positive constant.
Using Equations~\eqref{v in infty}, \eqref{v in S1 S3 II} and \eqref{v in S2}, one gets
$$\sum_{v\in\mathfrak{M}_K}\log\vert B_{\nu_n}\vert_v\leq -\dfrac{3}{8}(d-d')n\log n+C_2n+\log H\enspace.$$
We now choose
$$ H\geq H_0 \ \mbox{large enough}, \mbox{and} \ n=\frac{3\log H}{(d-d')\log\log H}\enspace.$$
The above inequality cannot hold and thus the hypothesis that all the linear form vanishes for all primes in $S_2$ is false.
\end{proof}
\section{Corrigendum to Linear Forms in Polylogarithms}
{The proof of Lemma~$4.8$, as written in \cite{DHK3} was incorrect, we provide for a rectified version of Lemma~$4.8$, with conclusion unchanged. Since Lemma~$4.8$ 
of\, {\it loc. cit.} was used as Lemma~$4.10$ in our subsequent paper \cite{DHK4}, one should apply the rectified version stated below instead  and no other change in this paper is needed.}
 
\bigskip

We keep all notations and conventions of \cite{DHK3} and state the new version of Lemma~$4.8$ \cite{DHK3}. 
The last line of the proof of Lemma~$4.8$ is incorrect since  inequality   $l-\vert\boldsymbol{I}\vert\geq 0$ is not enough  to conclude that 
$$\vert \boldsymbol{I}\vert+r\left(l-\vert\boldsymbol{I}\vert\right)\geq2 r^2n+r^2 \implies \vert \boldsymbol{I}\vert\geq (2n+1)r^2\enspace.$$ 

\begin{lemma} \label{condicombi}
Let $0\leq l$ be an integer and $\boldsymbol{I}=(a_1,\ldots,a_r)\in \N^r$ such that $\vert \boldsymbol{I}\vert \leq l$.
Assume further: 

\begin{itemize}\item[$({\rm{i}})$] The $2r$ dimensional vector $(\boldsymbol{k},\boldsymbol{k}-\boldsymbol{I})$ has two coordinates in common.
\end{itemize}
Then, 
$$\Delta_{\alpha}\circ \psi_{\beta,\boldsymbol{k}}\circ\psi_{\alpha,\boldsymbol{k}-\boldsymbol{I}}\left(g\frac{\partial^jf}{\partial\alpha^j}\right)=0 \kern20pt \mbox{for all $0\leq j\leq l-\vert\mathbf{I}\vert$}\enspace. $$
Moreover,  assume
\begin{itemize}
\item[(ii)]  $l<(2n+1)r^2$.  Then, for every  $\vert\boldsymbol{I}\vert\leq l$, and every $j$, $0\leq j\leq l-\vert\boldsymbol{I}\vert$, one has
\end{itemize}
$$\Delta_{\alpha}\circ \psi_{\beta,\boldsymbol{k}}\circ\psi_{\alpha,\boldsymbol{k}-\boldsymbol{I}}\left(g\frac{\partial^jf}{\partial\alpha^j}\right)=0\enspace.$$
\end{lemma}
\begin{proof}   The first part of the statement is a subset of the first part of the original version of Lemma~\ref{condicombi} unaffected by the error and needs not to be proven.

\bigskip

In view of condition (i), we can assume $a_s-s\geq 0$ for all $s$. Let $\boldsymbol{I}$ such that $\vert\boldsymbol{I}\vert\leq l$ and let $0\leq j\leq l-\vert\boldsymbol{I}\vert $. For two $r$ tuples of integers $\boldsymbol{a},\boldsymbol{b}$ we say that $\boldsymbol{a}\leq \boldsymbol{b}$ if for every $1\leq s\leq r$, $a_s\leq b_s$ (partial order). By Leibnitz formula, setting 
$\deri^{\boldsymbol{m}}=\bigcirc_{s=1}^r\deri_{X_s,x}^{m_s}$ (recall $\deri_{X_s,x}=\left(\frac{\partial}{\partial X_s}+x/X_s\right)\circ[X_s]$)
\begin{align*}
\psi&_{\alpha,\boldsymbol{k}-\boldsymbol{I}}\left(g\frac{\partial^j}{\partial \alpha^j}(f)\right)=\\
&\alpha^r\bigcirc_{s=1}^r\ev_{X_s\rightarrow \alpha}\left(\sum_{\boldsymbol{m}\leq \boldsymbol{I}-\boldsymbol{k}}\sum_{\boldsymbol{n}\leq \boldsymbol{I}-\boldsymbol{k}-\boldsymbol{m}}c(\boldsymbol{m},\boldsymbol{n})\deri^{\boldsymbol{m}}(g)\deri^{\boldsymbol{n}}\left(\frac{\partial^j}{\partial\alpha^j}(f)\right)\right)
\end{align*} 
where $c(\boldsymbol{m},\boldsymbol{n})$ is some combinatorial factor. By definition of $g$, if $\vert \boldsymbol{m}\vert < \frac{r(r-1)}{2}$, one has $$\bigcirc_{s=1}^r\ev_{X_s\rightarrow \alpha}\left(\deri^{\boldsymbol{m}}(g)\right)=0\enspace,$$
since $\prod_{1\leq i<j\leq r}(X_i-X_j)$ divides $g$.
 
There is thus no restriction to assume $\vert \boldsymbol{m}\vert\geq \frac{r(r-1)}{2}$. Now, since $\ev_{\beta\rightarrow \alpha}$ commutes with $\psi_{\beta,\boldsymbol{k}}, \psi_{\alpha,\boldsymbol{k}-\boldsymbol{I}}$, we  have 
$$\ev_{\beta\rightarrow\alpha}\circ \psi_{\beta,\boldsymbol{k}}\circ\psi_{\alpha,\boldsymbol{k}-\boldsymbol{I}}\left(g\frac{\partial^j}{\partial\alpha^j}(f)\right)=\psi_{\beta,\boldsymbol{k}}\circ\psi_{\alpha,\boldsymbol{k}-\boldsymbol{I}}\left(g\ev_{\beta\rightarrow\alpha}\left(\frac{\partial^j}{\partial\alpha^j}(f)\right)\right)\enspace.$$

We now consider 
$$\bigcirc_{s=1}^r\ev_{X_s\rightarrow\alpha}\circ\deri^{\boldsymbol{n}}\left(\ev_{\beta\rightarrow\alpha}\left(\frac{\partial^j}{\partial\alpha^j}(f)\right)\right)\enspace.$$
Since $\left[(X_s-\alpha)(X_s-\beta)\right]^{rn}\mid f$ for all $1\leq s\leq r$, this quantity vanishes as soon as $j+\vert\boldsymbol{n}\vert< 2nr^2$.  Hence, there is no restriction to assume $$j+\vert\boldsymbol{n}\vert\geq 2nr^2\enspace.$$

We can now conclude
$$2nr^2\leq j+\vert \boldsymbol{n}\vert\leq j+\vert\boldsymbol{I}\vert-\vert\boldsymbol{k}\vert-\vert\boldsymbol{m}\vert\leq l-\vert\boldsymbol{I}\vert+\vert\boldsymbol{I}\vert-\frac{r(r+1)}{2}-\frac{r(r-1)}{2}=l-r^2\enspace.
$$
The lemma follows.
\end{proof}

\subsection*{Acknowledgements}
The authors are sincerely grateful to the anonymous referee for detailed comments.

\subsection* {Funding}
This work is partially supported by the Research Institute for Mathematical Sciences, an international joint usage and research center located at Kyoto University, Institute for Mathematical Informatics of Meiji Gakuin University. 
The first author has been supported by IRL2000 Relax of CNRS while 
the second author was supported by JSPS KAKENHI Grant Number 21K03171 and the third author by JSPS KAKENHI Grant Number 24K16905.



\begin{scriptsize}
\begin{minipage}[t]{0.38\textwidth}

Sinnou David,
\\sinnou.david@imj-prg.fr
\\Institut de Math\'ematiques
\\de Jussieu-Paris Rive Gauche
\\CNRS UMR 7586,
Sorbonne Universit\'{e}
\\
4, place Jussieu, 75005 Paris, France,\\
/~CNRS UMI 2000 Relax\\
 Chennai Mathematical Institute\\
 H1, SIPCOT IT Park, Siruseri\\
 Kelambakkam 603103, India \\\\
\end{minipage}
\begin{minipage}[t]{0.32\textwidth} 
Noriko Hirata-Kohno, 
\\hiratakohno.noriko@nihon-u.ac.jp
\\Department of Mathematics
\\College of Science \& Technology
\\Nihon University
\\Kanda, Chiyoda, Tokyo 
\\101-8308, Japan\\\\
\end{minipage}
\begin{minipage}[t]{0.3\textwidth}

Makoto Kawashima,
\\{{kawashima\_makoto@mi.meijigakuin.ac.jp}}
\\{Institute for Mathematical Informatics, Meiji Gakuin University,}
\\{1518 Kamikurata-chyo Totsuka-ku Yokohama-shi Kanagawa,}
\\224-8539, Japan\\\\
\end{minipage}

\end{scriptsize}

\end{document}